\documentclass[preprint,12pt]{elsarticle}

\usepackage{amssymb}

\usepackage[utf8]{inputenc}

\usepackage[margin=2.5cm]{geometry}
\usepackage{graphicx}
\usepackage{amsmath}
\usepackage{amssymb}
\usepackage{amsfonts}
\usepackage{amsthm}
\usepackage{capt-of}
\usepackage{url}
\usepackage[all]{xy}
\usepackage{makeidx}
\usepackage{latexsym}
\usepackage{empheq}
\usepackage{verbatim}
\usepackage{stmaryrd}
\usepackage{enumitem}
\usepackage{bussproofs}
\usepackage{boxedminipage}
\usepackage{hyperref}
\usepackage{ccaption}
\usepackage{cancel}
\usepackage{multicol}
\usepackage{xcolor}
\usepackage{proof}
\usepackage{longtable}
\usepackage{marginnote}
\newcommand{\bba}{\mathbb{A}}
\newcommand{\bbas}{\mathbb{A}^\delta}
\newcommand{\nomj}{\mathbf{j}}
\newcommand{\nomk}{\mathbf{k}}
\newcommand{\nomi}{\mathbf{i}}
\newcommand{\nomh}{\mathbf{h}}
\newcommand{\cnomm}{\mathbf{m}}
\newcommand{\cnomn}{\mathbf{n}}
\newcommand{\cnomo}{\mathbf{o}}
\newcommand{\jty}{J^{\infty}}
\newcommand{\mty}{M^{\infty}}
\def\aol{\rule[0.5865ex]{1.38ex}{0.1ex}}
\def\pdla{\mbox{\rotatebox[origin=c]{180}{$\,>\mkern-8mu\raisebox{-0.065ex}{\aol}\,$}}}
\newcommand{\adnote}[1]{\textcolor{blue}{ADD:#1}}

\newcommand{\mpnote}[1]{\textcolor{red}{MP:#1}}

\newcommand{\afnote}[1]{\textcolor{orange}{AF:#1}}

\newcommand{\pcon}{\mathcal{C}}
\newcommand{\ppcon}{\mathcal{D}}
\newcommand{\npcon}{\cancel{\mathcal{C}}}
\newcommand{\nppcon}{\cancel{\mathcal{D}}}
\newcommand{\wt}{{\rhd}}
\newcommand{\bt}{{\blacktriangleright}}
\newcommand{\tw}{{\lhd}}
\newcommand{\tb}{{\blacktriangleleft}}
\newcommand{\lfilt}[1]{\mathrm{Fi}_{\mathcal{L}}\mathbb{#1}}

\newtheorem{theorem}{Theorem}[section]
\newtheorem{lemma}[theorem]{Lemma}

\newtheorem{definition}[theorem]{Definition}
\newtheorem{corollary}[theorem]{Corollary}
\newtheorem{proposition}[theorem]{Proposition}
\newtheorem{remark}[theorem]{Remark}

\journal{Annals of Pure and Applied Logic}

\begin{document}

\begin{frontmatter}



\title{Obligations and permissions, algebraically}

\author[inst1]{Andrea De Domenico}
\author[inst2]{Ali Farjami}
\author[inst1]{Krishna Manoorkar}
\author[inst1,inst3]{Alessandra Palmigiano}
\author[inst1]{Mattia Panettiere}
\author[inst1,inst4]{Xiaolong Wang*}

\cortext[cor1]{Corresponding author: Xiaolong Wang, Email: 390292381@qq.com }

\affiliation[inst1]{organization={School of Business and Economics, Vrije Universiteit Amsterdam},
            addressline={
            The Netherlands}}

\affiliation[inst2]{organization={School of Computer Engineering, Iran University of Science and Technology},
            addressline={Tehran},
            country={Iran}}

\affiliation[inst3]{organization={Department of Mathematics and Applied Mathematics, University of Johannesburg},
            addressline={South Africa}}

\affiliation[inst4]{organization={School of Philosophy and Social Development, Shandong University},
            addressline={Jinan},
            country={China}
            }

\begin{abstract}
We further develop the algebraic approach to input/output logic initiated  in \cite{wollic22}, where subordination algebras and a family of their generalizations were proposed as a semantic environment of various input/output logics.
In particular, we consider precontact algebras as a suitable algebraic environment for negative  permission, and we characterize properties of several types of permission (negative, static, dynamic), as well as their interactions with normative systems, by means of suitable modal languages encoding outputs.
\end{abstract}

\begin{keyword}
input/output logic\sep subordination algebras\sep precontact algebras\sep selfextensional logics\sep slanted algebras\sep algorithmic correspondence theory.

\MSC 03G25 \sep 03G10
\end{keyword}

\end{frontmatter}


\section{Introduction}
The present paper continues a line of investigation, recently initiated in \cite{wollic22}, which pursues
 the study of {\em input/output logic} from the viewpoint of algebraic logic \cite{font2003survey}, by establishing systematic connections between {\em input/output logic} and {\em subordination algebras}, two areas of research in non-classical logics which have been pursued  independently of one another by different research communities, with different motivations and  methods.

The framework of {\em input/output logic} \cite{Makinson00} is designed for modelling the interaction between logical inferences and other agency-related notions  such as  conditional obligations, goals, ideals, preferences, actions, and beliefs, in the context of the formalization of normative systems in philosophical logic and AI. Recently, the original framework of input/output logic, based on classical propositional logic, has been generalized  to incorporate various forms of
{\em nonclassical} reasoning \cite{parent2014intuitionistic,stolpe2015concept},
and these generalizations have   contextually motivated the introduction of algebraic and proof-theoretic methods in the study of input/output logic   \cite{sun2018proof,ciabattoni2023streamlining}.

{\em Subordination algebras} \cite{bezhanishvili2017irreducible} are tuples $(A, \prec)$ consisting of a Boolean algebra $A$  and  a binary relation $\prec$  on $A$  such that the direct (resp.~inverse) image of each element $a\in A$ is a filter (resp.~an ideal) of $A$. These structures have been introduced for connecting and systematising several notions (including pre-contact algebras \cite{dimov2005topological} and quasi-modal algebras \cite{celani2001quasi,celani2016precontact}, of which subordination algebras are equivalent presentations, and proximity lattices \cite{jung1996duality}) cropping up in the context of generalizations of Stone duality, within a research program which brings together constructive and point-free mathematics, non-classical logics, and their applications to the (denotational) semantics of programming languages. For instance, in the context of proximity lattices, the proximity relation $\prec$ relates (the algebraic interpretation of) two logical propositions $\varphi$ and $\psi$ if, whenever $\varphi$ is
{\em observed},  $\psi$ is {\em actually} true.
In \cite{wollic22}, subordination algebras and generalizations thereof are proposed as an  algebraic semantic framework for several input/output logics. Moreover, via the recently established  link between subordination algebras and {\em slanted algebras} \cite{de2020subordination}, well-known properties of normative systems and their associated output operators  have been equivalently characterized in terms of modal axioms, using  results from the general theory of {\em unified correspondence} \cite{conradie2014unified,conradie2019algorithmic,conradie2020constructive}.
In the present paper, we systematically study the algebraic counterparts of normative, permission, and dual permission systems based on the algebras of the classes canonically associated with (fully) selfextensional logics \cite{jansana2006referential, Jansana2006conjunction};  normative, permission, and dual permission systems on selfextensional logics have been introduced and studied in the companion paper \cite{dedomenico2024obligations}.  In particular, in the present paper, we refine and streamline  the proof of the main characterization result \cite[Proposition 3.7]{wollic22}, and  further   extend it to characterize properties of various notions of permission systems  also in relation to normative systems, by means of modal axioms in an expanded signature which includes negative modal operators (intuitively representing prohibitions). Moreover, we apply these characterization results to resolve different issues: we obtain logical characterizations of output operators for both normative  and permission systems;
we obtain dual characterizations of conditions on subordination algebras (resp.~ precontact algebras) and on their associated   spaces; we algebraize positive static permissions, and modally characterize the notion of cross-coherence; we clarify the relation between certain conditions on positive bi-subordination lattices and Dunn's axioms for positive modal logic.

\paragraph{Structure of the paper} In Section \ref{sec:prelim}, we collect basic definitions and facts about the abstract logical framework of selfextensional logics in which we are going to develop our results, and recall the definition of normative systems, introduced in \cite{wollic22}, pertaining to this framework. 
We also introduce, discuss and study the properties of negative permission systems in the context of selfextensional logics. In Section \ref{sec: algebraic models}, we partly recall, partly introduce, semantic structures generalizing subordination algebras and precontact algebras, which we take as the algebraic semantic environment of input/output logics based on selfextensional logics and their associated permission systems, and study the basic properties of these algebraic models. In Section \ref{sec: proto and slanted}, we partly recall, partly define, how these algebraic structures can be associated with certain classes of {\em slanted algebras} \cite{de2021slanted} (cf.~Definition \ref{def:slanted}), which form an algebraic environment in which {\em output operators} from input/output logic can be systematically represented as modal operators endowed with order-theoretic, algebraic, and topological properties. In Section \ref{sec: modal charact}, we
 characterize conditions on (the algebraic counterparts of) normative systems and permission systems  in terms of the validity of modal inequalities on their associated slanted algebras. In Section \ref{sec: applications}, we use the results of the previous section to characterize the output operators of both normative  and permission systems in terms of properties of their associated modal operators
 (cf.~Propositions \ref{prop:output}, \ref{prop:output contact}, and \ref{prop:output dual contact}), to define the algebraic counterparts of the  static positive permission systems and modally characterize the notion of cross coherence (cf.~Section \ref{ssec:algebraizing static permission}), to extend Celani's dual characterization results for subordination lattices, and obtain both these results and similar results as consequences of standard modal correspondence (cf.~Propositions \ref{prop: celani}, \ref{prop:more dual charact subordination}, and \ref{prop: precontact space}), and to clarify the relation between certain conditions on positive bi-subordination lattices and Dunn's axioms for positive modal logic (cf.~Section \ref{ssec:dunn-saxioms}). We conclude in Section \ref{sec:conclusions}. In the appendix, we adapt some of  Jansana's results on selfextensional logics with conjunction \cite{Jansana2006conjunction} to the setting of selfextensional logics with disjunction.

\section{Preliminaries}
\label{sec:prelim}
The present section 
collects preliminaries on  
selfextensional logics (cf.~Section \ref{ssec:selfextensional}),   on normative 
and permission systems  based on these (cf.~Section \ref{ssec:negperm}), on subordination algebras and related structures (cf.~Section \ref{ssec:subordination related}), on canonical extensions and slanted algebras (cf.~Section \ref{ssec:canext slanted}).
\subsection{Selfextensional logics}\label{ssec:selfextensional}
In what follows, we align to the literature in abstract algebraic logic \cite{font2003survey}, and  understand a {\em logic} to be a tuple
$\mathcal{L}= (\mathrm{Fm}, \vdash)$, such that $\mathrm{Fm}$ is the term algebra (in a given algebraic signature which, abusing notation, we will also denote by $\mathcal{L}$) over a set $\mathsf{Prop}$ of atomic propositions, and $\vdash$ is a {\em consequence relation} on $\mathrm{Fm}$, i.e.~$\vdash$ is a relation between sets of formulas and formulas  such that, for all $\Gamma, \Delta\subseteq \mathrm{Fm}$ and all $\varphi\in \mathrm{Fm}$, (a) if $\varphi\in \Gamma$ then $\Gamma\vdash \varphi$; (b) if  $\Gamma\vdash \varphi$ and $\Gamma \subseteq \Delta$, then $\Delta\vdash \varphi$; (c) if $\Delta\vdash \varphi$ and $\Gamma\vdash \psi$ for every $\psi\in \Delta$, then $\Gamma\vdash \varphi$; (d) if $\Gamma\vdash \varphi$, then $\sigma[\Gamma]\vdash \sigma (\varphi)$ for every homomorphism (i.e.~variable-substitution) $\sigma: \mathrm{Fm}\to \mathrm{Fm}$.
Clearly,  any such $\vdash$ induces a preorder on $\mathrm{Fm}$, which we still denote $\vdash$, by restricting to singletons.    A logic $\mathcal{L}$ is {\em selfextensional} (cf.~\cite{jansana2006referential}) if the relation ${\equiv} \subseteq\, \mathrm{Fm}\times \mathrm{Fm}$, defined by $\varphi\equiv \psi$ iff $\varphi\vdash \psi$ and $\psi\vdash \varphi$, is a congruence of $\mathrm{Fm}$. In this case, the {\em Lindenbaum-Tarski algebra} of $\mathcal{L}$ is the partially ordered algebra $Fm = (\mathrm{Fm}/{\equiv}, \vdash)$ where, abusing notation, $\vdash$ also denotes the order on $\mathrm{Fm}/{\equiv}$ defined as $[\varphi]_{\equiv}\vdash [\psi]_{\equiv}$ iff $\varphi\vdash \psi$.
For any algebra $\mathbb{A}$ of the same signature of $\mathcal{L}$, a subset $F$ of the domain $A$ of $\mathbb{A}$ is an {\em $\mathcal{L}$-filter} of $\mathbb{A}$ if for every $\Gamma\cup\{\varphi\}\subseteq \mathrm{Fm}$ and every homomorphism $h: \mathrm{Fm}\to \mathbb{A}$, if $\Gamma\vdash \varphi$ and $h[\Gamma]\subseteq F$, then  $h(\varphi)\in F$.  Let $\lfilt{A}$ denote the set of $\mathcal{L}$-filters of $\mathbb{A}$. It is easy to see that $\lfilt{A}$ is a  closure system over $\mathbb{A}$, where a {\em closure system} over $\mathbb{A}$ is a collection of subsets of $\mathbb{A}$ which is closed under arbitrary intersection. A {\em g-model} of $\mathcal{L}$ is a tuple $(\mathbb{A}, \mathcal{G})$ s.t.~$\mathbb{A}$ is an $\mathcal{L}$-algebra and $\mathcal{G}\subseteq \lfilt{A}$ is a closure system over $\mathbb{A}$.
Any such $\mathcal{G}$ induces a consequence relation $\models_{\mathcal{G}}$ on $\mathrm{Fm}$ defined by $\Gamma \models_{\mathcal{G}}\varphi$ iff for all $h\in Hom( \mathrm{Fm}, \mathbb{A})$ and all $F\in \mathcal{G}$, if $h[\Gamma]\subseteq F$ then $h(\varphi)\in F$. It follows straightforwardly from the definitions that $\vdash\ =\ \models_{\lfilt{A}}\ \subseteq \ \models_{\mathcal{G}}$. A g-model is {\em full} whenever $\vdash\ = \ \models_{\mathcal{G}}$. Any g-model $(\mathbb{A}, \mathcal{G})$ induces  an equivalence relation $\equiv_{\mathcal{G}}$ on $\mathbb{A}$, referred to as the {\em Frege relation of} $(\mathbb{A}, \mathcal{G})$, defined as follows: $a\equiv_{\mathcal{G}} b$ whenever $a\in F $ iff $ b\in F$ for any $F\in \mathcal{G}$. Clearly, $\mathcal{G}\subseteq \lfilt{A}$ implies that $\equiv_{\lfilt{A}}\ \subseteq\ \equiv_{\mathcal{G}}$. Then the logic $\mathcal{L}$ is {\em fully selfextensional} if $\equiv_{\mathcal{G}}$ is a congruence of (the algebra of) any full g-model $(\mathbb{A}, \mathcal{G})$ of $\mathcal{L}$. The class $\mathsf{Alg}(\mathcal{L})$ canonically associated with a fully selfextensional logic $\mathcal{L}$ is  the class of those $\mathcal{L}$-algebras $\mathbb{A}$ such that some full g-model $(\mathbb{A}, \mathcal{G})$  exists s.t.~$\equiv_{\mathcal{G}}$ is the identity $\Delta_{\mathbb{A}}$. Thus, $\mathsf{Alg}(\mathcal{L}) = \{\mathbb{A}\mid\  \equiv_{\lfilt{A}}\ =\ \Delta_{\mathbb{A}}\}$, and so, on any $\mathbb{A}\in \mathsf{Alg}(\mathcal{L})$, the specialization preorder induced by $\lfilt{A}$ is a partial order.

For any selfextensional logic $\mathcal{L}$, the {\em intrinsic variety} of $\mathcal{L}$ is the variety $\mathsf{K}_{\mathcal{L}}$ generated by the Lindenbaum-Tarski algebra $Fm$. Hence, $\mathsf{K}_{\mathcal{L}}\models \varphi =\psi$ iff $\varphi\dashv\vdash \psi$ for all $\varphi, \psi\in \mathrm{Fm}$. It is well known that $\mathsf{Alg}(\mathcal{L})\subseteq \mathsf{K}_{\mathcal{L}}$ (cf.~\cite[Lemma 3.10]{Jansana2006conjunction}).

\smallskip
The environment described above takes the notion of consequence relation as primary in defining a logical system, and in particular, it abstracts away from any concrete logical signature. However, the familiar logical connectives such as conjunction and disjunction can be reintroduced in terms of of their behaviour w.r.t.~the consequence relation of the given logic. Indeed,
for any $\Gamma\subseteq \mathrm{Fm}$, let $Cn(\Gamma): = \{ \psi \mid \Gamma \vdash \psi \}$.\footnote{In what follows, we write e.g.~$Cn(\varphi)$ for $Cn(\{\varphi\})$, and $Cn(\Gamma, \varphi)$ for $Cn(\Gamma\cup\{\varphi\})$.}
The {\em conjunction property} $(\wedge_P)$ holds for $\mathcal{L}$ if a term $t(x, y)$ (which we denote $x\wedge y$) exists in the language of $\mathcal{L}$   such that $Cn(\varphi\wedge \psi) = Cn(\{\varphi, \psi\})$ for all $\varphi, \psi\in \mathrm{Fm}$. The {\em disjunction property} $(\vee_P)$ holds for $\mathcal{L}$ if a term $t(x, y)$ (which we denote $x\vee y$) exists in the language of $\mathcal{L}$   such that $Cn(\varphi\vee \psi) = Cn(\varphi)\cap Cn( \psi)$ for all $\varphi, \psi\in \mathrm{Fm}$.
The {\em strong disjunction property} $(\vee_S)$ holds for $\mathcal{L}$ if a term $t(x, y)$ (which we denote $x\vee y$) exists in the language of $\mathcal{L}$   such that $Cn(\Gamma, \varphi\vee \psi) = Cn(\Gamma,\varphi)\cap Cn(\Gamma, \psi)$ for all $\varphi, \psi\in \mathrm{Fm}$ and every $\Gamma\subseteq \mathrm{Fm}$.

The {\em weak negation property} ($\neg_W$) holds for $\mathcal{L}$ if a term $t(x)$ (which we denote $\neg x$) exists in the language of $\mathcal{L}$   such that $\psi \in Cn(\varphi)$ implies $\neg \varphi \in Cn(\neg \psi)$ for any $\varphi, \psi \in \mathrm{Fm}$. For any logic $\mathcal{L}$ with $\neg_W$,
\begin{enumerate}
\item The {\em right-involutive negation property} ($\neg_{Ir}$) holds for $\mathcal{L}$ if   $Cn(\neg\neg \varphi) \subseteq Cn( \varphi)$ for any $\varphi\in \mathrm{Fm}$.
  \item    The {\em left-involutive negation property} ($\neg_{Il}$) holds for $\mathcal{L}$ if   $Cn( \varphi) \subseteq Cn(\neg\neg \varphi)$ for any $\varphi\in \mathrm{Fm}$.
  \item   The {\em involutive negation property} ($\neg_I$) holds for $\mathcal{L}$ if both $\neg_{Il}$ and $\neg_{Ir}$ hold for $\mathcal{L}$.
  \item   The {\em absurd negation property} ($\neg_A$) holds for $\mathcal{L}$ if    $Cn( \varphi, \neg \varphi) = \mathrm{Fm}$ for any $\varphi\in \mathrm{Fm}$.
  \item   The {\em pseudo negation property} ($\neg_P$) holds for $\mathcal{L}$ if  $\wedge_P$ holds for $\mathcal{L}$, and moreover,   $\neg \psi \in Cn(\varphi, \neg(\varphi \wedge \psi))$ for any $\varphi, \psi \in \mathrm{Fm}$.
   \item  The {\em excluded middle property} ($\sim_A$) holds for $\mathcal{L}$ if    $Cn( \varphi) \cap Cn(\neg \varphi) = Cn(\varnothing)$ for any $\varphi\in \mathrm{Fm}$.

  \item   The {\em pseudo co-negation property} ($\sim_P$) holds for $\mathcal{L}$ if  $\vee_P$ holds for $\mathcal{L}$, and moreover,   $\varphi \vee \neg (\varphi \vee \psi) \in Cn(\neg \psi)$ for all $\varphi, \psi \in \mathrm{Fm}$.
  \item   The {\em strong negation property} ($\neg_S$) holds for $\mathcal{L}$ if $Cn(\varphi, \psi) = \mathrm{Fm}$ implies $\neg \psi \in Cn(\varphi)$.
  \end{enumerate}
     The (weak)\footnote{We refer to this property as weak, because in the literature the property referred to as {\em deduction-detachment property} is $\psi \in Cn(\Gamma, \varphi)$ iff $\varphi \rightarrow \psi \in Cn(\Gamma)$ for all $\varphi, \psi \in \mathrm{Fm}$ and $\Gamma \subseteq \mathrm{Fm}$.} {\em deduction-detachment property} ($\rightarrow_P$) holds for $\mathcal{L}$ if a term $t(x,y)$ (which we denote $x\rightarrow y$) exists in the language of $\mathcal{L}$ such that $\psi \in Cn(\chi,\varphi)$ iff $\varphi \rightarrow \psi \in Cn(\chi)$ for all $\varphi, \psi \in \mathrm{Fm}$.

     The {\em co-implication property} ($\pdla_P$) holds for $\mathcal{L}$ if a term $t(x,y)$ (which we denote $x \pdla y$, to be read as ``$x$ excludes $y$'') exists in the language of $\mathcal{L}$ such that $\chi \in Cn(\varphi \pdla \psi)$ iff $Cn(\chi) \cap Cn(\psi) \subseteq Cn(\varphi)$ for all $\varphi,\psi, \chi \in \mathrm{Fm}$.

     \smallskip
     Every selfextensional logic with $\wedge_P$ is
fully selfextensional \cite[Theorems 4.31 and 4.46]{font2017general}. Moreover, in \cite{Jansana2006conjunction}, selfextensional logics with $\wedge_P$ are characterized as those  
such that the equations that define a semilattice hold for each element $A$ in the class $\mathsf{Alg}(\mathcal{L})$ of algebras canonically associated with $\mathcal{L}$, 
and the following condition is satisfied for all $\varphi_1,\ldots ,\varphi_n, \varphi\in \mathrm{Fm}$:
\[\varphi_1,...,\varphi_{n} \vdash \varphi\quad \text{iff}\quad h(\varphi_1)\wedge\cdots \wedge h(\varphi_n)\leq h(\varphi) \, \text{ for all }A\in\mathsf{Alg}(\mathcal{L}) \text{ and any }  h\in   Hom(\mathrm{Fm}, A).\]

Hence, for every selfextensional logic with $\wedge_P$
and all $\varphi,\psi\in\mathrm{Fm}$, $\varphi\vdash \psi$ iff $h(\varphi)\leq h(\psi)$ for every $A\in \mathsf{Alg}(\mathcal{L})$ and every homomorphism $h: \mathrm{Fm}\to A$. Similar properties hold for super-compact selfextensional logics with $\vee_P$ (cf.~Section \ref{sec: selfextensional w disj}).
Therefore, for any selfextensional logic $\mathcal{L}$, property  $\wedge_P$ (resp.~$\vee_P$ if $\mathcal{L}$ is also super-compact)
guarantees that the relation of logical entailment of $\mathcal{L}$ is completely captured by the order of the algebras in $\mathsf{Alg}(\mathcal{L})$, and that each such algebra is a meet (resp.~join)
semilattice w.r.t.~the (possibly defined) operation interpreting the term $\wedge$ (resp.~$\vee$).
This is the main reason why the theory developed from Section \ref{sec: algebraic models} on will take ordered algebras (and distinguished subclasses thereof, such as semilattices, lattices and distributive lattices) as its basic   environment.  In what follows, for the sake of readability, in defining and considering conditions on these algebras corresponding to e.g.~different metalogical properties of selfextensional logics, in relation with different closure properties of input/output logics based on them (see next subsection), we will sometimes omit reference to  the minimal assumptions presupposed by the satisfaction of those conditions.

\begin{lemma}
\label{lemma:antitonicity of neg} (cf.~\cite[Lemma 2.1 and Proposition 2.2]{dedomenico2024obligations})
For any logic $\mathcal{L} = (\mathrm{Fm}, \vdash)$,
\begin{enumerate}
    \item If  properties $\wedge_P$, $\vee_P$, and $\neg_W$ hold for $\mathcal{L}$, then $\neg\varphi\vee \neg \psi\vdash \neg (\varphi\wedge \psi)$ for all $\varphi, \psi\in \mathrm{Fm}$.
    \item If in addition  property $\neg_{Il}$ holds for $\mathcal{L}$, then $ \neg (\varphi\wedge \psi)\vdash \neg\varphi\vee \neg \psi$ for all $\varphi, \psi\in \mathrm{Fm}$.
    \item if $\wedge_P$ and $\vee_S$ hold for $\mathcal{L}$, then  $\alpha \wedge (\beta \vee \gamma) \vdash (\alpha \wedge \beta) \vee (\alpha \wedge \gamma)$ for all $\alpha, \beta, \gamma\in \mathrm{Fm}$.
    \end{enumerate}
\end{lemma}

\subsection{Input/output logics on selfextensional logics}\label{ssec:i-o logic}

Input/output logic  \cite{Makinson00} is a framework modelling the interaction between the relation of logical entailment between states of affair (being represented by formulas) and other binary relations on states of affair, representing e.g.~systems of norms, strategies, preferences, and so on. Although the original framework of input/output logic takes $\mathcal{L}$ to be classical propositional logic, in the present section we collect  basic definitions and facts, introduced in \cite{dedomenico2024obligations}, about input/output logic in the more general framework of selfextensional logics.
\paragraph{Normative systems.}
Let $\mathcal{L}= (\mathrm{Fm}, \vdash)$ be a logic in the sense specified in Section \ref{ssec:selfextensional}. 
A  {\em normative system} on $\mathcal{L}$ is a relation $N \subseteq \mathrm{Fm}\times \mathrm{Fm}$, the elements  $(\alpha,\varphi)$ of which are called {\em conditional norms} (or obligations).

  A normative system  $N \subseteq \mathrm{Fm}\times\mathrm{Fm}$ is {\em internally incoherent} if  $(\alpha, \varphi)$ and $ (\alpha, \psi) \in N$ for some $\alpha, \varphi, \psi\in \mathrm{Fm}$  such that  $Cn(\alpha) \neq \mathrm{Fm}$ and $Cn(\varphi,\psi) = \mathrm{Fm}$; a normative system $N$ is {\em internally coherent} if it is not internally incoherent.
  If  $N, N' \subseteq \mathrm{Fm}\times\mathrm{Fm}$ are  normative systems,  $N$ is {\em almost included} in $N'$ (in symbols: $N\subseteq_c N'$) if $(\alpha,\varphi) \in N$ and $Cn(\alpha)\neq\mathrm{Fm}$ imply $(\alpha,\varphi) \in N'$.

Each  norm $(\alpha,\varphi)\in N$ can be intuitively read as ``given $\alpha$, it {\em should} be the case that $\varphi$''.
This interpretation can be further specified according to the context: for instance, if $N$ formally represents a system of (real-life) rules/norms, then  we can read $(\alpha,\varphi)\in N$ as ``$\varphi$ is obligatory whenever $\alpha$ is the case''; if $N$ formally represents a scientific theory, then  we can read $(\alpha,\varphi)\in N$ as ``under conditions $\alpha$, one should observe $\varphi$'', in the sense that the scientific theory predicts $\varphi$ whenever $\alpha$; finally, if $N$ formally represents (the execution of) a program, then we can read $(\alpha,\varphi)\in N$ as ``in every state of computation in which $\alpha$ holds, the program will move to a state in which $\varphi$ holds''.
For any  $\Gamma\subseteq \mathrm{Fm}$, let $N(\Gamma) := \{ \psi \mid \exists \alpha(\alpha \in \Gamma \ \&\ (\alpha,\psi)\in N )\}$.

    An {\em input/output logic} is a tuple $\mathbb{L} = (\mathcal{L}, N)$ s.t.~$\mathcal{L}= (\mathrm{Fm}, \vdash)$ is a selfextensional logic, and $N$ is a normative system on $\mathcal{L}$.

For any input/output logic $\mathbb{L} = (\mathcal{L}, N)$, and each $1\leq i\leq 4$, the output operation $out_i^N$ is defined as follows: for any $\Gamma\subseteq \mathrm{Fm}$,

\[out_i^N(\Gamma): = N_i(\Gamma) = \{ \psi\in \mathrm{Fm} \mid \exists \alpha(\alpha \in \Gamma \ \&\ (\alpha,\psi)\in N_i )\} \] 
 where $N_i\subseteq \mathrm{Fm}\times \mathrm{Fm}$ is the  {\em closure}  of $N$ under (i.e.~the smallest extension of $N$ satisfying)  the  inference rules below, as  specified in the table.

\vspace{1mm}
\begin{center}
\begin{tabular}{lll}
\infer[(\top)]{(\top,\top)}{} \infer[(\bot)]{(\bot,\bot)}{} &
\infer[\mathrm{(SI)}]{(\beta,\varphi)}{(\alpha,\varphi) &  \beta \vdash \alpha} &
\infer[\mathrm{(WO)}]{(\alpha,\psi)}{(\alpha,\varphi) &  \varphi \vdash \psi} \\[2mm]
\infer[\mathrm{(AND)}]{(\alpha,\varphi \land \psi)}{(\alpha,\varphi) & (\alpha,\psi)} &
\infer[\mathrm{(OR)}]{(\alpha \lor \beta, \varphi)}{(\alpha,\varphi) & (\beta,\varphi)} &
\infer[\mathrm{(CT)}]{(\alpha,\psi)}{(\alpha,\varphi) & (\alpha \land \varphi, \psi)}
\end{tabular}
\end{center}

\begin{center}
\label{table1}
	\begin{tabular}{ l l}
		\hline
		$N_i$ & Rules   \\
		\hline	
		$N_{1}$  & $ \mathrm{(\top), (SI), (WO), (AND)}$   \\
		$N_{2}$  & $\mathrm{(\top), (SI), (WO), (AND), (OR)}$  \\
		$N_{3}$  & $\mathrm{(\top), (SI), (WO), (AND), (CT)}$  \\
		$N_{4}$  & $\mathrm{(\top), (SI), (WO),(AND), (OR), (CT)}$ \\
		\hline

 \end{tabular}
\captionof{table}{closures of normative systems}
\end{center}
Clearly, with the exception of  $\mathrm{(SI)}$ and $\mathrm{(WO)}$,  all the rules above (as well as the rules  below) apply only to those input/output logics based on selfextensional logics with the (minimal) metalogical properties guaranteeing the existence of the corresponding term-connectives. So, for instance, rules $\mathrm{(AND)}$ and $\mathrm{(CT)}$ only apply in the context of logics for which $\wedge_P$ holds, and so on. For the sake of a better readability, in the remainder of the paper we will implicitly assume these basic properties, and only mention the additional properties when it is required.

\paragraph{Negative permission systems.}
\label{ssec:negperm}

The following definition (cf.~\cite[Section 4.1]{dedomenico2024obligations})  generalizes the usual notion of negative permission (cf.~\cite[Section 2]{Makinson03}) while being formulated purely in terms of the consequence relation of the given selfextensional logic, and it informally says that any $\varphi$   is permitted under a given  $\alpha$ iff  $\varphi$  is not logically inconsistent with any obligation $\psi$ under $\alpha$.

    For any input/output logic $\mathbb{L} = (\mathcal{L}, N)$,
    \begin{center}
        $P_N: = \{(\alpha, \varphi)\mid \forall \psi ((\alpha,\psi)\in N  \Rightarrow Cn(\varphi, \psi) \neq \mathrm{Fm}   )\}$.
    \end{center}

As discussed in \cite[Section 4.1]{dedomenico2024obligations},
If $N$ is internally coherent, then $N \subseteq_c P_N$. Moreover,

    For any $N, N' \subseteq \mathrm{Fm}$, if $N \subseteq N'$ then $P_{N'} \subseteq P_{N}$.

The following closure rules  on  $P_N^c: = (\mathrm{Fm}\times \mathrm{Fm})\setminus P_N$:\footnote{That is, the reading of the rule \infer[]{(\alpha,\varphi)}{(\beta,\psi) &  \alpha \vdash \beta} is as follows: if $(\alpha, \psi)\notin P_N$ and $\alpha \vdash \beta$ then $(\beta, \varphi)\notin P_N$.} have been introduced and discussed in \cite[Section 4.1]{dedomenico2024obligations}.

\vspace{1mm}

\begin{center}
\begin{tabular}{lll}
\infer[\mathrm{(\top)}^\wt]{(\top,\bot)}{} & \infer[\mathrm{(\bot)}^\wt]{(\bot,\top)}{} &
\infer[\mathrm{(SI)}^\wt]{(\alpha,\varphi)}{(\beta,\varphi) &  \alpha \vdash \beta} \\[2mm]
\infer[\mathrm{(WO)}^\wt]{(\alpha,\varphi)}{(\alpha,\psi) &  \varphi \vdash \psi} &
\infer[\mathrm{(AND)}^\wt]{ (\alpha, \varphi\vee \psi)}{(\alpha,\varphi) & (\alpha, \psi)} &
\infer[\mathrm{(OR)}^\wt]{ (\alpha\vee \beta, \varphi)}{(\alpha,\varphi) & (\beta, \varphi)}
 \\[2mm]
\infer[\mathrm{(CT)}^\wt]{(\alpha, \psi)}{(\alpha,\varphi) \in N& (\alpha \wedge \varphi, \psi)}
\end{tabular}
\end{center}

For any input/output logic $\mathbb{L} = (\mathcal{L}, N)$ and any $1\leq i\leq 4$, we let $P_i: = P_{N_i}$.
\begin{proposition} (cf.~\cite[Corollary 4.6]{dedomenico2024obligations})
 For any input/output logic $\mathbb{L} = (\mathcal{L}, N)$, if $\wedge_P$, $\vee_S$, $\bot_P$, and $\top_W$ hold for $\mathcal{L}$, then
 $P_i^c$ for $1\leq i\leq 4$ is closed  under the rules indicated in the following table.
\end{proposition}

\begin{center}
	\begin{tabular}{ l l}
		\hline
		$P_i^c$ & Rules   \\
		\hline	
		$P_{1}^c$  & $ \mathrm{(\top)^\wt, (SI)^\wt, (WO)^\wt, (AND)^\wt}$   \\
		$P_{2}^c$  & $\mathrm{ (\top)^\wt, (SI)^\wt, (WO)^\wt, (AND)^\wt, (OR)^\wt}$  \\
		$P_{3}^c$  & $\mathrm{ (\top)^\wt, (SI)^\wt, (WO)^\wt, (AND)^\wt, (CT)^\wt}$  \\
		$P_{4}^c$  & $\mathrm{ (\top)^\wt, (SI)^\wt, (WO)^\wt,(AND)^\wt, (OR)^\wt, (CT)^\wt}$ \\
		\hline
	\end{tabular}
\end{center}

 In Section \ref{ssec: precontact} we will systematically connect the closure rules of  (the relative complements of) permission systems with the environment of (proto-)precontact algebras.

 In \cite{dedomenico2024obligations}, it is discussed how the  perspective afforded by the general setting of selfextensional logics
 makes it possible to consider a notion of {\em dual conditional permission system} which, in the setting of classical propositional logic, is absorbed by the usual notion of negative permission, namely the following:
\[(\alpha, \varphi)\in D_N \text{ iff } (\neg\alpha, \varphi)\notin N \text{ iff } (\neg\alpha, \neg\varphi)\in P_N.\]
However, in the same paper it was discussed that a more general version of $D_N$ can be introduced, similarly to the generalized definition of $P_N$ reported above, namely:
\[D_N: = \{(\alpha, \varphi)\mid \exists\beta((\beta, \varphi)\notin N\ \&\  Cn(\alpha, \beta) = \mathrm{Fm})\},\]
which cannot be subsumed by the definition of $P_N$. While the notion of negative permission $P_N$ intuitively characterizes those states of affair $\alpha$ and $\varphi$ which can both be the case without generating a violation of the  normative system $N$, the dual negative permission system $D_N$ characterizes those states of affair  $\alpha$ and $\varphi$ which can both {\em fail} to be the case without generating a violation of the  normative system $N$.
The following closure rules  on  $D_N^c: = (\mathrm{Fm}\times \mathrm{Fm})\setminus D_N$ have been introduced and discussed in  \cite[Section 4.2]{dedomenico2024obligations}:

\vspace{1mm}

\begin{center}
\begin{tabular}{lll}
\infer[\mathrm{(\top)}^\tw]{(\bot,\top)}{} & \infer[\mathrm{(\bot)}^\tw]{(\top,\bot)}{} &
\infer[\mathrm{(SI)}^\tw]{(\beta,\varphi)}{(\alpha,\varphi) &  \alpha \vdash \beta} \\[2mm]
\infer[\mathrm{(WO)}^\tw]{(\alpha,\psi)}{(\alpha,\varphi) &  \varphi \vdash \psi} &
\infer[\mathrm{(AND)}^\tw]{ (\alpha, \varphi\wedge \psi)}{(\alpha,\varphi) & (\alpha, \psi)} &
\infer[\mathrm{(OR)}^\tw]{ (\alpha\wedge \beta, \varphi)}{(\alpha,\varphi) & (\beta, \varphi)}
 \\[2mm]
$\infer[\mathrm{(CT)}^\tw]{(\alpha, \psi)}{( \alpha, \varphi)& (\varphi \pdla\alpha, \psi)\in N}$
\end{tabular}
\end{center}
For any input/output logic $\mathbb{L} = (\mathcal{L}, N)$ and any $1\leq i\leq 4$, we let $D_i: = D_{N_i}$.

\begin{proposition}
(cf.~\cite[Corollary 4.9]{dedomenico2024obligations})
 For any input/output logic $\mathbb{L} = (\mathcal{L}, N)$, if $\wedge_P$, $\vee_S$, $\top_P$, $\bot_P$, and $\neg_W$ hold for $\mathcal{L}$, then
 $D_i^c$ for $1\leq i\leq 4$ is closed  under the rules indicated in the following table.
\end{proposition}

\begin{center}
	\begin{tabular}{ l l}
		\hline
		$D_i^c$ & Rules   \\
		\hline	
		$D_{1}^c$  & $ \mathrm{(\top)^\tw, (SI)^\tw, (WO)^\tw, (AND)^\tw}$   \\
		$D_{2}^c$  & $\mathrm{ (\top)^\tw, (SI)^\tw, (WO)^\tw, (AND)^\tw, (OR)^\tw}$  \\
		$D_{3}^c$  & $\mathrm{ (\top)^\tw, (SI)^\tw, (WO)^\tw, (AND)^\tw, (CT)^\tw}$  \\
		$D_{4}^c$  & $\mathrm{ (\top)^\tw, (SI)^\tw, (WO)^\tw,(AND)^\tw, (OR)^\tw, (CT)^\tw}$ \\
		\hline
	\end{tabular}
\end{center}

\subsection{Subordination algebras and related structures}
\label{ssec:subordination related}
In the present section, we collect the definitions of a family of inter-related structures originally introduced and studied in the context of  a point-free approach to the region-based theories of discrete spaces. These structures, suitably generalized, serve in the present paper as the main semantic environment for normative and permission systems on selfextensional logics.

\paragraph{Precontact algebras and related structures.} A {\em precontact algebra} \cite{dimov2005topological, duntsch2007region} is a tuple $\mathbb{C} = (\mathbb{A}, \pcon)$ such that $\mathbb{A}$ is a Boolean algebra, and $\pcon$ is a binary relation on the domain of $\mathbb{A}$ such that, for all $a, b, c\in\mathbb{A}$,
 \begin{enumerate}
\item[(C1)] $a \pcon b$ implies $a, b \neq\bot$;
\item[(C2)] $a \pcon (b \vee c)$ iff  $a \pcon b$ or $a \pcon c$;
\item[(C3)] $(a \vee b) \pcon  c$ iff  $a \pcon c$ or $b \pcon c$.
\end{enumerate}
Additional conditions on precontact algebras considered in the literature (cf.~\cite{dimov2005topological, duntsch2007region}) are:
 \begin{enumerate}
\item[(C4)] If $a \neq \bot$ then $a\pcon a$;
\item[(C5)] If $a\pcon b$ then $b\pcon a$;
\item[(C6)] If $a \prec_\pcon c$ then $\exists b (a \prec_\pcon b \prec_\pcon c)$, where $a\prec_\pcon b$ iff $a\npcon \neg b$;
\item[(C7)] If $a \notin\{\bot, \top\}$  then $a\pcon \neg a$ or $\neg a\pcon a$;
\item[(C8)] if $a\wedge b\neq \bot$, then $a\pcon b$.
\end{enumerate}
A {\em contact algebra} is a precontact algebra satisfying (C4) and (C5). A precontact
algebra  is {\em connected} if it satisfies  (C7).

\paragraph{Subordination  algebras and related structures.} A {\em strong proximity lattice} \cite{jung1996duality} is a tuple $\mathbb{P} = (\mathbb{L}, \prec)$ such that $\mathbb{L}$ is a bounded distributive lattice
and $\prec$ is a
 binary relation on the domain of  $\mathbb{L}$ such that, for every $a, x, y \in\mathbb{L} $,
  \begin{enumerate}
  \item[(P0)] $\bot\prec a$ and $a\prec \top$;
\item[(P1)] ${\prec}\circ{\prec} = {\prec}$;
\item[(P2)]  $x\prec a$ and $y\prec a$ iff $x\vee y\prec a$;
\item[(P3)]  $a\prec x$ and $a\prec y$ iff $a\prec x\wedge y$;
\item[(P4)]  $a \prec x \vee y \Rightarrow \exists x'\exists y'( x'\prec x \ \&\ y'\prec y \ \&\ a \prec x' \vee y')$;
\item[(P5)] $x \wedge y \prec a\Rightarrow \exists x'\exists y'( x\prec x' \ \&\ y\prec y' \ \&\ x' \wedge y' \prec a)$.
\end{enumerate}

 A {\em subordination algebra} \cite[Definition 7.2.1]{sumit} is a tuple $\mathbb{S} = (\mathbb{A}, \prec)$ such that $\mathbb{A}$ is a Boolean algebra, and $\prec$ is a binary relation on the domain of $\mathbb{A}$ such that, for all $a, b, c, d\in\mathbb{A}$,
 \begin{enumerate}
\item[(S1)] $\bot\prec\bot$ and $\top\prec\top$;
\item[(S2)] if $a \prec b$ and $a\prec c$ then $a \prec b \wedge c$;
\item[(S3)]  if $a \prec c$ and $b\prec c$ then $a\vee b \prec  c$;
\item[(S4)] if $a \leq b \prec c \leq d$ then $a \prec d$.
\end{enumerate}
Clearly, $ \bot\prec a\prec\top$ for any $ a \in \mathbb{A}$. \noindent A
{\em compingent algebra} \cite{de1962compact} is a subordination algebra such that, for all $a, b\in\mathbb{A}$,
 \begin{enumerate}
\item[(S5)] if $a \prec b$ then $a \leq b$;
\item[(S6)]if $a \prec b$ then $\neg b \prec \neg a$;
\item[(S7)] if $a \prec b$ then $a \prec c \prec b$ for some $c \in \mathbb{A}$;
\item[(S8)] $a \neq \bot$ implies $b\prec a$ for some  $b\neq \bot$.
\end{enumerate}

Subordination algebras are equivalent representations of precontact algebras: for any subordination algebra $\mathbb{S} = (A, \prec)$, the tuple $\mathbb{S}_{\wt} := (A, \pcon_{\prec})$, where $a\pcon_\prec b$ iff $a  \not\prec \neg b$, is a precontact algebra. Conversely,  for any precontact algebra  $\mathbb{C} = (A, \pcon)$, the tuple $\mathbb{C}^\wt := (A, \prec_{\pcon})$, where $a\prec_{\pcon} b$ iff $a  \cancel{\pcon} \neg b$, is a subordination algebra. Finally, $\mathbb{S} = (\mathbb{S}_\wt)^\wt$ and $\mathbb{C} = (\mathbb{C}^\wt)_\wt$.
\paragraph{Dual precontact   algebras and related structures.} 
To our knowledge, the following structures have not independently emerged in the literature:
a {\em dual precontact algebra}  is a tuple $\mathbb{D} = (\mathbb{A}, \ppcon)$ such that $\mathbb{A}$ is a Boolean algebra, and $\ppcon$ is a binary relation on the domain of $\mathbb{A}$ such that, for all $a, b, c\in\mathbb{A}$,
 \begin{enumerate}
\item[(D1)] $a \ppcon b$ implies $a, b \neq\top$;
\item[(D2)] $a \ppcon (b \wedge c)$ iff  $a \ppcon b$ or $a \ppcon c$;
\item[(D3)] $(a \wedge b) \ppcon  c$ iff  $a \ppcon c$ or $b \ppcon c$.
\end{enumerate}

Additional properties:
\begin{enumerate}
\item[(D4)] If $a \neq \top$ then $a\ppcon a$;
\item[(D5)] If $a\ppcon b$ then $b\ppcon a$;
\item[(D6)] If $a \prec_\ppcon c$ then $\exists b (a \prec_\ppcon b \prec_\ppcon c)$; , where $a\prec_\ppcon b$ iff $a\nppcon \neg b$;
\item[(D7)] If $a \notin\{\bot, \top\}$  then $a\ppcon \neg a$ or $\neg a\ppcon a$;
\item[(D8)] if $a\vee b\neq \top$, then $a\ppcon b$.
\end{enumerate}

Dual precontact  algebras are equivalent representations of subordination algebras: for any subordination algebra $\mathbb{S} = (A, \prec)$, the tuple $\mathbb{S}_{\tw} := (A, \ppcon_{\prec})$, where $a\ppcon_\prec b$ iff $\neg a  \not\prec  b$, is a dual precontact algebra. Conversely,  for any dual precontact algebra  $\mathbb{D} = (A, \ppcon)$, the tuple $\mathbb{D}^\tw := (A, \prec_{\ppcon})$, where $a\prec_{\ppcon} b$ iff $\neg a  \cancel{\ppcon}  b$, is a subordination algebra. Finally, $\mathbb{S} = (\mathbb{S}_\tw)^\tw$ and $\mathbb{D} = (\mathbb{D}^\tw)_\tw$.  Hence, precontact algebras and dual precontact algebras are mutually equivalent presentations, via the assignments $\mathbb{D}\mapsto(\mathbb{D}^\tw)_{\wt}$ and $\mathbb{C}\mapsto(\mathbb{C}^\wt)_{\tw}$.

\paragraph{Examples.} As mentioned at the beginning of this section, the structures discussed above arise in the literature as abstract (e.g.~qualitative, point-free)  models of {\em spatial} reasoning; the basic tenet of this generalization is that {\em regions}, rather than points, are taken as the basic spatial notion. Prime examples of subordination, precontact, and dual precontact algebras arise as follows: for any topological space\footnote{\label{footnote:top spaces}A {\em topological space} is a tuple $\mathbb{X} = (X, \tau)$ s.t.~$X$ is a set and $\tau$ is a family of subsets of $X$  which is closed under arbitrary unions (hence $\varnothing  = \bigcup\varnothing\in \tau$) and finite intersections (hence $X = \bigcap\varnothing\in \tau$). Elements of $\tau$ are referred to as {\em open} sets, while complements of open sets are referred to as {\em closed} sets. We let $K(\mathbb{X})$ denote the set of the closed sets of $\mathbb{X}$.
By definition, $K(\mathbb{X})$ is closed under finite unions and arbitrary intersections, hence $\varnothing$ and $X$ are both closed and open sets.  For any $y\subseteq X$, let $cl(y)$ (resp.~$int(y)$) denote the smallest closed set which includes $y$, i.e.~the intersection of the closed sets which include $y$ (resp.~the largest open set  included in $y$, i.e.~the union of the open sets included in $y$). By definition, the assignments $y\mapsto int(y)$ and $y\mapsto cl(y)$ define monotone operations on $\mathcal{P}(X)$, and moreover, $int(y)\subseteq y\subseteq cl(y)$ for any $y\subseteq X$. Notice that $int(y)\cap int(z) = int(y\cap z)$: the right-to-left inclusion immediately follows from the monotonicity of $int(-)$; the converse inclusion follows from $int(y)\cap int(z)$ being an open subset of $y\cap z$. Likewise,  $cl(y)\cup cl(z) = cl(y\cup z)$. Finally, $(cl(y^c))^c = int(y)$: the left-to-right inclusion holds since $(cl(y^c))^c$ is an open set included in $y$; to show that $int(y)\subseteq (cl(y^c))^c$, 
notice that every element of $int(y)$ has a neighbourhood (i.e.~an open set to which the given element belongs) included in $y$ and hence disjoint from $y^c$, while every element in $cl(y^c)$ has no neighbourhood disjoint from $y^c$.
Dually, $cl(y) = 
(int (y^c))^c$, i.e.~$int (y^c) = (cl(y))^c$.} $\mathbb{X} = (X, \tau)$,
\begin{enumerate}
\item $\mathbb{S_X}: = (\mathcal{P}(X), \prec_{\tau})$, where $y\prec_{\tau} z$  iff $cl(y)\subseteq z$, is a subordination algebra\footnote{The relation $\prec_\tau$ is known as the `well inside' or `well below' relation (cf.~\cite{johnstone1982stone}).} satisfying (S5) and (S7). If $\mathbb{X}$ is T$_1$,\footnote{A topological space $\mathbb{X}$ is $\mathrm{T_1}$ if, for any two distinct points, each is contained in an open set not containing the other point.} then  $\mathbb{S_X}$ satisfies (S8) as well;
\item $\mathbb{C_X}: = (\mathcal{P}(X), \pcon_{\tau})$, where $y\pcon_{\tau} z$  iff $cl(y)\cap cl(z)\neq \varnothing$, is a precontact algebra (cf.~\cite{vakarelov2002proximity}) satisfying (C4), (C5), and (C8). If $\mathbb{X}$ is connected\footnote{A topological space $\mathbb{X}$ is \emph{connected} if it is not the union of two or more disjoint non-empty open subsets.}, then $\mathbb{C_X}$ satisfies (C7) as well;
\item  $\mathbb{D_X}: = (\mathcal{P}(X), \ppcon_{\tau})$, where $y\ppcon_{\tau} z$  iff $int(y)\cup int(x)\neq X$, is a dual precontact algebra satisfying (D4), (D5), and (D8). If $\mathbb{X}$ is connected, then $\mathbb{D_X}$ satisfies (D7) as well.
\end{enumerate}
 Item 2 (resp.~item 3) above follows from $cl(-)$ (resp.~$int(-)$) preserving finite, hence empty unions (resp.~intersections), as discussed in Footnote \ref{footnote:top spaces}. Item 1 can be verified straightforwardly (for (S3), one uses that $cl(-)$  preserves finite unions; for (S4), $b\subseteq cl(b)$).

\begin{lemma} For all $y, z\subseteq X$,
\begin{enumerate}
\item $y\prec_{\pcon_\tau} z$ implies $y\prec_{\tau} z$, and $y\prec_{\ppcon_\tau} z$ implies $y\prec_{\tau} z$;
\item $y\pcon_{\prec_\tau} z$ implies $y\pcon_{\tau} z$, and $y\ppcon_{\prec_\tau} z$ implies $y\ppcon_{\tau} z$;
\item $y\pcon_{\tau} z$ iff $y\pcon_{\ppcon_\tau} z$, and $y\ppcon_{\tau} z$ iff $y\ppcon_{\pcon_\tau} z$.
\end{enumerate}
\end{lemma}
\begin{proof}
1. By definition, $y\prec_{\pcon_\tau}z$ iff $cl(y)\cap cl(z^c) = \varnothing$, hence (cf.~Footnote \ref{footnote:top spaces}) $cl(y)\subseteq (cl(z^c))^c\subseteq int(z)\subseteq z$, as required. For the second part, $y\prec_{\ppcon_\tau} z$ iff $int(y^c)\cup int(z) = X$, hence (cf.~Footnote \ref{footnote:top spaces}) $cl(y)\cap (int(z))^c\subseteq (int(y^c))^c\cap (int(z))^c = (int(y^c)\cup int(z))^c = \varnothing$ implies $cl(y)\subseteq int(z)\subseteq z$, i.e.~$y\prec_\tau z$, as required.

2. By definition, $y\pcon_{\prec_\tau}z$ iff $cl(y)\not\subseteq z^c$, i.e.~$cl(y)\cap z\neq \varnothing$ which implies $cl(y)\cap cl(z)\neq\varnothing$ i.e.~$y\pcon_\tau z$, as required.  For the second part, by definition, $y\ppcon_{\prec_\tau}z$ iff $cl(y^c)\not\subseteq z$, i.e.~$cl(y^c)\cap z^c\neq \varnothing$, hence (cf.~Footnote \ref{footnote:top spaces}) $int(y)\cup int(z)\subseteq (cl(y^c))^c\cup z = (cl(y^c)\cap z^c)^c\neq X$ implies that $int(y)\cup int(z)\neq X$, i.e.~$y\ppcon_\tau z$, as required.

3. By definition,
$y\pcon_{\tau} z$ iff $cl(y)\cap cl(z)\neq \varnothing$, iff (cf.~Footnote \ref{footnote:top spaces})  $int(y^c)\cup int(z^c) = (cl(y))^c\cup (cl(z))^c = (cl(y)\cap cl(z))^c\neq X$, 
i.e.~$y\pcon_{\ppcon_\tau} z$, as required.
For the second part, by definition,
$y\ppcon_{\tau} z$ iff $int(y)\cup int(z)\neq X$, iff (cf.~Footnote \ref{footnote:top spaces}) $(cl(y^c)\cap cl(z^c))^c = (cl(y^c))^c \cup (cl(z^c))^c = int(y)\cup int(z)\neq X$, 
i.e.~$y\ppcon_{\pcon_\tau}z$, as required.
\end{proof}
The inclusions of items 1 and 2 of the previous lemma are all proper in general.

Besides serving as abstract models of qualitative spatial reasoning, the structures discussed above also serve as models of {\em information theory}, particularly in the context of the (denotational) semantics of programming languages. For instance,  as observed in \cite{el2006priestley},   the argument was
made in \cite{jung1999multi} that the proximity relation $\prec$ can be interpreted as a stronger relation between  (the algebraic interpretation of) two logical propositions $\varphi$ and $\psi$ than the logical entailment; namely, $\varphi\prec\psi$ if, whenever $\varphi$ is
{\em observed},  $\psi$ is {\em actually} true. However, this interpretation is by no means the only one; for instance,  taking $\prec$ to represent  a {\em scientific theory}, encoded as a set of predictions (i.e.~$a\prec b$ reads as ``$a$ predicts $b$'', in the sense that, whenever $a$ is the case,  $b$  should be empirically observed), then 
$a\pcon_\prec b$ can be interpreted as ``$b$ is {\em compatible} with $a$'', in the sense that the simultaneous observation of $a$ and $b$ does not  lead to the rejection of the `scientific theory' $\prec$.
Similarly, the dual precontact relation $\ppcon_{\prec}$ can be understood as encoding a form of {\em negative compatibility}, in the sense  that  $a\ppcon_{\prec} b$ iff  $a$ and $b$ can be simultaneously {\em refuted}  without this  leading to a violation of $\prec$.

An analogous interpretation allows us to  take (generalizations of) subordination algebras, precontact algebras, and dual precontact algebras as models of normative, permission, and dual permission systems respectively. We will do so in Section \ref{sec: algebraic models}, where we will  also systematically connect these structures with the slanted algebras we discuss in the next subsection.

\subsection{Canonical extensions and slanted algebras}
\label{ssec:canext slanted}
In the present subsection, we adapt material from  \cite[Sections 2.2 and 3.1]{de2020subordination},\cite[Section 2]{DuGePa05}. For any poset $A$, a subset $B\subseteq A$ is {\em upward closed}, or an {\em up-set} (resp.~{\em downward closed}, or a {\em down-set}) if $\lfloor B\rfloor: = \{c\in A\mid \exists b(b\in B\ \&\ b\leq c)\}\subseteq B$ (resp.~$\lceil B\rceil: = \{c\in A\mid \exists b(b\in B\ \&\ c\leq b)\}\subseteq B$); a subset $B\subseteq A$  is {\em down-directed} (resp.~{\em up-directed}) if, for all  $a, b\in B$,  some $x\in B$ exists s.t.~$x\leq a$ and $x\leq b$ (resp.~$a\leq x$ and $b\leq x$). 
It is straightforward to verify that when $A$ is a lattice, down-directed upsets and up-directed down-sets coincide with  lattice filters and ideals, respectively.
\begin{definition}
Let $A$ be a  subposet of a complete lattice $A'$.
\begin{enumerate}
\item An element $k\in A'$ is {\em closed} if $k = \bigwedge F$ for some non-empty and  down-directed  $F\subseteq A$; an element $o\in A'$ is {\em open} if $o = \bigvee I$ for some non-empty and up-directed  $I\subseteq A$;
\item  $A$ is {\em dense} in $A'$ if every element of $A'$ can be expressed both as the join of closed elements and as the meet of open elements of $A$.
\item $A$ is {\em compact} in $A'$ if, for all $F, I\subseteq A$ s.t.~$F$ is non-empty and down-directed, $I$ is non-empty and up-directed, if $\bigwedge F\leq \bigvee I$ then $a\leq b$ for some $a\in  F$ and $b\in I$.\footnote{When the poset $A$ is a lattice, the compactness can be equivalently reformulated by dropping the requirements that $F$ be down-directed and $I$ be up-directed.}
\item The {\em canonical extension} of a poset $A$ is a complete lattice $A^\delta$ containing $A$
as a dense and compact subposet.
\end{enumerate}
\end{definition}
The canonical extension $A^\delta$ of any poset $A$ always exists\footnote{For instance, the canonical extension of a Boolean algebra $A$ is (isomorphic to) the powerset algebra $\mathcal{P}(Ult(A))$, where $Ult(A)$ is the set of the ultrafilters of $A$.} and is  unique up to an isomorphism fixing $A$ (cf.\ \cite[Propositions 2.6 and 2.7]{DuGePa05}).
%
 The set of the closed (resp.~open) elements    of $A^\delta$ is denoted $K(A^\delta)$ (resp.~$O(A^\delta)$). The following proposition collects well known facts which we will use in the remainder of the paper. In particular, items 1(iii) and (iv) are variants of \cite[Lemma 3.2]{gehrke2001bounded}.

 \begin{proposition}
 \label{prop:background can-ext}
 For every poset $A$,
 \begin{enumerate}
 \item for all $k_1, k_2\in K(A^\delta)$  all $o_1, o_2\in O(A^\delta)$ and all $u_1, u_2\in A^\delta$,
 \begin{enumerate}[label=(\roman*)]
     \item $k_1\leq k_2$ iff $k_2\leq b$ implies $k_1\leq b$ for all $b\in A$.
     \item $o_1\leq o_2$ iff $b\leq o_1$ implies $b\leq o_2$ for all $b\in A$.
     \item $u_1\leq u_2$ iff  $k\leq u_1$ implies $k\leq u_2$ for all $k\in K(A^\delta)$, iff $u_2\leq o$ implies $u_1\leq o$ for all $o\in O(A^\delta)$.
     \item If $A$ is a $\vee$-SL, then  $k_1\vee k_2\in K(A^\delta)$.
     \item If $A$ is a $\wedge$-SL, then  $o_1\wedge o_2\in O(A^\delta)$.
  \end{enumerate}
\item if $\neg: A\to A$ is  antitone and s.t.~$(A, \neg)\models \forall a\forall b(\neg a\leq  b \Leftrightarrow \neg b\leq  a)$, then $\neg^\sigma: A^\delta\to A^\delta$ defined as $\neg^\sigma o: =\bigwedge \{\neg a\mid a\leq o\}$ for any $o\in O(A^\delta)$ and $\neg^\sigma u: =\bigvee \{\neg^\sigma o\mid u\leq o\}$ for any $u\in A^\delta$ is antitone and s.t.~$(A^\delta, \neg^\sigma) \models\forall u\forall v(\neg u\leq v \Leftrightarrow \neg v\leq u)$. If in addition, $(A, \neg)\models a\leq \neg\neg a$, then $(A^\delta,\neg^\sigma)\models u\leq \neg\neg u$. Hence, if $(A, \neg)\models a = \neg\neg a$ (i.e.~$\neg$ is involutive), then  $(A^\delta,\neg^\sigma)\models u = \neg\neg u$.
     \item if $\neg: A\to A$ is  antitone and s.t.~$(A, \neg)\models \forall a\forall b( a\leq \neg b \Leftrightarrow  b\leq \neg a)$, then $\neg^\pi: A^\delta\to A^\delta$ defined as $\neg^\pi k: =\bigvee \{\neg a\mid k\leq a\}$ for any $k\in K(A^\delta)$ and $\neg^\pi u: =\bigwedge \{\neg^\pi k\mid k\leq u\}$ for any $u\in A^\delta$ is antitone and s.t.~$(A^\delta, \neg^\pi)\models\forall u\forall v( u\leq \neg v \Leftrightarrow  v\leq\neg u)$. If in addition, $(A, \neg)\models  \neg\neg a\leq a$, then $(A^\delta,\neg^\pi)\models \neg\neg u\leq  u$. Hence, if $\neg$ is involutive, then so is $\neg^\pi$.
 \end{enumerate}
 \end{proposition}
\begin{proof}
    1.
    (i) The left-to-right direction immediately follows from the transitivity of the order. Conversely, if $k_1$ and $k_2\in K(A^\delta)$, then   $k_i=\bigwedge F_i$ for some (nonempty down-directed) $F_i\subseteq A$. Hence in particular, $k_2\leq b$ for all $b\in F_2$. Then the assumption implies that $k_1\leq b$ for all $b\in F_2$, i.e.~$k_1\leq \bigwedge F_2 = k_2$, as required.
    The proofs of (ii) and (iii) are similar, using denseness for the latter item.\\
    (iv) If $k_1$ and $k_2$ are as in the proof of item (i), let us show that $k_1\vee k_2 = \bigwedge F_1\vee \bigwedge F_2 = \bigwedge\{c_1\vee c_2\mid c_i\in F_i\}$. For the left-to-right inequality, let us show that $k_i\leq c_1\vee c_2$ for any $1\leq i, j\leq 2$ and any $c_j\in F_j$. Indeed, $k_i = \bigwedge F_i\leq c_i\leq c_1\vee c_2$. Therefore, $k_1\vee k_2\leq \bigwedge \{c_1\vee c_2\mid c_i\in F_i\}$. For the converse inequality, by item (iii), it is enough to show that if $o\in O(A^\delta)$ and $k_1\vee k_2\leq o$ (i.e.~$k_i\leq o$) then $\bigwedge \{c_1\vee c_2\mid c_i\in F_i\}\leq o$. By compactness, $k_i\leq o$ implies that some $d_i\in F_i$ exists such that $k_i\leq d_i\leq o$. Hence,  $\bigwedge \{c_1\vee c_2\mid c_i\in F_i\}\leq d_1\vee d_2\leq o$, as required. To finish the proof that $k_1\vee k_2\in K(A^\delta)$, it is enough to show that the set $\{c_1\vee c_2\mid c_i\in F_i\}$, which is a subset of $A$ since $A$ is a $\vee$-SL, is also nonempty and down-directed.  Clearly, this set is nonempty since $F_1$ and $F_2$ are; if $c_i, d_i\in F_i$, then $e_i\leq c_i$ and $e_i\leq d_i$ for some $e_i\in F_i$. Hence, $e_1\vee e_2\leq c_1\vee c_2$ and $e_1\vee e_2\leq d_1\vee d_2$, as required. The proof of (v) is order-dual.

    2. For the first part of the statement, see \cite[Proposition 3.6]{DuGePa05}. Let us assume that $(A, \neg)\models a\leq \neg\neg a$, and show that $(A^\delta,\neg^\sigma)\models u\leq \neg\neg u$. The following chain of equivalences holds in $(A^\delta,\neg^\sigma)$, where  $k$ ranges in $K(A^\delta)$ and $o$ in $O(A^\delta)$:
    \begin{center}
        \begin{tabular}{c ll}
             & $\forall u( u\leq \neg\neg u)$ \\
           iff   & $\forall u\forall k\forall o((k\leq u\ \&\ \neg\neg u\leq o) \Rightarrow k\leq o)$ & denseness  \\ iff   & $\forall k\forall o(\exists u(k\leq u\ \&\ \neg\neg u\leq o) \Rightarrow k\leq o)$ &   \\
       iff   & $\forall k\forall o( \neg\neg k\leq o \Rightarrow k\leq o)$ & Ackermann's lemma (cf.~\cite[Lemma 1]{conradie2014unified}) \\    iff   & $\forall k(  k\leq \neg\neg k)$. & denseness \\
        \end{tabular}
    \end{center}
    Hence, to complete the proof, it is enough to show that, if $k\in K(A^\delta)$, then $k\leq \neg\neg k$. By definition, $k = \bigwedge D$ for some down-directed $D\subseteq A$. Since $\neg^\sigma$ is a (contravariant) left adjoint, $\neg^\sigma$ is completely meet-reversing. Hence, $\neg k = \neg(\bigwedge D) = \bigvee \{\neg d\mid d\in D\}$, and since $D$ being down-directed implies that $\{\neg d\mid d\in D\}\subseteq A$ is up-directed, we deduce that $\neg k\in O(A^\delta)$. Hence,
    \begin{center}
        \begin{tabular}{rcll}
              $\neg\neg k$
             & $=$&$ \bigwedge \{\neg a\mid a\leq \neg k\}$ &$\neg k\in O(A^\delta)$ \\
             & $=$&$  \bigwedge \{\neg a\mid a\leq \bigvee \{\neg d\mid d\in D\}\}$\\
             & $=$&$  \bigwedge \{\neg a\mid \exists d(d\in D\ \&\ a\leq \neg d )\}$. & compactness\\ \end{tabular}
    \end{center}
    Hence, to show that $\bigwedge \lfloor D\rfloor = \bigwedge D = k\leq \neg\neg k$, it is enough to show that if $a\in A$ is s.t.~$\exists d(d\in D\ \&\ a\leq \neg d )$, then $d'\leq \neg a$ for some $d'\in D$. From $a\leq \neg d$, by the antitonicity of $\neg$, it follows $\neg\neg d\leq \neg a$; combining this inequality with $d\leq \neg\neg d$ which holds by assumption for all $d\in A$, we get  $d': = d\leq \neg a$, as required. Finally, notice that by instantiating the left-hand inequality in the equivalence $(A^\delta, \neg^\sigma) \models\forall u\forall v(\neg u\leq v \Leftrightarrow \neg v\leq u)$ with $v: = \neg u$, one immediately gets $(A^\delta, \neg^\sigma) \models\forall u(\neg \neg u\leq u)$.  The proof of 3 is dual to that of 2.
\end{proof}

\begin{proposition}
    \label{prop: compactness and existential ackermann}
 For any  poset $A$, for all $a, b\in A$, $k, k_1, k_2 \in K(A^\delta)$, and $o, o_1, o_2 \in O(A^\delta)$, and all $D_i\subseteq A$  nonempty and down-directed,  and $U_i\subseteq A$  nonempty and up-directed for $1\leq i\leq 2$,
\begin{enumerate}
    \item if  $A$ is a $\wedge$-semilattice, then
    \begin{enumerate}[label=(\roman*)]
        \item  $D_1\wedge D_2\coloneqq \{c_1\wedge c_2\ |\ c_1 \in  D_1\ \&\ c_2 \in D_2\}$ is nonempty and down-directed;
        \item if $k_i = \bigwedge D_i$, then $ k_1\wedge k_2 = \bigwedge (D_1\wedge D_2) \in K(A^\delta)$;
        \item $ k_1\wedge k_2 \leq b$ implies $a_1\wedge a_2\leq b$ for some $a_1, a_2 \in A$ s.t.~$ k_i \leq a_i$;
         \item $ k_1\wedge k_2 \leq o$ implies $a_1\wedge a_2\leq b$ for some $a_1, a_2 ,b \in A$ s.t.~$ k_i \leq a_i$ and $b\leq o$.
    \end{enumerate}
    \item if  $A$ is a $\vee$-semilattice, then
    \begin{enumerate}[label=(\roman*)]
        \item  $U_1\vee U_2\coloneqq \{c_1\vee c_2\ |\ c_1 \in  U_1\ \&\ c_2 \in U_2\}$ is nonempty and up-directed;
        \item if $o_i = \bigvee U_i$, then $ o_1\vee o_2 = \bigvee (U_1\vee U_2) \in O(A^\delta)$;
        \item $a\leq o_1\vee o_2 $ implies $a \leq b_1\vee b_2$ for some $b_1, b_2 \in A$ s.t.~$  b_i\leq o_i$;
         \item $ k \leq o_1\vee o_2 $ implies $a \leq b_1\vee b_2$ for some $a, b_1, b_2 \in A$ s.t.~$ b_i \leq o_i$ and $ k \leq a$.
    \end{enumerate}
    \item If $A$ is a $\wedge$-semilattice with top, then $\bigwedge K\in K(A^\delta)$ for every  $K\subseteq K(A^\delta)$.
    \item If $A$ is a $\vee$-semilattice with bottom, then $\bigvee O\in O(A^\delta)$ for every  $O\subseteq O(A^\delta)$.
    \end{enumerate}
\end{proposition}
\begin{proof}
 We only prove items 1 and 3, the proofs of  items 2 and 4 being dual to 1 and 3, respectively.\\
1 (i) By assumption, some $c_i\in D_i$ exists for any $1\leq i\leq 2$. Hence, $c_1\wedge c_2\in D_1\wedge D_2\neq \varnothing$. To show that $D_1\wedge D_2$ is down-directed, let $c_i, d_i\in D_i $ for any $1\leq i\leq 2$ s.t.~$c_1\wedge c_2, d_1\wedge d_2\in D_1\wedge D_2$. Since $D_i$ is down-directed, some $e_i\in D_i$ exists s.t.~$e_i\leq c_i$ and $e_i\leq d_i$ for any $1\leq i\leq 2$. then $e_1\wedge e_2\in D_1\wedge D_2$ and $e_1\wedge e_2\leq c_1\wedge c_2$ and $e_1\wedge e_2\leq d_1\wedge d_2$, as required. 
 (ii) By assumption,  $k_i = \bigwedge D_i\leq d_i$ for all $d_i\in D_i$, hence $k_1 \wedge k_2 \leq d_1 \wedge d_2$, i.e.~$k_1\wedge k_2$ is a lower bound of $D_1\wedge D_2$. To show that $k_1\wedge
 k_2$ is the greatest lower bound, let $c$ be a lower bound  of $D_1\wedge D_2$. Then $c\leq d_1\wedge d_2\leq d_i$ for any $d_i\in D_i$, hence $c $ is a lower bound of $D_i$, and thus $c\leq \bigwedge D_i = k_i$. This shows that $c\leq k_1\wedge k_2$, as required. Finally, together with item (i), the identity just proved implies that $k_1\wedge k_2\in K(A^\delta)$.
\\ (iii) By the previous item, $k_1\wedge k_2\in K(A^\delta)$. Hence, by compactness, $ k_1\wedge k_2 = \bigwedge (D_1\wedge D_2) \leq b$ implies that $a_1\wedge a_2\leq b$ for some  $a_i\in D_i$. Moreover, $k_i = \bigwedge D_i$ implies that $k_i \leq a_i$, as required.
\\(iv) By assumption, $o = \bigvee U $ for some nonempty up-directed subset $U\subseteq A$. Then by compactness, the assumption implies that  $ a_1 \wedge a_2 \leq b$ for some $a_i\in D_i$ and some $b\in U$. Hence $k_i = \bigwedge D_i\leq a_i$ and $b\leq\bigvee U =  o$, as required.

3. If $A$ is a semilattice with top, then $\bigwedge S\in K(A^\delta)$ for every $S\subseteq A$; indeed, $\bigwedge S = \bigwedge \lfloor S\rfloor$, where $\lfloor S\rfloor$ is the meet-semilattice filter generated by $S$, which is nonempty (since $\top\in \lfloor S\rfloor$) and down-directed. Then, to prove the statement, it is enough to show that $\bigwedge K = \bigwedge S$ where $S: = \{a\in A\mid \exists k_a(k_a\in K\ \&\ k_a\leq a)\}$. By definition, $\bigwedge K\leq k_a\leq a$ for every $a \in S$, hence $\bigwedge K \leq \bigwedge S$. Conversely, if $k\in K$ and $a\in A$ s.t.~$k\leq a$, then $a\in S$. This shows that, for every $k\in K$, the inclusion $\{a\in A\mid k\leq a\}\subseteq S$ holds, and hence $\bigwedge S\leq \bigwedge \{a\in A\mid k\leq A\} = k$  is a lower bound of $K$. Thus, $\bigwedge S\leq \bigwedge K$, as required.
\end{proof}

The following definition introduces the algebraic environment for interpreting the output operators associated with normative and permission systems.

\begin{definition}
\label{def:slanted}
Consider the modal signatures $\tau_{\Diamond, \blacksquare}: = \{\Diamond, \blacksquare\}$, $\tau_{\wt, \bt}: = \{\wt, \bt\}$ and $\tau_{\tw, \tb}: = \{\tw, \tb \}$.
\begin{enumerate}
\item A {\em slanted $\tau_{\Diamond, \blacksquare}$-algebra} is a triple $\mathbb{A} = (A, \Diamond, \blacksquare)$ s.t.~$A$ is an ordered algebra, and $\Diamond, \blacksquare: A\to A^\delta$   s.t.~$\Diamond a \in K(A^\delta)$  and $\blacksquare a \in O(A^\delta)$ for every $a\in A$. Such a slanted algebra  is:
\begin{enumerate}
\item {\em monotone} if $\Diamond$ and $\blacksquare$ are {\em monotone}, i.e.~$a\leq b$ implies $\Diamond a\leq \Diamond b$ and $\blacksquare a\leq \blacksquare b$ for all $a, b\in A$;
\item {\em regular}\footnote{\label{footn: tense and normal}Of course, in order for a slanted algebra $\mathbb{A}$ to be regular, the poset $A$ needs to be a lattice, and in order for $\mathbb{A}$ to be normal, $A$ needs to be a bounded lattice. As mentioned earlier on, for the sake of readability, we will sometimes omit to mention assumptions which can be inferred from the context. } if $\Diamond$ and $\blacksquare$ are {\em regular}, i.e.~$\Diamond(a\vee b) = \Diamond a \vee \Diamond b$, and $\blacksquare(a\wedge b) = \blacksquare a \wedge \blacksquare b$ for all $a, b\in A$;
\item {\em normal} if $\Diamond$ and $\blacksquare$ are {\em normal}, i.e.~they are regular and $\Diamond\bot = \bot$  and $\blacksquare\top = \top$;
\item {\em tense} if $\Diamond a\leq b$ iff $a\leq \blacksquare b$ for all $a, b\in A$.
\end{enumerate}
\item A {\em slanted $\tau_{\wt, \bt}$-algebra} is a triple $\mathbb{A} = (A,\wt, \bt)$ s.t.~$A$ is an ordered algebra, and $\wt, \bt: A\to A^\delta$  s.t.~$\wt a, \bt a \in O(A^\delta)$  for every $a\in A$. Such a slanted algebra  is:
\begin{enumerate}
\item {\em antitone} if $\wt$ and $\bt$ are {\em antitone}, i.e.~$a\leq b$ implies $\wt b\leq \wt a$ and $\bt b\leq \bt a$ for all $a, b\in A$;
\item {\em regular} if $\wt$ and $\bt$ are {\em regular}, i.e.~$\wt(a\vee b) = \wt a \wedge \wt b$, and $\bt(a\vee b) = \bt a \wedge \bt b$ for all $a, b\in A$;
\item {\em normal} if $\wt$ and $\bt$ are {\em normal}, i.e.~they are regular and $\wt\bot = \top = \bt\bot$;
\item {\em tense} if $ a\leq \wt b$ iff $b\leq \bt a$ for all $a, b\in A$.
\end{enumerate}
\item A {\em slanted   $\tau_{\tw, \tb}$-algebra} is a triple $(A,\tw, \tb)$ such that $A$ is an ordered algebra, and  $\tw, \tb: A\to A^\delta$    s.t.~ $\tw a, \tb a \in K(A^\delta)$ for every $a\in A$.
Such a slanted algebra  is:
\begin{enumerate}
\item {\em antitone} if $\tw$ and $\tb$ are {\em antitone}, i.e.~$a\leq b$ implies $\tw b\leq \tw a$ and $\tb b\leq \tb a$ for all $a, b\in A$;
\item {\em regular} if $\tw$ and $\tb$ are {\em regular}, i.e.~$\wt(a\wedge b) = \wt a \vee \wt b$, and $\bt(a\wedge b) = \bt a \vee \bt b$ for all $a, b\in A$;
\item {\em normal} if $\tw$ and $\tb$ are {\em normal}, i.e.~they are regular and $\tw \top = \bot = \tb \top$;
\item {\em tense} if $ \tw a\leq  b$ iff $\tb b\leq  a$ for all $a, b\in A$.
\end{enumerate}
\end{enumerate}
\end{definition}
Each operation mapping every $a\in A$ to $K(A^\delta)$ (resp.~to $O(A^\delta)$) is referred to as {\em c-slanted} (resp.~{\em o-slanted}). In what follows, we will typically omit reference to the modal signature of the given slanted algebra, and rely on the context for the disambiguation, whenever necessary. Also, while the slanted algebras of the three modal signatures considered above will be the focus of most of the paper, from Proposition \ref{prop:poly correspondence} on, we will also consider slanted algebras of different modal signatures, which we omit to explicitly define, since their definition will be clear from the context.

The antitone modal operators $\tw$ and $\wt$ mentioned above are the `slanted' versions of the operations in classical logic  defined as $\wt \varphi: = \neg \Diamond \varphi$ and $\tw \varphi: = \neg \Box  \varphi$. Hence, under the alethic (temporal) interpretation of $\Box$ and $\Diamond$, the operators $\wt$ and $\tw$ express the impossibility of and skepticism about $\varphi$ being the case (in the future), respectively, while  $\bt$ and $\tb$ encode analogous meanings but oriented towards ``the past''. In Section \ref{sec: proto and slanted}, we discuss how slanted algebras arise from algebraic models of normative and permission systems on selfextensional logics.

\smallskip
The following definition is framed in the context of monotone (resp.~antitone) slanted algebras, but can be given for arbitrary slanted algebras, albeit at the price of complicating the definition of $\Diamond^\sigma$, $ \blacksquare^\pi$, $\wt^\pi$, $\bt^\pi$, $\tw^\sigma$, and $\tb^\sigma$. Because we are mostly going to apply it in the monotone (resp.~antitone) setting, we present the simplified version here.
\begin{definition}
\label{def: sigma and pi extensions of slanted}
For any monotone slanted algebra $\mathbb{A} = (A, \Diamond, \blacksquare)$, antitone slanted algebra $\mathbb{A} = (A, \wt, \bt)$ and $\mathbb{A} = (A, \tw, \tb)$, the {\em canonical extension} of $\mathbb{A}$ is the (standard!) modal algebra $\mathbb{A}^\delta: = (A^\delta, \Diamond^\sigma, \blacksquare^\pi)$ (resp.~$\mathbb{A}^\delta: = (A^\delta, \wt^\pi, \bt^\pi)$, $\mathbb{A}^\delta: = (A^\delta, \tw^\sigma, \tb^\sigma)$) such that $\Diamond^\sigma, \blacksquare^\pi, \wt^\pi, \bt^\pi, \tw^\sigma, \tb^\sigma: A^\delta\to A^\delta$ are defined as follows:  for every $k\in K(A^\delta)$, $o\in O(A^\delta)$ and $u\in A^\delta$,
\[\Diamond^\sigma k:= \bigwedge\{ \Diamond a\mid a\in A\mbox{ and } k\leq a\}\quad \Diamond^\sigma u:= \bigvee\{ \Diamond^\sigma k\mid k\in K(A^\delta)\mbox{ and } k\leq u\}\]
\[\blacksquare^\pi o:= \bigvee\{ \blacksquare a\mid a\in A\mbox{ and } a\leq o\},\quad \blacksquare^\pi u:= \bigwedge\{ \blacksquare^\pi o\mid o\in O(A^\delta)\mbox{ and } u\leq o\}\]
\[\wt^\pi k:= \bigvee\{ \wt a\mid a\in A\mbox{ and } k\leq a\},\quad \wt^\pi u:= \bigwedge\{ \wt^\pi k\mid k\in K(A^\delta)\mbox{ and } k\leq u\}\]
\[\bt^\pi k:= \bigvee\{ \bt a\mid a\in A\mbox{ and } k\leq a\},\quad \bt^\pi u:= \bigwedge\{ \bt^\pi k\mid k\in K(A^\delta)\mbox{ and } k\leq u\}.\]
\[\tw^\sigma o:= \bigwedge\{ \tw a\mid a\in A\mbox{ and }  a \leq o \},\quad \tw^\sigma u:= \bigvee\{ \tw^\sigma o\mid o\in O(A^\delta)\mbox{ and }  u\leq o \}\]
\[\tb^\sigma o:= \bigwedge\{ \tb a\mid a\in A\mbox{ and } a\leq o\},\quad \tb^\sigma u:= \bigvee\{ \tb^\sigma o\mid o\in O(A^\delta)\mbox{ and } u\leq o\}.\]
\end{definition}
The equivalences in the conclusions of the items in the lemma below say that the maps $\Diamond^\sigma$ and $\blacksquare^\pi$ (resp.~$\wt^\pi$ and $\bt^\pi$, and also $\tw^\sigma$ and $\tb^\sigma$) form {\em adjoint pairs}.
\begin{lemma}
\label{lemma:tense and lifted adjunction}
For any tense slanted algebra $\mathbb{A}$ which is based on a bounded lattice,
\begin{enumerate}
    \item if $\mathbb{A} = (A, \Diamond, \blacksquare)$, then $\Diamond^\sigma u\leq v$ iff $u\leq \blacksquare^\pi v$ for all $u, v\in A^\delta$.
    \item if $\mathbb{A} = (A, \wt, \bt)$, then $ u\leq \wt^\pi v$ iff $v\leq \bt^\pi u$ for all $u, v\in A^\delta$.
    \item if $\mathbb{A} = (A, \tw, \tb)$, then $ \tw^\sigma u\leq  v$ iff $\tb^\sigma  v\leq  u$ for all $u, v\in A^\delta$.
\end{enumerate}
\end{lemma}
\begin{proof}
1. Let $u, v\in A^\delta$. In what follows,  $k \in K(A^\delta)$,  $o \in O(A^\delta)$, and $a, b \in A$.
\begin{center}
    \begin{tabular}{r cll}
        $\Diamond^\sigma u\leq v$ &iff & $ \bigvee \{\Diamond^\sigma k\mid  k\leq u\}\leq \bigwedge \{o\mid  v\leq o\} $& denseness + def $\Diamond^\sigma u$ \\
        &iff & $  \bigvee \{\Diamond^\sigma k\mid k\leq u\}\leq \bigwedge \{o\mid  v\leq o\}$ &  \\
       & iff   & $\forall k\forall o(k\leq u\ \&\ v\leq o\Rightarrow \Diamond^\sigma k\leq o)$\\
       & iff   & $\forall k\forall o(k\leq u\ \&\ v\leq o\Rightarrow  k\leq \blacksquare^\pi o)$ & $(\ast)$\\
       &iff & $\bigvee\{k\mid   k\leq u\}\leq \bigwedge \{\blacksquare^\pi o\mid v\leq o\}$\\
       &iff & $u\leq \blacksquare^\pi v$ & denseness + def $\blacksquare^\pi v$.
    \end{tabular}
\end{center}
To justify the equivalence marked with $(\ast)$, let us show that $\Diamond^\sigma k\leq o$ iff $k\leq \blacksquare^\pi o$ for all $k$ and $o$.
\begin{center}
    \begin{tabular}{r cll}
   $\Diamond^\sigma k\leq o$ &iff & $ \bigwedge \{\Diamond a\mid  k\leq a\}\leq \bigvee \{b\mid  b\leq o\} $&  def $\Diamond^\sigma k$ and open element \\
   & iff   & $\exists a\exists b(k\leq a \ \&\ b\leq o\ \&\ \Diamond a\leq b)$ & compactness\\
   & iff   & $\exists a\exists b(k\leq a \ \&\ b\leq o\ \&\  a\leq \blacksquare b)$ & $\mathbb{A}$ is tense\\
  &  iff & $\bigwedge \{ a\mid  k\leq a\}\leq \bigvee \{\blacksquare b\mid  b\leq o\} $ &  compactness \\
  &  iff & $k\leq \blacksquare o$.& def $\blacksquare^\pi o$ and closed element.
    \end{tabular}
\end{center}
Compactness can be applied in the argument above since, thanks to $A$ being a bounded lattice and by items 3 and 4 of Proposition \ref{prop: compactness and existential ackermann}, $\bigwedge \{\Diamond a\mid k\leq a\}\in K(A^\delta)$ and $\bigvee \{\blacksquare b\mid b\leq o\}\in O(A^\delta)$. The proof of the remaining items is similar and is omitted.
\end{proof}
 By well known order-theoretic facts about adjoint pairs of maps (cf.~\cite{davey2002introduction}), the lemma above immediately implies the following
\begin{corollary}
\label{cor:tense implies normal}
    Any tense and lattice-based slanted algebra  $\mathbb{A}$  is normal.
\end{corollary}
In what follows,  slanted algebras will  be considered in the context of some input/output logic based on some  (fully) selfextensional logic $\mathcal{L}$. Therefore, for any algebraic signature $\mathcal{L}$ and  any modal signature $\tau$ disjoint from $\mathcal{L}$,  a {\em slanted} $\mathcal{L}_\tau$-{\em algebra} is a structure $\bba = (A, \tau^{\mathbb{A}})$, such that $A$ is an  $\mathcal{L}$-algebra, and $\tau^{\mathbb{A}}$ is a $\tau$-indexed set of operations $A\to A^\delta$  each of which is either c-slanted or o-slanted.   For any slanted $\mathcal{L}_\tau$-algebra $\bba$, any  assignment $v:\mathsf{Prop}\to \bba$ uniquely extends to a homomorphism $v: \mathcal{L}_\tau\to \bbas$ (abusing notation, the same symbol denotes   both the assignment and its homomorphic extension).  Hence,

\begin{definition}
\label{def:slanted satisfaction and validity}
 For any  algebraic signature $\mathcal{L}$ and any modal signature $\tau$, an $\mathcal{L}_\tau$-inequality $\phi\leq\psi$ is {\em satisfied} in a slanted $\mathcal{L}_\tau$-algebra $\bba$ under the assignment $v$ (notation: $(\bba, v)\models \phi\leq\psi$) if $(\bbas, e\cdot v)\models \phi\leq\psi$ in the usual sense, where $e\cdot v$ is the assignment on $\bbas$ obtained by composing the assignment $v:\mathsf{Prop}\to \bba$ and the canonical embedding $e: \bba\to \bbas$.

 Moreover, $\phi\leq\psi$ is {\em valid} in $\bba$ (notation: $\bba\models \phi\leq\psi$) if $(\bbas, e\cdot v)\models \phi\leq\psi$ for every assignment $v$ into $\bba$  (notation: $\bbas\models_{\bba} \phi\leq\psi$).
\end{definition}

\section{Algebraic models of normative and permission systems on selfextensional logics} \label{sec: algebraic models}
In the present section, we discuss the three main types of algebras with relations which generalize the structures discussed in Section \ref{ssec:subordination related}, and which form the basic semantic environment for normative and permission systems on selfextensional logics.

\subsection{(Proto-)subordination algebras}
\label{ssec:subordination}

\begin{definition}[(Proto-)subordination algebra]
\label{def: subordination algebra} A {\em proto-subordination algebra} is a tuple $\mathbb{S} = (A, \prec)$ such that $A$ is a (possibly bounded) poset (with bottom denoted $\bot$ and top denoted $\top$ when they exist), and $\prec\ \subseteq A\times A$. A  proto-subordination algebra  is named as indicated in the left-hand column in the table below when $\prec$ satisfies the properties indicated in the right-hand column. In what follows, we will refer to a proto-subordination algebra  $\mathbb{S} = (A, \prec)$ as e.g.~join- or meet-semilattice based (abbreviated as $\vee$-SL based and $\wedge$-SL based, respectively), {\em (distributive) lattice-based ((D)L-based)}, or {\em Boolean-based (B-based)} if $A$ is a (distributive) lattice, a Boolean algebra, and so on. More in general, for any fully selfextensional logic $\mathcal{L}$, we say that $\mathbb{S} = (A, \prec)$ is $\mathsf{Alg}(\mathcal{L})$-{\em based} if $A\in \mathsf{Alg}(\mathcal{L})$. 

\begin{center}
\begin{tabular}{rlrl}
     $\mathrm{(\bot)}$ & $\bot \prec \bot$  &
     $\mathrm{(\top)}$ & $\top \prec \top$ \\

$\mathrm{(SB)}$ & $ \exists b( b \prec a)$ &
$\mathrm{(SF)}$ & $ \exists b( a \prec b)$\\
$\mathrm{(SI)}$ & $a \leq b \prec x \Rightarrow a \prec x$ &
     $\mathrm{(WO)}$ & $b \prec x \leq y \Rightarrow b \prec y$ \\
     $\mathrm{(AND)}$ & $a \prec x \ \& \ a \prec y \Rightarrow a \prec x \wedge y$ &
     $\mathrm{(OR)}$ & $a \prec x \ \& \ b \prec x \Rightarrow a \vee b \prec x$ \\
     $\mathrm{(DD)}$ & \multicolumn{3}{l}{ $a \prec x_1\ \&\ a \prec x_2 \Rightarrow \exists x (a \prec x\ \&\ x \leq x_1\ \&\ x\leq x_2)$} \\
	 $\mathrm{(UD)}$ & \multicolumn{3}{l}{ $a_1 \prec x\ \&\ a_2 \prec x \Rightarrow \exists a (a \prec x\ \&\ a_1 \leq a\ \&\ a_2 \leq a)$} \\

\end{tabular}
\end{center}

\end{definition}
\begin{center}
	\begin{tabular}{ l l}
		\hline
		Name & Properties  \\
		\hline
  	$\Diamond$-defined & (DD) (SF)   \\
   $\blacksquare$-defined & (UD) (SB)   \\
   defined & $\Diamond$-defined + $\blacksquare$-defined  \\
		$\Diamond$-premonotone & $\Diamond$-defined + (SI)   \\
			$\blacksquare$-premonotone & $\blacksquare$-defined + (WO)   \\

			$\Diamond$-monotone & $\Diamond$-premonotone + (WO)  \\
			$\blacksquare$-monotone & $\blacksquare$-premonotone + (SI)   \\
			monotone & $\Diamond$-monotone + $\blacksquare$-monotone  \\
			$\Diamond$-regular & $\Diamond$-monotone + (OR) \\
			$\blacksquare$-regular &$\blacksquare$-monotone + (AND) \\
			regular & $\Diamond$-regular  +  $\blacksquare$-regular\\
			$\Diamond$-normal & $\Diamond$-regular + ($\bot$)  \\
			$\blacksquare$-normal & $\blacksquare$-regular + ($\top$) \\
			subordination algebra & $\Diamond$-normal + $\blacksquare$-normal\\
		\hline
	\end{tabular}
\end{center}

In what follows, we will also consider the  properties listed below. Some of them are the algebraic counterparts of well known closure rules of normative systems in input/output logic \cite{Makinson00,Makinson03}, and others have been considered in the context of subordination algebras and related structures \cite{celani2016precontact}.

\begin{center}
\begin{tabular}{rlrl}

    (D) & $a \prec c \Rightarrow \exists b(a \prec b\ \&\ b \prec c)$ & (S6) & $a \prec b \Rightarrow \neg b \prec \neg a$ \\
     (CT) & $a \prec b \ \& \ a \wedge b \prec c \Rightarrow a \prec c$ &  (T) & $a \prec b \ \& \ b \prec c \Rightarrow a \prec c$ \\
     (DCT) & $c\prec a\vee b \ \&\ b\prec a \Rightarrow c\prec a$ &   \\

     (S9)& \multicolumn{3}{l}{ $\exists c( c \prec b\ \&\ x\prec a \vee c) \Leftrightarrow \exists a^\prime \exists b^\prime (a^\prime \prec a\ \&\ b^\prime \prec b\ \&\ x \leq a^\prime \vee b^\prime)$
     } \\
     (SL1) & \multicolumn{3}{l}{ $a\prec b \vee c \Rightarrow \exists b^\prime \exists c^\prime (b^\prime \prec b\ \&\ c^\prime \prec c\ \&\ a \prec b^\prime \vee c^\prime)$
     } \\
     (SL2) & \multicolumn{3}{l}{ $b \wedge c \prec a \Rightarrow \exists b^\prime \exists c^\prime (b^\prime \prec b\ \&\ c^\prime \prec c\ \&\ b^\prime \wedge  c^\prime \prec a)$
     }
\end{tabular}
\end{center}

\begin{lemma}
For any 
proto-subordination algebra $\mathbb{S} = (A, \prec)$,
\begin{enumerate}
\item if $\mathbb{S}$ is $\vee$-semilattice based, then
\begin{enumerate}[label=(\roman*)]
    \item $\mathbb{S}\models\mathrm{(OR)}$ implies $\mathbb{S}\models\mathrm{(UD)}$;
    \item if $\mathbb{S}\models\mathrm{ (SI)}$, then $\mathbb{S}\models\mathrm{(UD)}$ iff $\mathbb{S}\models\mathrm{(OR)}$;
    \end{enumerate}
\item if $\mathbb{S}$ is $\wedge$-semilattice based, then
\begin{enumerate}[label=(\roman*)]
    \item $\mathbb{S}\models\mathrm{(AND)}$ implies $\mathbb{S}\models\mathrm{(DD)}$;

    \item if $\mathbb{S}\models\mathrm{ (WO)}$, then $\mathbb{S}\models\mathrm{ (DD)}$ iff $\mathbb{S}\models\mathrm{ (AND)}$.
\end{enumerate}
\end{enumerate}
\end{lemma}
\begin{proof}
    1(i) and 2(i) are straightforward. As for 1(ii), by 1(i), to complete the proof we need to show the `only if' direction. Let $a, b, x\in A$ s.t.~$a \prec x$ and $b \prec x$. By (UD), this implies that $c\prec x$ for some $c\in A$ such that $a\leq c$ and $b\leq c$. Since $A$ is a $\vee$-semilattice, this implies that $a\vee b\leq c\prec x$, and by (SI), this implies that $a\vee b\prec x$, as required.
   2(ii) is proven similarly.
   \end{proof}

   \begin{lemma}
For any bounded proto-subordination algebra $\mathbb{S} = (A, \prec)$,
\begin{enumerate}

\item If $\mathbb{S}\models\mathrm{(SI)}$, then
\begin{enumerate}[label=(\roman*)]
 \item
   $\mathbb{S}\models\mathrm{(\top)}$ implies $\mathbb{S}\models\mathrm{(SF)}$;
  \item
   $\mathbb{S}\models\mathrm{(SB)}$ implies $\mathbb{S}\models\mathrm{(\bot)}$.
\end{enumerate}
\item If $\mathbb{S}\models\mathrm{(WO)}$, then
\begin{enumerate}[label=(\roman*)]
    \item
$\mathbb{S}\models\mathrm{(\bot)}$ implies $\mathbb{S}\models\mathrm{(SB)}$ ;
\item
$\mathbb{S}\models\mathrm{(SF)}$ implies $\mathbb{S}\models\mathrm{(\top)}$.
\end{enumerate}
\end{enumerate}
\end{lemma}
\begin{proof}
    1. (i) Let $a\in A$. Then by assumption, $ a\leq\top \prec \top$, which by (SI) implies $a\prec \top$, as required.

    (ii) (SB) implies that $\bot \leq b\prec \bot$ for some  $b\in A$. Hence, $\bot\prec \bot$ follows from (SI), as required.
    The proof of 2. is similar.
   \end{proof}

Normative systems can be interpreted in proto-subordination algebras as follows:

\begin{definition}
 A {\em model} for an input/output logic $\mathbb{L} =(\mathcal{L}, N)$ is a tuple $\mathbb{M} = (\mathbb{S}, h)$ s.t.~$\mathbb{S} = (A, \prec)$ is an $\mathsf{Alg}(\mathcal{L})$-based proto-subordination algebra, 
 and $h:\mathrm{Fm}\to A$ is a homomorphism s.t.~for all $\varphi, \psi\in \mathrm{Fm}$, if $(\varphi, \psi)\in N$, then $h(\varphi)\prec h(\psi)$.

\end{definition}

\subsection{(Proto-)precontact algebras }
\label{ssec: precontact}
\begin{definition}[(Proto-)precontact algebra]
\label{def: protocontact algebra} A {\em proto-precontact algebra} is a tuple $\mathbb{C} = (A, \pcon)$ such that $A$ is a (possibly bounded) poset (with bottom denoted $\bot$ and top denoted $\top$ when they exist), and $\pcon\ \subseteq A\times A$. A  proto-precontact algebra  is named as indicated in the left-hand column in the table below when $\pcon$ satisfies the properties indicated in the right-hand column.\footnote{In the tables below, most properties are stated contrapositively and mention $\npcon$ rather than $\pcon$, since this is the form in which they will be used in the subsequent sections.} In what follows, we will refer to a proto-precontact algebra  $\mathbb{C} = (A, \pcon)$ as e.g.~{\em (distributive) lattice-based ((D)L-based)}, or {\em Boolean-based (B-based)} if $A$ is a (distributive) lattice, a Boolean algebra, and so on. More in general, for any fully selfextensional logic $\mathcal{L}$, we say that $\mathbb{C} = (A, \pcon)$ is $\mathsf{Alg}(\mathcal{L})$-{\em based} if $A\in \mathsf{Alg}(\mathcal{L})$.

\begin{center}
\begin{tabular}{rlrl}
     $\mathrm{(\bot)}^\wt$ & $\bot \cancel\pcon \top$  &
     $\mathrm{(\top)}^\wt$ & $\top \cancel{\pcon} \bot$ \\
      $\mathrm{(SB)}^\wt$ & $ \exists b( b \cancel{\pcon} a)$ & $\mathrm{(SF)}^\wt$ & $ \exists b( a \cancel{\pcon} b)$\\
     $\mathrm{(SI)}^\wt$ & $b \npcon c\ \&\ a \leq b \Rightarrow a \npcon c$ &
     $\mathrm{(WO)}^\wt$ & $a \npcon c \ \&\ b \leq c \Rightarrow a \npcon b$ \\
     $\mathrm{(AND)}^\wt$ & $a \npcon b\ \&\ a\npcon c \Rightarrow a \npcon (b\vee c) $ &
     $\mathrm{(OR)}^\wt$ & $a\npcon c\ \&\ b\npcon c \Rightarrow (a\vee b )\npcon c $ \\

\end{tabular}
\end{center}
\end{definition}

\begin{center}
	\begin{tabular}{ l l}
		\hline
		Name & Properties  \\
		\hline
		  	$\wt$-defined & (DD)$^\wt$ (SF)$^\wt$   \\
   $\bt$-defined & (UD)$^\wt$ (SB)$^\wt$   \\
   defined &   	$\wt$-defined +	$\bt$-defined  \\
		$\wt$-preantitone & 	$\wt$-defined + (SI)$^\wt$   \\
			$\bt$-preantitone & 	$\bt$-defined + (WO)$^\wt$   \\
			preantitone & $\wt$-preantitone + $\bt$-preantitone  \\
			$\wt$-antitone & $\wt$-preantitone + (WO)$^\wt$  \\
			$\bt$-antitone & $\bt$-preantitone + (SI)$^\wt$   \\
			antitone & $\wt$-antitone + $\bt$-antitone \\
			$\wt$-regular & $\wt$-antitone  + (OR)$^\wt$ \\
			$\bt$-regular & $\bt$-antitone  + (AND)$^\wt$ \\
			regular & $\wt$-regular + $\bt$-regular\\
			$\wt$-normal &  $\wt$-regular + ($\bot$)$^\wt$  \\
			$\bt$-normal &  $\bt$-regular + ($\top$)$^\wt$ \\
			precontact algebra & $\wt$-normal + $\bt$-normal\\
		\hline
	\end{tabular}
\end{center}

In what follows, we will also consider the properties listed below:\footnote{See also Section \ref{prop:poly correspondence}, above Proposition \ref{prop:poly correspondence permission}, for more conditions also involving $\prec$.}

\begin{center}
\begin{tabular}{rlrl}

    (NS) & $a\npcon a\Rightarrow a\leq \bot $ &  \\
    (SFN) & $\exists b(a\npcon b\ \&\ \top \leq a\vee b) $ & \\
    (ALT)$^\wt$ & \multicolumn{3}{l}{$a\npcon b \Rightarrow \exists c\exists d(\top \leq c \vee d \ \&\ a\npcon c \ \&\ b\npcon d) $  } \\
    (ALT)$^\bt$ & \multicolumn{3}{l}{$b\npcon a \Rightarrow \exists c\exists d(\top \leq c \vee d \ \&\ c\npcon a \ \&\ d\npcon b) $  } \\
  $\mathrm{(CMO)}$  & $a\npcon b\ \&\ a\npcon c \Rightarrow   (a\wedge c)\npcon b$\\
$\mathrm{(\vee\wedge)}$ & $(a\vee y)\npcon x \ \&\ a\npcon(x\wedge y)\Rightarrow  (a\wedge x)\npcon y$
\end{tabular}
\end{center}

\begin{lemma}
For any  $\vee$-semilattice-based proto-precontact algebra $\mathbb{C} = (A, \pcon)$,
\begin{enumerate}[label=(\roman*)]
       \item $\mathbb{C}\models\mathrm{(OR)^\wt}$ implies $\mathbb{C}\models\mathrm{(UD)^\wt}$;
    \item $\mathbb{C}\models\mathrm{(AND)^\wt}$ implies $\mathbb{C}\models\mathrm{(DD)^\wt}$;
    \item if $\mathbb{C}\models\mathrm{ (SI)^\wt}$, then $\mathbb{C}\models\mathrm{(UD)^\wt}$ iff $\mathbb{C}\models\mathrm{(OR)^\wt}$;
    \item if $\mathbb{C}\models\mathrm{ (WO)^\wt}$, then $\mathbb{C}\models\mathrm{ (DD)^\wt}$ iff $\mathbb{C}\models\mathrm{ (AND)^\wt}$.
\end{enumerate}
\end{lemma}
\begin{proof}
As to (i), let $a, b\in A$ s.t.~$\forall d ( a\leq d\ \&\ b \leq d \Rightarrow d\pcon c)$; then $(a\vee b)\pcon c$, which implies, by (OR)$^\wt$, that either  $a\pcon c$ or $ b\pcon c$, as required. The proof of (ii) is similar. As to (iii), by (i), to complete the proof we need to show the `only if' direction. Let $a, b, c\in A$ s.t.~$(a\vee b)\pcon c$. Then, by (SI)$^\wt$, the left-hand side of (UD)$^\wt$ is satisfied with $a_1: = a$, $a_2: = b$,  and $x: = c$. Therefore,  $a\pcon c$ or $b\pcon c$, as required. The proof of (iv) is similar.
\end{proof}

\begin{lemma}
For any  bounded proto-precontact algebra $\mathbb{C} = (A, \pcon)$,
\begin{enumerate}

\item If $\mathbb{C}\models\mathrm{(SI)^\wt}$, then
\begin{enumerate}[label=(\roman*)]
    \item $\mathbb{C}\models\mathrm{(\top)^\wt}$ implies $\mathbb{C}\models\mathrm{(SF)^\wt}$;
    \item $\mathbb{C}\models\mathrm{(SB)^\wt}$ implies $\mathbb{C}\models\mathrm{(\bot)^\wt}$.
\end{enumerate}
\item If $\mathbb{C}\models\mathrm{(WO)^\wt}$, then
\begin{enumerate}[label=(\roman*)]
    \item
$\mathbb{C}\models\mathrm{(\bot)^\wt}$ implies $\mathbb{C}\models\mathrm{(SB)^\wt}$;
\item
$\mathbb{C}\models\mathrm{(SF)^\wt}$ implies $\mathbb{C}\models\mathrm{(\top)^\wt}$.
\end{enumerate}
\end{enumerate}
\end{lemma}
\begin{proof}
    (i)   Let $a\in A$. Then by assumption, $ a\leq\top \npcon \bot$, which by (SI)$^\wt$ implies $a\npcon \bot $, as required.

     (ii) (SB)$^\wt$ implies that $\bot \leq b\npcon \top $ for some  $b\in A$. Hence, $\bot\npcon \top $ follows from (SI)$^\wt$, as required.
    The proof of 2. is similar.
   \end{proof}

\subsection{Dual (proto-)precontact algebras}
\label{ssec: postcontact algebras}

\begin{definition}[Dual (proto-)precontact algebra]
\label{def: protopostcontact algebra} A {\em dual proto-precontact algebra} is a tuple $\mathbb{D} = (A, \ppcon)$ such that $A$ is a (possibly bounded) poset (with bottom denoted $\bot$ and top denoted $\top$ when they exist), and $\ppcon\ \subseteq A\times A$. A  dual proto-precontact algebra  is named as indicated in the left-hand column in the table below when $\ppcon$ satisfies the properties indicated in the right-hand column.\footnote{In the tables below, most properties are stated contrapositively and mention $\nppcon$ rather than $\ppcon$, since this is the form in which they will be used in the subsequent sections.} In what follows, we will refer to a dual proto-precontact algebra  $\mathbb{D} = (A, \ppcon)$ as e.g.~{\em (distributive) lattice-based ((D)L-based)}, or {\em Boolean-based (B-based)} if $A$ is a (distributive) lattice, a Boolean algebra, and so on. More in general, for any fully selfextensional logic $\mathcal{L}$, we say that $\mathbb{D} = (A, \ppcon)$ is $\mathsf{Alg}(\mathcal{L})$-{\em based} if $A\in \mathsf{Alg}(\mathcal{L})$.

\begin{center}
\begin{longtable}{rlrl}
     $\mathrm{(\bot)}^\tw$ & $\top \nppcon \bot$  &
     $\mathrm{(\top)}^\tw$ & $\bot \nppcon \top$ \\


      $\mathrm{(SB)}^\tw$ & $ \exists b( b \nppcon a)$ & $\mathrm{(SF)}^\tw$ & $ \exists b( a \nppcon b)$\\

     $\mathrm{(SI)}^\tw$ & $a \leq b\ \&\ a \nppcon c \Rightarrow b \nppcon c$ &
     $\mathrm{(WO)}^\tw$ & $a \nppcon c \ \&\ c \leq b \Rightarrow a \nppcon b$ \\
     $\mathrm{(AND)}^\tw$ & $a \nppcon\ \&\ a\nppcon c\Rightarrow a \nppcon (b\wedge c) $ &
     $\mathrm{(OR)}^\tw$ & $a \nppcon c\ \&\ b\nppcon c \Rightarrow (a\wedge b)\nppcon c $ \\

     $\mathrm{(DD)}^\tw$ & \multicolumn{3}{l}{ $(a \nppcon x_1\ \&  \ a \nppcon x_2) \Rightarrow \exists x [x \leq x_1\ \&\ x\leq x_2\ \&\ a \nppcon x] $} \\
	 $\mathrm{(UD)}^\tw$ & \multicolumn{3}{l}{ $(a_1 \nppcon x\ \& \ a_2 \nppcon x )\Rightarrow \exists a [ a \leq a_1\ \&\ a \leq a_2\ \&\ a \nppcon x] $} \\

\end{longtable}
\end{center}
\end{definition}

\begin{center}
	\begin{tabular}{ l l}
		\hline
		Name & Properties  \\
		\hline
		  	$\tw$-defined & (DD)$^\tw$ (SF)$^\tw$   \\
   $\tb$-defined & (UD)$^\tw$ (SB)$^\tw$   \\
   defined &   	$\tw$-defined +	$\tb$-defined  \\
		$\tw$-preantitone & 	$\tw$-defined + (SI)$^\tw$   \\
			$\tb$-preantitone & 	$\tb$-defined + (WO)$^\tw$   \\
			preantitone & $\tw$-preantitone + $\tb$-preantitone  \\
			$\tw$-antitone & $\tw$-preantitone + (WO)$^\tw$  \\
			$\tb$-antitone & $\tb$-preantitone + (SI)$^\tw$   \\
			antitone & $\tw$-antitone + $\tb$-antitone \\
			$\tw$-regular & $\tw$-antitone  + (OR)$^\tw$ \\
			$\tb$-regular & $\tb$-antitone  + (AND)$^\tw$ \\
			regular & $\tw$-regular + $\tb$-regular\\
			$\tw$-normal &  $\tw$-regular + ($\bot$)$^\tw$  \\
			$\tb$-normal &  $\tb$-regular + ($\top$)$^\tw$ \\
			dual precontact algebra & $\tw$-normal + $\tb$-normal\\
		\hline
	\end{tabular}
\end{center}

In what follows, we will also consider the properties listed below:\footnote{Unlike the other conditions in the list, $\mathrm{(CT)}^\tw$ also involves $\prec$, and will be treated separately in Proposition \ref{prop:poly correspondence permission}.5.}

\begin{center}
\begin{tabular}{rlrl}

    (SR) & $a\nppcon a\Rightarrow \top \leq a $ & $\mathrm{(CMO)}^\tw$ & $a\nppcon b\ \&\ a\nppcon c\Rightarrow (a\vee b)\nppcon c$   \\
    (SFN)$^\tw$ & $\exists b(a\nppcon b\ \&\ a\wedge b \leq \bot)$ & $\mathrm{(CT)}^\tw$ & $a\nppcon b\ \& \ b\pdla a \prec c \Rightarrow a \nppcon c$ \\
\end{tabular}
\end{center}
\begin{lemma}
For any  $\wedge$-semilattice based  dual proto-precontact algebra $\mathbb{D} = (A, \ppcon)$,
     \begin{enumerate}[label=(\roman*)]
       \item $\mathbb{D}\models\mathrm{(OR)^\tw}$ implies $\mathbb{D}\models\mathrm{(UD)^\tw}$;
    \item $\mathbb{D}\models\mathrm{(AND)^\tw}$ implies $\mathbb{D}\models\mathrm{(DD)^\tw}$;
    \item if $\mathbb{D}\models\mathrm{ (SI)^\tw}$, then $\mathbb{D}\models\mathrm{(UD)^\tw}$ iff $\mathbb{D}\models\mathrm{(OR)^\tw}$;
    \item if $\mathbb{D}\models\mathrm{ (WO)^\tw}$, then $\mathbb{D}\models\mathrm{ (DD)^\tw}$ iff $\mathbb{D}\models\mathrm{ (AND)^\tw}$.
\end{enumerate}
\end{lemma}
\begin{proof}
As to (i), let $a, b\in A$ s.t.~$\forall d ( d\leq a\ \&\ d \leq b \Rightarrow d\ppcon c)$; then $(a\wedge b)\ppcon c$, which implies, by (OR)$^\tw$, that either  $a\ppcon c$ or $ b\ppcon c$, as required. The proof of (ii) is similar. As to (iii), by (i), to complete the proof we need to show the `only if' direction. Let $a, b, c\in A$ s.t.~$(a\wedge b)\ppcon c$. Then, by (SI)$^\tw$, the left-hand side of (UD)$^\tw$ is satisfied with $a_1: = a$, $a_2: = b$,  and $x: = c$. Therefore,  $a\pcon c$ or $b\pcon c$, as required. The proof of (iv) is similar.
\end{proof}
\begin{lemma}
For any  bounded dual proto-precontact algebra $\mathbb{D} = (A, \ppcon)$,
\begin{enumerate}

\item If $\mathbb{D}\models\mathrm{(SI)^\tw}$, then
\begin{enumerate}[label=(\roman*)]
    \item $\mathbb{D}\models\mathrm{(\top)^\tw}$ implies $\mathbb{D}\models\mathrm{(SF)^\tw}$;
    \item $\mathbb{D}\models\mathrm{(SB)^\tw}$ implies $\mathbb{D}\models\mathrm{(\bot)^\tw}$.
\end{enumerate}
\item If $\mathbb{D}\models\mathrm{(WO)^\tw}$, then
\begin{enumerate}[label=(\roman*)]
    \item
$\mathbb{D}\models\mathrm{(\bot)^\tw}$ implies $\mathbb{D}\models\mathrm{(SB)^\tw}$;
\item
$\mathbb{D}\models\mathrm{(SF)^\tw}$ implies $\mathbb{D}\models\mathrm{(\top)^\tw}$.
\end{enumerate}

\end{enumerate}
\end{lemma}
\begin{proof}
     (i)   Let $a\in A$. Then by assumption, $ \bot \nppcon \top \ \&\ \bot \leq a $, which by (SI)$^\tw $ implies $a\nppcon \top$, as required.

     (ii) (SB)$^\tw$ implies that $b \nppcon \bot \ \&\ b \leq \top $ for some  $b\in A$. Hence, $\top \nppcon \bot $ follows from (SI)$^\tw$, as required.
    The proof of 2. is similar.
\end{proof}

\subsection{Algebraizing negative and dual negative permission systems}
We finish the present section by introducing the algebraic counterparts of the negative and dual negative permission systems associated with a given normative system in the literature on input/output logic. Originally introduced in the context of normative systems on classical propositional logic \cite{Makinson00}, these notions have been generalized in \cite[Sections 4.1 and 4.2]{dedomenico2024obligations} to normative systems based on arbitrary selfextensional logics. In what follows, we extend these notions to  the algebraic counterparts of input/output logics on fully selfextensional logics. 
For any such logic $\mathcal{L}$, any $A\in \mathsf{Alg}(\mathcal{L})$, and any subset $X\subseteq A$, we let $F(X)$ denote the $\mathcal{L}$-filter generated by $X$. We will write e.g.~$F(a, b)$ for $F(\{a, b\})$.
\begin{definition}
\label{def:compatible P}
Let $\mathcal{L}$ be a fully selfextensional logic. 
For any $\mathsf{Alg}(\mathcal{L})$-based proto-subordination algebra $\mathbb{S} = (A, \prec)$,  we let
    \begin{center}
        $\mathcal{C}_\prec: = \{(a, b)\mid \forall c (a\prec c  \Rightarrow F(b, c) \neq A   )\}\quad $ $\quad \quad\mathcal{D}_\prec: = \{(a, b)\mid \exists c (c\not\prec b  \ \&\ F(a, c)  = A   )\} $.
    \end{center}
\end{definition}
The following proposition is the algebraic counterpart of \cite[Propositions 4.2 and  4.7]{dedomenico2024obligations}.
\begin{proposition}
\label{prop: projection negative and dual negative}
    Let $\mathcal{L}$ be a  selfextensional logic with $\wedge_P$, $\bot_P$,  $\neg_A$ and $\neg_S$. For any  $\mathsf{Alg}(\mathcal{L})$-based proto-subordination algebra $\mathbb{S} = (A, \prec)$, 
    \begin{enumerate}
    \item if $\mathbb{S}\models \mathrm{(WO)}$, then
    $\mathcal{C}_\prec = \{(a, b)\mid a \not \prec  \neg b\} $;
    \item if $\mathbb{S}\models \mathrm{(SI)}$, then $ \mathcal{D}_\prec = \{(a, b)\mid  \neg a\not\prec  b\}$.
    \end{enumerate}\end{proposition}
    \begin{proof}
    The assumptions imply that $F(a, \neg a) = A$ for every $a\in A$ and $A\models \forall a\forall b(F (a, b) = A \Rightarrow b\leq \neg a)$, where $F(a, b)$ is the $\mathcal{L}$-filter of $A$ generated by $\{a, b\}$. \\
        1. For the left-to-right inclusion, let $c: = \neg b$; hence  $F( b,\neg b) = A $, which implies, by the definition of $\mathcal{C}_{\prec}$, $a\not\prec \neg b$, as required.
        Conversely, let $a, b\in A$ s.t.~$a\not\prec \neg b$ and let $c\in A$ s.t.~$F( b, c) = A$.  By the preliminary observations on $A$, this implies that  $c\leq \neg b$. Hence, $a\not\prec c$, for otherwise, by $\mathrm{(WO)}$,  $a\prec c$ and $c\leq \neg b$ would imply that $a\prec \neg b$, against the assumption.\\
        2.  For the right-to-left inclusion, take $c: = \neg a$ as the witness; as discussed above, $F(a,\neg a) = A$, as required.
        Conversely,  let $a, b\in A$ s.t.~$c\not\prec b$ for some $c\in A$ s.t.~$F( a, c) = A$. As discussed above, this implies that   $c\leq  \neg a$. Hence, $\neg a\not \prec  b$, for otherwise, by $\mathrm{(SI)}$,  $c\prec b$, against the assumption.
    \end{proof}

\begin{definition}
    For any  proto-subordination algebra $\mathbb{S} = (A, \prec)$,
  \begin{enumerate}
      \item the  proto-precontact algebra associated with $\mathbb{S}$ is $\mathbb{S}_{\wt} := (A, \pcon_{\prec})$;
      \item the dual proto-precontact algebra associated with $\mathbb{S}$ is $\mathbb{S}_{\tw} := (A, \ppcon_{\prec})$.

      \end{enumerate}
\end{definition}
    The following is a straightforward consequence of  Proposition \ref{prop: projection negative and dual negative}.
\begin{corollary}
	 Let $\mathbb{S} = (A, \prec)$ 
  be a proto-subordination algebra as in Proposition \ref{prop: projection negative and dual negative}.  
  For any $\mathrm{(X)\ \in \{(\top), (\bot),}$ $\mathrm{ (SB), (SF), (SI), (WO), (DD), (UD), (AND), (OR)}\}$,
  \begin{enumerate}
      \item $\mathbb{S}\models \mathrm{(X)}\, $ iff $\, \mathbb{S}_\wt \models \mathrm{(X)^{\wt}}$.
       \item $\mathbb{S}\models \mathrm{(X)}\, $ iff $\, \mathbb{S}_\tw\models \mathrm{(X)}^\tw$.

  \end{enumerate}
\end{corollary}
The corollary above can be further refined to various more general classes of proto-subordination algebras, so as to extend  \cite[Propositions 4.5 and 4.8]{dedomenico2024obligations} to an algebraic setting.


An $\mathsf{Alg}(\mathcal{L})$ proto-subordination algebra  $\mathbb{S} = (A, \prec)$ is {\em internally incoherent} if  $a\prec b$ and $ a\prec c$ for some $a, b, c\in A$  such that  $F(a) \neq A$ and $F(b,c) = A$;  $\mathbb{S}$ is {\em internally coherent} if it is not internally incoherent.

\section{Slanted algebras arising from defined relational algebras}
\label{sec: proto and slanted}

\begin{definition}
\noindent Let $\mathbb{S} = (A, \prec)$, $\mathbb{C} = (A, \pcon)$, and $\mathbb{D} = (A, \ppcon)$   be a proto-subordination algebra, a proto-precontact algebra, and a dual proto-precontact algebra, respectively.
\begin{enumerate}
\item if $\mathbb{S}\models \mathrm{(DD)+(UD)+(SB)+(SF)}$, then the slanted algebra associated with  $\mathbb{S}$ is  $\mathbb{S}^* = (A, \Diamond, \blacksquare)$ s.t.~$\Diamond a \coloneqq \bigwedge {\prec}[a]$ and $\blacksquare a \coloneqq \bigvee {\prec^{-1}}[a]$ for any $a$.
\item if $\mathbb{C}\models \mathrm{(DD)^\wt+(UD)^\wt+(SB)^\wt+(SF)^\wt}$, then the slanted algebra associated with  $\mathbb{C}$ is  $\mathbb{C}^* = (A, \wt, \bt)$ s.t.~$\wt a \coloneqq \bigvee (\pcon[a])^c$ and $\bt a \coloneqq \bigvee ({\pcon^{-1}}[a])^c$ for any $a$.
\item
$\mathbb{D}\models \mathrm{(DD)^\tw+(UD)^\tw+(SB)^\tw+(SF)^\tw}$, then the slanted algebra associated with  $\mathbb{D}$ is  $\mathbb{D}^* = (A, \tw, \tb)$ s.t.~$\tw a \coloneqq \bigwedge (\ppcon[a])^c$ and $\tb a \coloneqq \bigwedge ({\ppcon^{-1}}[a])^c$ for any $a$.
\end{enumerate}
\end{definition}
The assumption $\mathbb{S}\models \mathrm{(DD)}$ implies that ${\prec}[a]$ is down-directed for every $a\in A$, and $\mathbb{S}\models \mathrm{(SF)}$  implies that ${\prec}[a]\neq \varnothing$. Hence, $\Diamond a\in K(A^\delta)$. Likewise, $\mathbb{S}\models \mathrm{(UD)+(SB)}$ guarantees that $\blacksquare a\in O(A^\delta)$, and $\mathbb{C}\models \mathrm{(DD)^\wt+(SF)^\wt}$ (resp.~$\mathbb{D}\models \mathrm{(DD)^\tw+(SF)^\tw}$) guarantees that $\wt a\in O(A^\delta)$ (resp.~$\tw a\in K(A^\delta)$) and $\mathbb{C}\models \mathrm{(UD)^\wt+(SB)^\wt}$ (resp.~$\mathbb{D}\models \mathrm{(UD)^\tw+(SB)^\tw}$) guarantees that $\bt a\in O(A^\delta)$ (resp.~$\tb a\in K(A^\delta)$) for all $a\in A$.

\begin{lemma} \label{lem: diamond-output equivalence}
For any proto-subordination algebra $\mathbb{S} = (A, \prec)$, any proto-precontact algebra $\mathbb{C} = (A, \pcon)$, any dual proto-precontact algebra $\mathbb{D} = (A, \ppcon)$ and all $a, b\in A$,
\begin{enumerate}
  \item if $ \mathbb{S}$ is defined (i.e.~$ \mathbb{S}\models \mathrm{(DD) + (SF)+(UD) + (SB)}$), then
    \begin{enumerate}[label=(\roman*)]
    \item $a \prec b$ implies $\Diamond a \leq b$ and $a \leq \blacksquare b$.
    \item if $ \mathbb{S}\models \mathrm{(WO)}$, then $\Diamond a\leq b$ iff $a\prec b$.
    \item if $\mathbb{S}\models \mathrm{ (SI) }$, then $a\leq \blacksquare b$ iff $a\prec b$.
    \end{enumerate}
  \item  if $ \mathbb{C}$ is defined (i.e.~$ \mathbb{C}\models \mathrm{(DD)^\wt + (SF)^\wt+(UD)^\wt + (SB)^\wt}$), then
    \begin{enumerate}[label=(\roman*)]
    \item $a \cancel\pcon b $ implies    $ b \leq \wt a $ and $a \leq \bt b$.
    \item if $ \mathbb{C}\models \mathrm{(WO)^\wt}$, then $b \leq \wt a $ iff $a\cancel\pcon b$.
    \item if $\mathbb{C}\models \mathrm{ (SI)^\wt}$, then $a\leq \bt b$ iff $a\cancel\pcon b$.
\end{enumerate}
\item  if $ \mathbb{D}$ is defined (i.e.~$ \mathbb{D}\models \mathrm{(DD)^\tw + (SF)^\tw+(UD)^\tw + (SB)^\tw}$), then
    \begin{enumerate}[label=(\roman*)]
    \item $a \nppcon b $ implies    $ \tw a \leq b $ and $\tb b \leq a$.
    \item if $ \mathbb{D}\models \mathrm{(WO)^\tw}$, then $\tw a \leq b $ iff $a\nppcon b$.
    \item if $\mathbb{D}\models \mathrm{ (SI)^\tw}$, then $\tb b\leq a$ iff $a\nppcon b$.
\end{enumerate}
\end{enumerate}
\end{lemma}
\begin{proof} Below, we group together the proofs of sub-items (i) of each item, and likewise for  sub-items (ii) and (iii).

    (i) $a\prec b$ iff $b\in {\prec}[a]$ iff $a \in {\prec}^{-1}[b]$, hence $a\prec b$ implies $b\geq \bigwedge {\prec}[a] = \Diamond a$ and $a \leq \bigvee{\prec}^{-1}[b] = \blacksquare b$.

    Likewise, $a\npcon b$ iff $b\notin {\pcon}[a]$ (i.e.~$b\in ({\pcon}[a])^c$) iff $a \notin {\pcon}^{-1}[b]$ (i.e.~$a \in ({\pcon}^{-1}[b])^c$). Hence, $a\npcon b$ implies $b\leq \bigvee ({\pcon}[a])^c = \wt a$ and $a \leq \bigvee({\pcon}^{-1}[b])^c= \bt b$. Finally,  $a\nppcon b$ iff $b\notin {\ppcon}[a]$ (i.e.~$b\in ({\ppcon}[a])^c$) iff $a \notin {\ppcon}^{-1}[b]$ (i.e.~$a \in ({\ppcon}^{-1}[b])^c$). Hence, $a\nppcon b$ implies $b\geq \bigwedge ({\ppcon}[a])^c = \tw a$ and $a \geq \bigwedge({\ppcon}^{-1}[b])^c= \tb b$.

     (ii) By (i), to complete the proof, we need to show the `only if' direction. The assumption $ \mathbb{S}\models \mathrm{(DD)+(SF)}$ implies that  ${\prec}[a]$ is non-empty and down-directed  for any $a\in A$. Hence,  by compactness,  $\bigwedge {\prec}[a] = \Diamond a\leq b$ implies  that $c \leq b$ for some $c \in {\prec}[a]$, i.e.~$a \prec c \leq b$ for some $ c \in A$, and by (WO), this implies that $a\prec b$, as required.

Likewise, the assumption $ \mathbb{C}\models \mathrm{(DD)^\wt+(SF)^\wt}$ implies that  $({\pcon}[a])^c$ is non-empty and up-directed   for any $a\in A$. Hence,  by compactness,  $b\leq \wt a =\bigvee ({\pcon}[a])^c$ implies  that $b \leq c$ for some $c \in ({\pcon}[a])^c$ (i.e.~$a\npcon c$), i.e.~$a \cancel{\pcon} c$ and $b\leq c$ for some $ c \in A$, i.e.~$a \npcon c$ and $b\geq c$ for some $ c \in A$. By 
$\mathrm{(WO)}^\wt$, this implies that $a\npcon b$, as required.

Finally, $ \mathbb{D}\models \mathrm{(DD)^\tw+(SF)^\tw}$ implies that  $({\ppcon}[a])^c$ is non-empty and down-directed for any $a\in A$. Hence,  by compactness, $\bigwedge ({\ppcon}[a])^c  = \tw a \leq b $ implies  that $c \leq b$ for some $c \in ({\ppcon}[a])^c$ (i.e.~$a\nppcon c$), i.e.~$a \nppcon c$ and $c\leq b$ for some $ c \in A$. By 
$\mathrm{(WO)}^\tw$, this implies that $a \nppcon b $, as required.

    (iii) is proven similarly to (ii), by observing that $\mathbb{S}\models \mathrm{(UD)+(SB)}$ implies that ${\prec}^{-1}[a]$ is non-empty and up-directed for every $a\in A$, that $\mathbb{C}\models \mathrm{(UD)^\wt+(SB)^\wt}$  implies that $(\pcon ^{-1}[a])^c$ is nonempty and up-directed for every $a\in A$, and that $\mathbb{D}\models \mathrm{(UD)^\tw+(SB)^\tw}$ implies that $({\ppcon}^{-1}[a])^c$ is non-empty and  down-directed for every $a\in A$.
\end{proof}

An immediate consequence of the lemma above and Corollary \ref{cor:tense implies normal} is the following
\begin{corollary}
\label{cor: defined+si+wo imply tense hence normal}
    For any defined proto-subordination algebra $\mathbb{S} = (A, \prec)$, any defined proto-precontact algebra $\mathbb{C} = (A, \pcon)$, and any defined dual proto-precontact algebra $\mathbb{D} = (A, \ppcon)$, if $A$ is a bounded lattice, then
    \begin{enumerate}
        \item if $\mathbb{S}\models \mathrm{(SI)+(WO)}$, then $\mathbb{S}^\ast$ is tense, hence is normal.
        \item if $\mathbb{C}\models \mathrm{(SI)^\wt+(WO)^\wt}$, then $\mathbb{C}^\ast$ is tense, hence is normal.
        \item if $\mathbb{D}\models \mathrm{(SI)^\tw+(WO)^\tw}$, then $\mathbb{D}^\ast$ is tense, hence is normal.
    \end{enumerate}
\end{corollary}
\begin{lemma} \label{lem: rleation to axioms}
For any defined proto-subordination algebra $\mathbb{S} = (A, \prec)$, any defined proto-precontact algebra $\mathbb{C} = (A, \pcon)$, and any defined dual proto-precontact algebra $\mathbb{D} = (A, \ppcon)$,
\begin{enumerate}
\item If $\mathbb{S}$ is bounded and   $\mathbb{S}\models\mathrm{(\bot)}$, then $\mathbb{S}^\ast\models \Diamond\bot \leq \bot$.
\item If $\mathbb{S}$ is bounded and  $\mathbb{S}\models\mathrm{(\top)}$, then  $\mathbb{S}^\ast\models\top\leq \blacksquare\top$.
\item If $\mathbb{C}$ is bounded and  $\mathbb{C}\models\mathrm{(\bot)^\wt}$, then $\mathbb{C}^\ast\models \top \leq \wt \bot$.
\item If $\mathbb{C}$ is bounded and    $\mathbb{C}\models\mathrm{(\top)^\wt}$, then $\mathbb{C}^\ast\models  \top  \leq \bt\bot$.
\item If $\mathbb{D}$ is bounded and  $\mathbb{D}\models\mathrm{(\bot)^\tw}$, then $\mathbb{D}^\ast\models \tw \top \leq  \bot$.
\item If $\mathbb{D}$ is bounded and    $\mathbb{D}\models\mathrm{(\top)^\tw}$, then $\mathbb{D}^\ast\models  \tb \top  \leq \bot$.
\item If  $\mathbb{S}\models\mathrm{(SI)}$, then
    $\Diamond$  on $\mathbb{S}^*$ is monotone;

  \item if $\mathbb{S}\models\mathrm{(WO)}$, then
 $\blacksquare$ on $\mathbb{S}^*$ is monotone.

\item If  $\mathbb{C}\models\mathrm{ (SI)^\wt}$, then
    $\wt$  on $\mathbb{C}^*$ is antitone.
\item If  $\mathbb{C}\models\mathrm{ (WO)^\wt}$, then
    $\bt$  on $\mathbb{C}^*$ is antitone.
\item If  $\mathbb{D}\models\mathrm{ (SI)^\tw}$, then    $\tw$  on $\mathbb{D}^*$ is antitone.
\item If  $\mathbb{D}\models\mathrm{ (WO)^\tw}$, then
    $\tb$  on $\mathbb{D}^*$ is antitone.

\end{enumerate}
\end{lemma}

\begin{proof} 1. By assumption,  $\bot\in {\prec}[\bot]$; hence, $\Diamond\bot = \bigwedge {\prec}[\bot]\leq \bot$, as required. The proofs of 2-6 are similar.
	7. Let $a, b\in A$ s.t.~$a \leq b$.  To show that $\Diamond a = \bigwedge{\prec}[a]\leq \bigwedge{\prec}[b] = \Diamond b$, it is enough to show that ${\prec}[b]\subseteq {\prec}[a]$, i.e.~that if $x\in A$ and $b\prec x$, then $a\prec x$. Indeed, by (SI), $a\leq b\prec x$ implies $a\prec x$, as required. Item  8 is shown similarly. Likewise, as to 9, to show that $\wt b = \bigvee ({\pcon}[b])^c\leq \bigvee ({\pcon}[a])^c = \wt a$ whenever $a\leq b$, it is enough to show that $ ({\pcon}[b])^c\subseteq  ({\pcon}[a])^c$, i.e.~that $ {\pcon}[a]\subseteq  {\pcon}[b]$, i.e.~that if $c\in A$ and $a\pcon c$, then $b\pcon c$, which is exactly what is guaranteed by (SI)$^\wt$ under $a\leq b$. The proofs of 10-12 are similar.
\end{proof}

The following lemma shows that the converse conditions of Lemma \ref{lem: rleation to axioms} hold under additional assumptions. 

\begin{lemma}\label{lem: axioms to relation}
For any  defined proto-subordination algebra $\mathbb{S} =(A, \prec)$, any  defined proto-precontact algebra $\mathbb{C} =(A, \pcon)$, and any  defined dual proto-precontact algebra $\mathbb{D} =(A, \ppcon)$,
\begin{enumerate}
    \item If $\mathbb{S}\models\mathrm{(WO)}$,  then
    \begin{enumerate}[label=(\roman*)]
        \item $\mathbb{S}\models\mathrm{(SI)}\quad$ iff $\quad \Diamond$ on $\mathbb{S}^\ast$ is monotone.
        \item if   $\mathbb{S}$ is $\vee$-SL based, then $\mathbb{S}\models\mathrm{(OR)}\quad $ iff $\quad \mathbb{S}^\ast\models \Diamond (a \vee b) \leq \Diamond a \vee \Diamond b$.
        \item if   $\mathbb{S}$ is bounded, then $\mathbb{S}\models(\bot)\quad $ iff $\quad\mathbb{S}^\ast\models\Diamond \bot \leq \bot$.
    \end{enumerate}
    \item If $\mathbb{S}\models\mathrm{(SI)} $, then
    \begin{enumerate}[label=(\roman*)]
        \item $\mathbb{S}\models\mathrm{(WO)}\quad $ iff $\quad\blacksquare$ on $\mathbb{S}^\ast$ is monotone.
        \item if   $\mathbb{S}$ is $\wedge$-SL based, then $\mathbb{S}\models\mathrm{(AND)}\quad $ iff $\quad\mathbb{S}^\ast\models\blacksquare a \wedge \blacksquare b \leq \blacksquare (a \wedge b)$.
        \item if   $\mathbb{S}$ is bounded, then $\mathbb{S}\models(\top)\quad $ iff $\quad\mathbb{S}^\ast\models\top  \leq \blacksquare\top$.
    \end{enumerate}
     \item If $\mathbb{C}\models\mathrm{(WO)^\wt}$,  then
    \begin{enumerate}[label=(\roman*)]
        \item $\mathbb{C}\models\mathrm{(SI)^\wt}\quad$ iff $\quad \wt $ on $\mathbb{C}^\ast$ is antitone.
        \item if   $\mathbb{C}$ is $\vee$-SL based, then $\mathbb{C}\models\mathrm{(OR)^\wt}\quad $ iff $\quad \mathbb{C}^\ast\models \wt a \wedge \wt b \leq \wt (a \vee  b)$.
        \item if   $\mathbb{C}$ is bounded, then $\mathbb{C}\models(\bot)^\wt\quad $ iff $\quad\mathbb{C}^\ast\models \top \leq \wt \bot$.
    \end{enumerate}
    \item If $\mathbb{C}\models\mathrm{(SI)^\wt}$, then
    \begin{enumerate}[label=(\roman*)]
        \item $\mathbb{C}\models\mathrm{(WO)^\wt}\quad $ iff $\quad\bt$ on $\mathbb{C}^\ast$ is antitone.
        \item if   $\mathbb{C}$ is $\vee$-SL based, then $\mathbb{C}\models\mathrm{(AND)^\wt}\; $ iff $\;
        \mathbb{C}^\ast\models\bt a \wedge \bt b \leq \bt (a \vee b)$.
        \item if   $\mathbb{C}$ is bounded, then $\mathbb{C}\models(\top)^\wt\quad $ iff $\quad\mathbb{C}^\ast\models\top \leq \bt \bot $.
    \end{enumerate}
\item If $\mathbb{D}\models\mathrm{(WO)^\tw}$,  then
    \begin{enumerate}[label=(\roman*)]
        \item $\mathbb{D}\models\mathrm{(SI)^\tw}\quad$ iff $\quad \tw $ on $\mathbb{D}^\ast$ is antitone.
        \item if   $\mathbb{D}$ is $\vee$-SL based, then $\mathbb{D}\models\mathrm{(OR)^\tw}\quad $ iff $\quad \mathbb{D}^\ast\models \tw (a \wedge b) \leq \tw a \vee \tw b$.
        \item if   $\mathbb{D}$ is bounded, then $\mathbb{D}\models(\bot)^\tw\quad $ iff $\quad\mathbb{D}^\ast\models \tw \top \leq  \bot$.
    \end{enumerate}
  \item If $\mathbb{D}\models\mathrm{(SI)^\tw}$, then
    \begin{enumerate}[label=(\roman*)]
        \item $\mathbb{D}\models\mathrm{(WO)^\tw}\quad $ iff $\quad\tb$ on $\mathbb{D}^\ast$ is antitone.
        \item if   $\mathbb{D}$ is $\vee$-SL based, then $\mathbb{D}\models\mathrm{(AND)^\tw}\; $ iff $\;
        \mathbb{D}^\ast\models\tb (a \wedge  b )\leq \tb a \vee \tb b$.
        \item if   $\mathbb{D}$ is bounded, then $\mathbb{D}\models(\top)^\tw\quad $ iff $\quad\mathbb{D}^\ast\models\tb\top  \leq \bot$.
    \end{enumerate}
    \end{enumerate}
\end{lemma}

\begin{proof}
  1. (i)  By Lemma \ref{lem: rleation to axioms}  7, the proof is complete if we show the `if' direction.  Let $a, b, x\in A$ s.t.~$a\leq b\prec x$. By Lemma \ref{lem: diamond-output equivalence} 1(ii), to show that $a\prec x$,  it is enough to show that $\Diamond a\leq x$. Since $\Diamond$ is monotone, $a\leq b$ implies $\Diamond a\leq \Diamond b$, and, by Lemma \ref{lem: diamond-output equivalence} 1(i), $b\prec x$ implies that $\Diamond b\leq x$. Hence, $\Diamond a\leq x$, as required.

  (ii) From left to right, let $a, b\in A$. Since both $\Diamond(a\vee b)$ and $\Diamond a\vee \Diamond b$ are closed elements of $A^\delta$ (cf.~Proposition \ref{prop:background can-ext} 1(iv)), 
 by Proposition \ref{prop:background can-ext} 1(i), to show that $\Diamond (a\vee b)\leq \Diamond a \vee \Diamond b$, it is enough to show that for any $x\in A$, if $\Diamond a \vee \Diamond b\leq x$, then $\Diamond (a\vee b)\leq x$.

\begin{center}
\begin{tabular}{rcll}
$\Diamond a \vee \Diamond b \leq x$ &
iff & $\Diamond a \leq x$ and $\Diamond b \leq x$  \\
&iff & $a \prec x$ and $b \prec x$ & Lemma \ref{lem: diamond-output equivalence}.1(ii) (WO) \\
& implies & $a \vee b \prec x$ &  (OR)\\
& implies & $\Diamond (a \vee b) \leq x$ & Lemma \ref{lem: diamond-output equivalence}.1(i)\\
\end{tabular}
\end{center}

Conversely, let $a, b, x\in A$ s.t.~$a\prec x$ and $b\prec x$. By Lemma \ref{lem: diamond-output equivalence} (ii), to show that $a\vee b\prec x$,  it is enough to show that $\Diamond (a\vee b)\leq x$, and since $\mathbb{S}^\ast\models \Diamond (a \vee b) \leq \Diamond a \vee \Diamond b$, it is enough to show that $\Diamond a\vee\Diamond b\leq x$, i.e.~that $\Diamond a\leq x$ and $ \Diamond b\leq x$. These two inequalities hold by Lemma \ref{lem: diamond-output equivalence} 1(i), and the assumptions on $a, b$ and $x$.

    (iii)  By Lemma \ref{lem: diamond-output equivalence} (ii),  $\bot\prec \bot$ is equivalent to $\Diamond \bot\leq \bot$, as required.
    2. is proven similarly.

    3. (i)  By Lemma \ref{lem: rleation to axioms}  9, the proof is complete if we show the `if' direction. Let $a, b, x\in A$ s.t.~$a\leq b\cancel{\pcon} x$. By Lemma \ref{lem: diamond-output equivalence} 2(ii), to show that $a\cancel{\pcon} x$,  it is enough to show that $x\leq \wt a$. Since $\wt$ is antitone, $a\leq b$ implies $\wt b\leq \wt a$, and, by Lemma \ref{lem: diamond-output equivalence} 2(i), $b\cancel{\pcon} x$ implies that $x\leq \wt b$. Hence, $x\leq \wt a$, as required.

   (ii) From left to right, let $a, b\in A$. Since both $\wt(a\vee b)$ and $\wt a\wedge \wt b$ are open elements of $A^\delta$ (cf.~Proposition \ref{prop:background can-ext} 1(v)), 
   by Proposition \ref{prop:background can-ext} 1(ii), to show that $\wt a \wedge \wt b\leq \wt (a\vee b)$, it is enough to show that for any $x\in A$, if $x\leq \wt a \wedge \wt b$, then $x\leq \wt (a\vee b)$.

\begin{center}
\begin{tabular}{rcll}
$x\leq \wt a \wedge \wt b $ &
iff & $x\leq \wt a $ and $x\leq \wt b$  \\
&iff & $x \cancel{\pcon} a$ and $x \cancel{\pcon} b$ & Lemma \ref{lem: diamond-output equivalence}.2(ii) (WO)$^\wt$ \\
& implies & $(a \vee b) \cancel{\pcon} x$ &  (OR)$^\wt$ \\
& implies & $x\leq \wt (a \vee b) $ & Lemma \ref{lem: diamond-output equivalence}.2(i)\\
\end{tabular}
\end{center}

   Conversely, let $a, b, x\in A$ s.t.~$a\cancel{\pcon} x$ and $b\cancel{\pcon} x$. By Lemma \ref{lem: diamond-output equivalence}.2(ii), to show that $(a\vee b)\cancel{\pcon} x$,  it is enough to show that $x\leq \wt (a\vee b)$, and since $\mathbb{S}^\ast\models \wt a \wedge \wt b \leq \wt (a \vee b)$, it is enough to show that $x\leq \wt a \wedge \wt b $, i.e.~that $x\leq \wt a$ and $x\leq  \wt b$. These two inequalities hold by Lemma \ref{lem: diamond-output equivalence}.2(i), and the assumptions on $a, b$ and $x$.

    (iii)  By Lemma \ref{lem: diamond-output equivalence}.2(ii),  $\bot\cancel{\pcon} \top$ is equivalent to $\top \leq \wt \bot$, as required.
    4. is proven similarly.

 5. (i)  By Lemma \ref{lem: rleation to axioms}.11, the proof is complete if we show the `if' direction. For  contraposition, let $a, b, x\in A$ s.t.~$a\leq b\ \&\ a\nppcon x$. By Lemma \ref{lem: diamond-output equivalence}.3(ii), to show that $b\nppcon x$,  it is enough to show that $\tw b \leq x$. Since $\tw$ is antitone, $a\leq b$ implies $\tw b\leq \tw a$, and, by Lemma \ref{lem: diamond-output equivalence}.3(i), $a\nppcon x$ implies that $\tw a\leq x $. Hence, $\tw b\leq x$, as required.

   (ii) From left to right, let $a, b\in A$. Since both $\tw(a\wedge b)$ and $\tw a\vee \tw b$ are closed elements of $A^\delta$ (cf.~Proposition \ref{prop:background can-ext}.1(iv)),   by Proposition \ref{prop:background can-ext}.1(i), to show that $\tw (a \wedge b) \leq \tw a\vee \tw b$, it is enough to show that for any $x\in A$, if  $\tw a\vee \tw b \leq x$, then$ \tw (a \wedge b) \leq x$.

\begin{center}
\begin{tabular}{rcll}
 $\tw a\vee \tw b \leq x$ &
iff &  $\tw a \leq x\ \&\ \tw b\leq x$ & \\
&iff & $a \nppcon x$ and $b\nppcon b$ & Lemma \ref{lem: diamond-output equivalence}.3(ii) (WO)$^\tw$ \\
& implies & $(a \wedge b) \nppcon x$ &  (OR)$^\tw$ \\
& implies & $\tw  (a \wedge b)\leq x $ & Lemma \ref{lem: diamond-output equivalence}.3(i)\\
\end{tabular}
\end{center}

   Conversely, for contraposition, let $a, b, x\in A$ s.t.~$a\nppcon x$ and $b \nppcon x$. By Lemma \ref{lem: diamond-output equivalence}.3(ii), to show that $(a\wedge b)\nppcon x$,  it is enough to show that $\tw (a\wedge b)\leq x$, and since $\mathbb{S}^\ast\models \tw (a \wedge b) \leq \tw a \vee \tw b$, it is enough to show that $\tw a \vee \tw b \leq x $, i.e.~that $ \tw a\leq x$ and $ \tw b\leq x$. These two inequalities hold by Lemma \ref{lem: diamond-output equivalence}.3(i), and the assumptions on $a, b$ and $x$.

    (iii)  By Lemma \ref{lem: diamond-output equivalence}.3(ii),  $ \top\nppcon\bot$ is equivalent to $\tw \top \leq  \bot$, as required.
    6. is proven similarly.    \end{proof}
\begin{corollary}
\label{cor: charact monotone reg norm}
For any  proto-subordination algebra $\mathbb{S} = (A, \prec)$, any  proto-precontact algebra $\mathbb{C} = (A, \pcon)$, and any  dual proto-precontact algebra $\mathbb{D} = (A, \ppcon)$,
\begin{enumerate}
\item if $\mathbb{S}$ is defined,  then
    \begin{enumerate}[label=(\roman*)]
        \item
 $\mathbb{S}$ is monotone iff $\mathbb{S}^*$ is monotone;
    \item $\mathbb{S}$ is regular iff $\mathbb{S}^*$ is regular;
    \item $\mathbb{S}$ is a subordination algebra iff $\mathbb{S}^*$ is normal.
    \end{enumerate}

\item if $\mathbb{C}$ is defined,  then
    \begin{enumerate}[label=(\roman*)]
        \item $\mathbb{C}$ is antitone iff $\mathbb{C}^*$ is antitone;
    \item $\mathbb{C}$ is regular iff $\mathbb{C}^*$ is regular;
    \item $\mathbb{C}$ is a precontact algebra iff $\mathbb{C}^*$ is normal.
    \end{enumerate}

\item if $\mathbb{D}$ is defined,  then
    \begin{enumerate}[label=(\roman*)]
        \item $\mathbb{D}$ is antitone iff $\mathbb{D}^*$ is antitone;
    \item $\mathbb{D}$ is regular iff $\mathbb{D}^*$ is regular;
    \item $\mathbb{D}$ is a dual precontact algebra iff $\mathbb{D}^*$ is normal.
    \end{enumerate}
\end{enumerate}
\end{corollary}



\begin{lemma} \label{lem: diamond-output equivalence extended}
 For any  proto-subordination algebra $\mathbb{S} = (A, \prec)$, any  proto-precontact algebra $\mathbb{C} = (A, \pcon)$ and any dual proto-precontact algebra $\mathbb{D} = (A, \ppcon)$, for all $a, b\in A$, $k, h \in K(A^\delta)$, and $o, p \in O(A^\delta)$, and all $D\subseteq A$  nonempty and down-directed,  and $U\subseteq A$  nonempty and up-directed,
\begin{enumerate}
    \item if  $\mathbb{S}$ is $\Diamond$-monotone (i.e.~$\mathbb{S}\models\mathrm{(SF)+ (DD) + (WO) + (SI)}$), then
    \begin{enumerate}[label=(\roman*)]
        \item  ${\prec}[D]\coloneqq \{c\ |\ \exists a (a \in  D\ \&\ a \prec c)\}$ is nonempty and down-directed;
        \item if $k = \bigwedge D$, then $\Diamond k = \bigwedge {\prec}[D] \in K(A^\delta)$;
        \item $\Diamond k \leq b$ implies $a\prec b$ for some $a \in A$ s.t.~$ k \leq a$;
         \item $\Diamond k \leq o$ implies $a\prec b$ for some $a,b \in A$ s.t.~$ k \leq a$ and $b\leq o$.
    \end{enumerate}
    \item if    $\mathbb{S}$ is $\blacksquare$-monotone (i.e.~$\mathbb{S}\models\mathrm{(SB)+ (UD) + (WO) + (SI)}$), then
    \begin{enumerate}[label=(\roman*)]
        \item ${\prec}^{-1}[U]\coloneqq \{c\ |\ \exists a (a \in  U\ \&\ c \prec a)\}$ is nonempty and up-directed;
        \item if $o = \bigvee U$, then $\blacksquare o = \bigvee {\prec}^{-1}[U] \in O(A^\delta)$;
        \item $a \leq \blacksquare o$ implies $a\prec b$ for some $b \in A$ s.t.~$ b \leq o$.
        \item $k \leq \blacksquare o$ implies $a\prec b$ for some $a,b \in A$ s.t.~$k\leq a$ and $ b \leq o$.
    \end{enumerate}

    \item if  $\mathbb{C}$ is $\wt$-antitone (i.e.~$\mathbb{C}\models\mathrm{(SF)^\wt+ (DD)^\wt + (WO)^\wt + (SI)^\wt}$), then
    \begin{enumerate}[label=(\roman*)]
        \item  ${\npcon}[D]\coloneqq \{c\ |\ \exists a (a \in  D\ \&\ a \npcon c)\}$ is nonempty and up-directed;
        \item if $k = \bigwedge D$, then $\wt k = \bigvee {\npcon}[D] \in O(A^\delta)$;
        \item $a\leq \wt k $ implies $c\npcon a$ for some $c \in A$ s.t.~$ k \leq c$;
         \item $h\leq \wt k $ implies $c\npcon b$  for some $b, c\in A$ s.t.~$h\leq b$ and  $k\leq c$.
    \end{enumerate}
    \item if  $\mathbb{C}$ is $\bt$-antitone (i.e.~$\mathbb{C}\models\mathrm{(SB)^\wt+ (UD)^\wt + (WO)^\wt + (SI)^\wt}$), then
    \begin{enumerate}[label=(\roman*)]
        \item  ${\npcon}^{-1}[D]\coloneqq \{c\ |\ \exists a (a \in  D\ \&\ c \npcon a)\}$ is nonempty and up-directed;
        \item if $k = \bigwedge D$, then $\bt k = \bigvee {\npcon}^{-1}[D] \in O(A^\delta)$;
        \item $a\leq \bt k $ implies $a\npcon c$ for some $c \in A$ s.t.~$ k \leq c$;
         \item $h\leq \bt k $ implies $b\npcon c$  for some $b, c\in A$ s.t.~$h\leq b$ and  $k\leq c$.
    \end{enumerate}
\item if  $\mathbb{D}$ is $\tw$-antitone (i.e.~$\mathbb{D}\models\mathrm{(SF)^\tw+ (DD)^\tw + (WO)^\tw + (SI)^\tw}$), then
    \begin{enumerate}[label=(\roman*)]
        \item  $\nppcon [U]\coloneqq \{c\ |\ \exists a (a \in  U\ \&\ a \nppcon c)\}$ is nonempty and down-directed;
        \item if $o = \bigvee U$, then $\tw o = \bigwedge \nppcon[U] \in K(A^\delta)$;
        \item $\tw o \leq a $ implies $c\nppcon a$ for some $c \in A$ s.t.~$ c \leq o$;
         \item $\tw o\leq p $ implies $c\nppcon a$  for some $a, c\in A$ s.t.~$c\leq o$ and  $a\leq p$.
    \end{enumerate}
    \item if  $\mathbb{D}$ is $\tb$-antitone (i.e.~$\mathbb{D}\models\mathrm{(SB)^\tw+ (UD)^\tw + (WO)^\tw + (SI)^\tw}$), then
    \begin{enumerate}[label=(\roman*)]
        \item  ${\npcon}^{-1}[U]\coloneqq \{c\ |\ \exists a (a \in  U\ \&\ c \nppcon a)\}$ is nonempty and down-directed;
        \item if $o = \bigvee U$, then $\tb o = \bigwedge {\nppcon}^{-1}[U] \in K(A^\delta)$;
        \item $\tb o \leq a $ implies $a \nppcon c$ for some $c \in A$ s.t.~$ c \leq o$;
         \item $\tb o \leq p $ implies $a\nppcon c$  for some $a, c\in A$ s.t.~$a\leq p$ and  $c\leq o$.
    \end{enumerate}
\end{enumerate}
\end{lemma}
\begin{proof}
  1.   (i) By (SF), $D$  nonempty implies that so is ${\prec}[D]$.  If $c_i  \in {\prec}[D]$ for $1\leq i\leq 2$, then $a_i \prec c_i$  for some $a_i \in D$. Since $D$ is down-directed, some $a \in D$ exists s.t.~$a \leq a_i$ for each $i$. Thus,  $\mathrm{(SI)}$ implies that $a \prec c_i$, from which the claim  follows by $\mathrm{(DD)}$.

  (ii) By definition, $\Diamond k \coloneqq \bigwedge \{ \Diamond a\ |\ a \in A, k \leq a\} = \bigwedge \{c\ |\ \exists a (a \prec c\ \&\  k \leq a)\}$. Since $k = \bigwedge D$, by compactness, $k \leq a$ iff $d \leq a$ for some $d \in D$. Thus, $\Diamond k = \bigwedge  \{c\ |\ \exists a (a \prec c\ \&\  k \leq a)\} = \bigwedge \{c\ |\ \exists a (a \in \lfloor D\rfloor\ \&\ a \prec c)\} = \bigwedge {\prec}[D]\in K(A^\delta)$, the last membership holding by (i).

  (iii) Let $k = \bigwedge D$ for some $D\subseteq A$ nonempty and down-directed. By (ii),  $K(A^\delta)\ni \Diamond k = \bigwedge {\prec}[D]$. Hence, $\Diamond k\leq b$ implies by compactness that $c \leq b$ for some $c\in A$ s.t.~$a \prec c$ for some $a \in  D$ (hence $k = \bigwedge D\leq a$). By $\mathrm{(WO)}$, this implies that $a \prec b$ for some $a \in A$ s.t.~$k \leq a$, as required.

  (iv) By (ii), $\Diamond k   \in K(A^\delta)$. Moreover, $o\in O(A^\delta)$ implies that $o = \bigvee U$ for some nonempty and up-directed $U\subseteq A$. Hence, by compactness and (iii), $\Diamond k\leq o$ implies that $a\prec b$ for some $a\in A$ s.t.~$k\leq a$ and some $b\in U$ (for which $b\leq o$). The proof of 2 is similar.

  3.   (i) By (SF)$^\wt$, $D$ is nonempty  implies that so is $\npcon[D]$.  If $c_i  \in \npcon[D]$ for $1\leq i\leq 2$, then $a_i \npcon c_i$  for some $a_i \in D$. Since $D$ is down-directed, some $a \in D$ exists s.t.~$a \leq a_i$ for each $i$. Thus,  $\mathrm{(SI)^\wt}$ implies that $a \npcon c_i$, from which the claim  follows by $\mathrm{(DD)^\wt}$.

  (ii) By definition, $\wt k \coloneqq \bigvee \{ \wt a\ |\ a \in A, k \leq a\} = \bigvee \{c\ |\ \exists a (a \npcon c\ \&\  k \leq a)\}$. Since $k = \bigwedge D$, by compactness, $k \leq a$ iff $d \leq a$ for some $d \in D$. Thus, $\wt k = \bigvee  \{c\ |\ \exists a (a \npcon c\ \&\  k \leq a)\} = \bigvee \{c\ |\ \exists a (a \in \lfloor D\rfloor\ \&\ a \npcon c)\} = \bigvee \npcon[D]\in O(A^\delta)$, the last membership holding by (i).

   (iii) Let $k = \bigwedge D$ for some $D\subseteq A$ nonempty and down-directed. By (ii),  $O(A^\delta)\ni \wt k = \bigvee \npcon[D]$. Hence, $a\leq \wt k$ implies by compactness that $a \leq b$ for some $b\in A$ s.t.~$c \npcon b$ for some $c \in  D$ (hence $k = \bigwedge D\leq c$). By $\mathrm{(WO)}^\wt$, this implies that $c \npcon a$ for some $c \in A$ s.t.~$k \leq c$, as required.

   (iv) By (ii), $\wt k   \in O(A^\delta)$. Moreover, $h\in K(A^\delta)$ implies that $h = \bigwedge E$ for some nonempty and down-directed $E\subseteq A$. Hence, by compactness and (iii), $h\leq \wt k$ implies that $h\leq b$ for some $b\in \npcon [D]$, i.e.~$c\npcon b$ for some $c\in D$  (for which $k\leq c$). The proof of 4 is similar.

5.   (i) By (SF)$^\tw$, $D$ is nonempty  implies that so is $\nppcon[U]$.  If $c_i  \in \nppcon[U]$ for $1\leq i\leq 2$, then $a_i \nppcon c_i$  for some $a_i \in U$. Since $U$ is up-directed, some $a \in U$ exists s.t.~$a_i \leq a$ for each $i$. Thus,  $\mathrm{(SI)^\tw}$ implies that $a \nppcon c_i$, from which the claim  follows by $\mathrm{(DD)^\tw}$.

  (ii) By definition, $\tw o \coloneqq \bigwedge \{ \tw a\ |\ a \in A, a \leq o\} = \bigwedge \{c\ |\ \exists a (a \nppcon c\ \&\  a \leq o)\}$. Since $o = \bigvee U$, by compactness, $a \leq o$ iff $a \leq d$ for some $d \in U$. Thus, $\tw o = \bigwedge  \{c\ |\ \exists a (a \nppcon c\ \&\  a \leq o)\} = \bigwedge \{c\ |\ \exists a (a \in \lceil U\rceil\ \&\ a \nppcon c)\} = \bigwedge \nppcon[U]\in K(A^\delta)$, the last membership holding by (i).

   (iii) Let $o = \bigvee U$ for some $U\subseteq A$ nonempty and down-directed. By (ii),  $K(A^\delta)\ni \tw o = \bigwedge \nppcon[U]$. Hence, $\tw o\leq a$ implies by compactness that $b \leq a$ for some $b\in A$ s.t.~$c \nppcon b$ for some $c \in  U$ (hence $o = \bigvee U\geq c$). By $\mathrm{(WO)}^\tw$, this implies that $c \nppcon a$ for some $c \in A$ s.t.~$c \leq o$, as required.

   (iv) By (ii), $\tw o  \in K(A^\delta)$. Moreover, $p\in O(A^\delta)$ implies that $p = \bigvee V$ for some nonempty and up-directed $V\subseteq A$. Hence, by compactness and (iii), $\tw o\leq  p$ implies that $b\leq p$ for some $b\in \nppcon [U]$, i.e.~$c\nppcon b$ for some $c\in U$  (for which $c\leq o$). The proof of 4 is similar.
\end{proof}


\section{Modal characterization of classes of relational algebras}
\label{sec: modal charact}

\begin{proposition}
\label{prop:characteriz}
For any  proto-subordination algebra $\mathbb{S} = (A, \prec)$,
\begin{enumerate}
\item  If $\mathbb{S}$ is monotone, then
\begin{enumerate}[label=(\roman*)]
   \item  $\mathbb{S}\models\; \prec\ \subseteq\ \leq\ $ iff $\ \mathbb{S}^*\models a\leq \Diamond a\ $ iff $\ \mathbb{S}^*\models \blacksquare a\leq  a$.

   \item  $\mathbb{S}\models\; \leq\ \subseteq\ \prec\ $ iff $\ \mathbb{S}^*\models \Diamond a\leq a\ $ iff $\ \mathbb{S}^*\models  a\leq \blacksquare a$.
    \item $\mathbb{S}\models \mathrm{(T)}\quad $  iff $\quad \mathbb{S}^*\models\Diamond a\leq \Diamond \Diamond a$.
\item
    $\mathbb{S}\models \mathrm{(D)}\quad $  iff $\quad \mathbb{S}^*\models\Diamond\Diamond a\leq  \Diamond a$.
    \end{enumerate}
\item If $\mathbb{S}$ is monotone and  $\wedge$-SL based, then
\begin{enumerate}[label=(\roman*)]
              \item
$\mathbb{S}\models \mathrm{(CT)}\ $  iff $\ \mathbb{S}^*\models\Diamond a\leq \Diamond (a\wedge \Diamond a)$.
 \item
    $\mathbb{S}\models \mathrm{(SL2)}\ $  iff $\ \mathbb{S}^*\models \Diamond (\Diamond a \wedge \Diamond b) \leq \Diamond (a \wedge b)$.
    \end{enumerate}
\item If $\mathbb{S}$ is monotone and $\vee$-SL based, then
\begin{enumerate}[label=(\roman*)]
              \item $\mathbb{S}\models \mathrm{(SL1)}\ $  iff $\ \mathbb{S}^*\models\blacksquare (a \vee b)\leq \blacksquare(\blacksquare a \vee \blacksquare b) $.
    \item $\mathbb{S}\models \mathrm{(S9 \Leftarrow) }\ $  iff $\ \mathbb{S^*}\models \blacksquare a \vee \blacksquare b\leq \blacksquare (a \vee \blacksquare b)$.
 \item $\mathbb{S}\models \mathrm{(S9 \Rightarrow)}  \ $  iff $\ \mathbb{S^*}\models \blacksquare (a \vee \blacksquare b) \leq \blacksquare a \vee \blacksquare b$.
  \item $\mathbb{S}\models \mathrm{(DCT)} $ iff $\ \mathbb{S^*}\models \blacksquare (a \vee \blacksquare a) \leq \blacksquare a$.
  \end{enumerate}
\item  if $\mathbb{S}$ is  monotone and based on $(A, \neg)$ with    $A\models \forall a\forall b(\neg a\leq b \Leftrightarrow \neg b\leq a)$,
              \begin{enumerate}[label=(\roman*)]
              \item $\mathbb{S}\models \mathrm{(S6)}\ $ iff  $\ \mathbb{S^*}\models \neg \Diamond a = \blacksquare \neg a$. Thus, if $A\models \neg\neg a = a$, then  $\blacksquare a \coloneqq \neg \Diamond \neg a$.
              \end{enumerate}
              \item  if $\mathbb{S}$ is  monotone and based on $(A, \neg)$ with     $A\models \forall a\forall b(a\leq \neg b \Leftrightarrow  b\leq \neg a)$,
              \begin{enumerate}[label=(\roman*)]
              \item $\mathbb{S}\models \mathrm{(S6)}\ $ iff  $\ \mathbb{S^*}\models  \Diamond \neg a =  \neg\blacksquare  a$. Thus,  if $A\models \neg\neg a = a$, then  $\Diamond a \coloneqq \neg \blacksquare \neg a$.
              \end{enumerate}

\end{enumerate}
\end{proposition}

\begin{proof}
 In what follows, variables $a, b, c\ldots$ range in $A$.  Also, for the sake of a more concise and readable presentation, we will write conjunctions of  relational atoms as chains; for instance, $a\leq b \ \& \ b\prec c$ will be written as $a\leq b\prec c$.

1 (i)
\begin{center}
    \begin{tabular}{ccll}
          $\forall a(a\leq \Diamond a)$ &
       iff  & $\forall a   \forall b (\Diamond a \leq b \Rightarrow a \leq b )$ & Proposition \ref{prop:background can-ext}.1(i)\\
      & iff & $\forall a  \forall b(a \prec b \Rightarrow a \leq b )$ &  \\
      & iff &  $\prec \ \subseteq \ \leq $ &  \\
    \end{tabular}
\end{center}

(ii)
\begin{center}
    \begin{tabular}{ccll}
          $\forall a( \Diamond a\leq a)$ &
       iff  & $\forall a   \forall b (   a \leq b \Rightarrow \Diamond a \leq b )$ & Proposition \ref{prop:background can-ext}.1(i)\\

      & iff  & $ \forall a\forall b (a\leq b  \Rightarrow a \prec b )$ &  \\
       &iff &  $\leq \ \subseteq \ \prec $ &  \\
    \end{tabular}
\end{center}

(iii)
\begin{center}
    \begin{tabular}{cll}
         &$\forall a (\Diamond a \leq \Diamond \Diamond a)$ &  \\
     iff & $\forall a   \forall c (\Diamond \Diamond a \leq c\Rightarrow \Diamond a \leq c)$   & Proposition \ref{prop:background can-ext}.1(i)\\
     iff & $\forall a   \forall c (\exists b(  a\prec b \ \& \ \Diamond b\leq c) \Rightarrow \Diamond a \leq c)$   & Lemma \ref{lem: diamond-output equivalence extended}.1(iii)\\

    iff &  $\forall a   \forall b \forall c(a\prec b \prec c  \Rightarrow  a \prec c)$ & Lemma \ref{lem: diamond-output equivalence}.1(ii) \\
    \end{tabular}
\end{center}

(iv)
\begin{center}
    \begin{tabular}{cll}
         &$\forall a ( \Diamond \Diamond a\leq \Diamond a)$ &  \\
     iff & $\forall a   \forall c (\Diamond a \leq c\Rightarrow \Diamond \Diamond  a \leq c)$   & Proposition \ref{prop:background can-ext}.1(i)\\
     iff & $\forall a   \forall c ( \Diamond a \leq c\Rightarrow \exists b(b\prec c\ \& \ \Diamond  a \leq b)$   &  Lemma \ref{lem: diamond-output equivalence extended}.1(iii)\\

     iff & $\forall a   \forall c ( a \prec c\Rightarrow \exists b( a\prec b \ \&\   b \prec c))$   & Lemma \ref{lem: diamond-output equivalence}.1(ii)  \\
    \end{tabular}
\end{center}
 %

2(i)

\begin{center}
    \begin{tabular}{cll}
         &$\forall a (\Diamond a\leq \Diamond (a\wedge \Diamond a))$  &\\
    iff  & $\forall a  \forall c (\Diamond(a \wedge \Diamond a)\leq c \Rightarrow \Diamond a \leq c)$ & Proposition \ref{prop:background can-ext}.1(i)\\
    iff & $\forall a  \forall c(\exists d( a \wedge \Diamond a\leq d\prec c) \Rightarrow \Diamond a \leq c)$ & Lemma \ref{lem: diamond-output equivalence extended}.1(iii)\\
    iff & $\forall a  \forall c \forall d(a \wedge \Diamond a\leq d\prec c \Rightarrow \Diamond a \leq c)$ &  \\

      iff & $\forall a  \forall c\forall d(\exists b(\Diamond a\leq b \ \&\ a \wedge b\leq d \prec c) \Rightarrow \Diamond a \leq c)$  & Proposition \ref{prop: compactness and existential ackermann}.1(iii)\\
       iff & $\forall a \forall b  \forall c\forall d(a\prec b\ \&\ a \wedge b\leq d \prec c    \Rightarrow  a \prec c)$  & \\

    iff &$\forall a \forall b \forall c(a \prec b\ \&\ a\wedge b \prec c \Rightarrow a \prec c)$ &$(\ast\ast)$\\
    \end{tabular}
\end{center}
Let us show the equivalence marked with $(\ast\ast)$.  The top-to-bottom direction immediately obtains by letting $d:=a\wedge b\in A$, since $A$ is a $\wedge$-SL.
  Conversely, let $a, b, c, d\in A$ s.t.~$a\prec  b$ and $a\wedge b \leq d\prec c$;  by  (SI), the latter chain implies  $a\wedge b\prec  c$, hence by assumption we conclude $a\prec c$, as required.

(ii)
\begin{center}
    \begin{tabular}{cll}
         & $\forall a \forall b(\Diamond (\Diamond a \wedge \Diamond b) \leq \Diamond (a \wedge b))$ &\\
    iff & $\forall a \forall b  \forall c(\Diamond(a\wedge b)\leq c \Rightarrow \Diamond (\Diamond a\wedge \Diamond b) \leq c) $ & Proposition \ref{prop:background can-ext}.1(i)\\
     iff & $\forall a \forall b  \forall c(a\wedge b\prec c \Rightarrow \exists d(d\prec c \ \& \ \Diamond a\wedge \Diamond b \leq d) $ & Lemma \ref{lem: diamond-output equivalence extended}.1(iii)\\

    iff & $\forall a \forall b \forall c (a\wedge b\prec  c \Rightarrow \exists d(  d\prec c \ \&\ \exists e\exists f(a\prec e\ \&\
 b\prec f \ \&\ e\wedge f \leq d))) $ & Proposition \ref{prop: compactness and existential ackermann}.1(iii) \\
  iff & $\forall a \forall b \forall c (a\wedge b\prec  c \Rightarrow \exists d\exists e\exists f(a\prec e\ \&\
 b\prec f \ \&\ e\wedge f \leq d\prec c)) $ &  \\

    iff & $\forall a \forall b \forall c(a\wedge b\prec c \Rightarrow \exists e \exists f( ( a\prec e\ \&\ b\prec f \ \&\ e\wedge f \prec c))  $ & $(\ast\ast)$\\
    \end{tabular}
\end{center}
For the equivalence marked with $(\ast\ast)$, the bottom-to-top direction immediately obtains by letting $d: = e\wedge f\in A$. The converse direction follows from (SI).

3 (i)
\begin{center}
    \begin{tabular}{cll}
         & $\forall a \forall b(\blacksquare (a \vee  b) \leq \blacksquare (\blacksquare a \vee \blacksquare b))$ &\\
      iff &$\forall a \forall b \forall c (c\leq \blacksquare (a \vee b) \Rightarrow c\leq \blacksquare(\blacksquare a \vee \blacksquare b))$ & Proposition \ref{prop:background can-ext}.1(ii)\\

       iff & $\forall a \forall b \forall c   (c\prec  a \vee b\Rightarrow c\leq \blacksquare(\blacksquare a \vee \blacksquare b ))$  & Lemma \ref{lem: diamond-output equivalence}.1(iii) \\
      iff & $\forall a \forall b \forall c (c\prec a\vee b\Rightarrow \exists d(c\prec d \leq \blacksquare a \vee \blacksquare b))$  & Lemma \ref{lem: diamond-output equivalence extended}.2(iii) \\
    iff &$\forall a \forall b \forall c  (c\prec a\vee  b\Rightarrow \exists d\exists e \exists f (c\prec d \leq e\vee f\ \&\ e\leq \blacksquare a \ \&\ f\leq
    \blacksquare b))$ &Proposition \ref{prop: compactness and existential ackermann}.2(iii)\\
    iff &$\forall a \forall b \forall c  (c\prec a\vee  b\Rightarrow \exists e \exists f (c\prec  e\vee f\ \&\ e\leq \blacksquare a \ \&\ f\leq
    \blacksquare b))$ & $(\ast)$\\
    iff &$\forall a \forall b \forall c  (c\prec a\vee  b\Rightarrow \exists e \exists f (c\prec e\vee f\ \&\ e\prec a \ \&\ f\prec b))$ &Lemma \ref{lem: diamond-output equivalence}.1(iii)\\
    \end{tabular}
\end{center}
As to the equivalence marked with $(\ast)$, the bottom-to-top direction obtains by instantiating $d: = e\vee f\in A$, the converse direction follows from (WO).

(iii)
\begin{center}
    \begin{tabular}{cll}
         & $\forall a \forall b(\blacksquare (a \vee \blacksquare b) \leq \blacksquare a \vee \blacksquare b)$ &\\
      iff &$\forall a \forall b \forall c  (c\leq \blacksquare (a \vee \blacksquare b) \Rightarrow c\leq \blacksquare a \vee \blacksquare b)$ & Proposition \ref{prop:background can-ext}.1(ii)\\
      iff & $\forall a \forall b \forall c \forall d   (c\prec d \leq a \vee \blacksquare b\Rightarrow c\leq \blacksquare a \vee \blacksquare b )$  & Lemma \ref{lem: diamond-output equivalence extended}.2(iv) \\
      iff & $\forall a \forall b \forall c \forall d\forall e   (c\prec d\leq a\vee e \ \& \ e\prec   b$&Proposition \ref{prop: compactness and existential ackermann}.2(iii)\\
      & $\Rightarrow \exists f\exists g(c\leq f\vee g\ \& \ f\prec a\ \& \ g\prec b))$  &  \\

    iff &$\forall a \forall b \forall c \forall e (c\prec a\vee e\ \&\ e\prec b\Rightarrow \exists f\exists g (c\leq f\vee g \ \&\ f\prec a \ \&\ g\prec b))$ &$(\ast\ast)$\\
    \end{tabular}
\end{center}
Let us show the equivalence marked with $(\ast\ast)$. The top-to-bottom direction obtains by instantiating $d: = a\vee e\in A$. The converse direction follows from (WO). The proof of (ii) is similar.

(iv)
\begin{center}
    \begin{tabular}{cll}
         & $\forall a (\blacksquare (a \vee \blacksquare a) \leq \blacksquare a)$ &  \\
     iff & $\forall a \forall b (b \leq \blacksquare (a \vee \blacksquare a) \Rightarrow b\leq \blacksquare a) $   & Proposition \ref{prop:background can-ext}.1(ii)\\
    iff &  $\forall a \forall b\forall c ( b \prec c\leq a \vee \blacksquare a  \Rightarrow b\prec a)$ & Lemma \ref{lem: diamond-output equivalence}.1(iii)\\

    iff & $\forall a  \forall b \forall c(\exists d (b \prec c  \leq a \vee d\ \&\ d\leq \blacksquare a ) \Rightarrow b\prec a)$ & Proposition \ref{prop: compactness and existential ackermann}.2(iii)\\
    iff & $\forall a  \forall b \forall c(\exists d (b \prec c  \leq a \vee d\ \&\ d\prec a ) \Rightarrow b\prec a)$ & Lemma \ref{lem: diamond-output equivalence}.1(iii)\\
    iff & $\forall a  \forall b \forall d (b \prec a \vee d\ \&\ d\prec a  \Rightarrow b\prec a)$ & ($\ast \ast$)\\
    \end{tabular}
\end{center}
As to the equivalence marked with $(\ast\ast)$, the top-to-bottom direction obtains by instantiating $c: = a\vee d\in A$. The converse direction follows from (WO).

4 (i)
The proof of this item will make use of the fact that the assumption on $(A,\neg)$ implies, by Proposition \ref{prop:background can-ext}.2, and basic properties of adjoint maps between complete lattices (cf.~\cite{davey2002introduction}), that  $\neg^{\sigma}$ is completely meet-reversing. Hence, using this latter property in combination with denseness (in what follows, we will flag this out as `denseness +'), a term such as $\neg k$, for any $k\in K(A^\delta)$, can be equivalently rewritten as follows:
\[\neg k = \neg \bigwedge \{o\in O(A^\delta)\mid k\leq o\} = \bigvee \{\neg o\mid o\in O(A^\delta)\ \&\ k\leq o \}.\]

\begin{center}
    \begin{tabular}{cll}
         &$\forall a (\neg \Diamond a \leq \blacksquare \neg a)$ &  \\
     iff & $\forall a   \forall o (\Diamond a \leq o \Rightarrow \neg o \leq \blacksquare \neg a)$   & denseness +\\
     iff & $\forall a  \forall b \forall o (a \prec b\leq o \Rightarrow \neg o \leq \blacksquare \neg a)$   & Lemma \ref{lem: diamond-output equivalence extended}.1(iv) \\
     iff & $\forall a  \forall b  (a \prec b\Rightarrow \neg b \leq \blacksquare \neg a)$   & $(\ast)$ \\

    iff & $ \forall a \forall b(a\prec b \Rightarrow  \neg b \prec \neg a)$ & Lemma \ref{lem: diamond-output equivalence}.1(iii) \\
    \end{tabular}
\end{center}
For the equivalence marked with $(\ast)$, the top-to-bottom direction obtains by instantiating  $o: = b\in A\subseteq O(A^\delta)$, and the converse follows by the antitonicity of $\neg$.
The proofs of 5(i) is similar, using the fact that  the assumption on $(A,\neg)$ implies, by Proposition \ref{prop:background can-ext}.3, that  $\neg^{\sigma}$ is  completely join-reversing. 
 \end{proof}

\begin{remark} By Proposition \ref{prop:characteriz} 1(i) and 2(i),
for any monotone  proto-subordination algebra $\mathbb{S}$ which is $\wedge$-SL based, 
$\mathbb{S}\models \prec \ \subseteq \ \leq $ iff
 $\mathbb{S}^\ast\models  a\leq \Diamond a$, which implies $\mathbb{S}^\ast\models\Diamond a\leq \Diamond(a\wedge \Diamond a)$, i.e.~$\mathbb{S}\models \mathrm{(CT)}$. This observation generalizes
   \cite[Observation 15]{Makinson00}.
\end{remark}

\begin{proposition}
\label{prop:precontact correspondence}
For any    proto-precontact algebra $\mathbb{C} = (A, \pcon)$,
\begin{enumerate}
\item if $\mathbb{C}$ is antitone, then
\begin{enumerate}[label=(\roman*)]
   \item  $\mathbb{C}\models\; \npcon\ \subseteq\ \leq\ $ iff $\ \mathbb{C}^*\models\wt a\leq  a$.
\item $\mathbb{C}\models\;  \leq\ \subseteq\ \npcon\ $ iff $\ \mathbb{C}^*\models  a\leq \wt a$.
  \item  $\mathbb{C}\models\; \npcon\ \subseteq\ \geq\ $ iff $\ \mathbb{C}^*\models \bt a\leq  a$.
\item $\mathbb{C}\models\;  \geq\ \subseteq\ \npcon\ $ iff $\ \mathbb{C}^*\models  a\leq \bt a$.
\item $\mathbb{C}\models\;  \mathrm{(NS)}$ iff $\ \mathbb{C}^*\models  a\wedge \wt a\leq \bot$.
 \item $\mathbb{C}\models\;  \mathrm{(SFN)}$ iff $\ \mathbb{C}^*\models \top \leq  a\vee \wt a $.

\end{enumerate}
\item if $\mathbb{C}$ is a precontact algebra, then
\begin{enumerate}[label=(\roman*)]
\item      $\mathbb{C}\models\; \npcon^{-1}\circ \npcon \subseteq   \npcon \ $ iff $\ \mathbb{C}^*\models \wt a\leq  \wt \wt a\ $ iff     $\ \mathbb{C}^*\models \wt a\leq  \bt \wt a\ $ iff
      $\ \mathbb{C}\models\; \npcon^{-1}\circ \npcon \subseteq   \npcon^{-1}  $
     \item      $\mathbb{C}\models\; \npcon^{-1}\circ \npcon \subseteq  \npcon\ $ iff  $\ \mathbb{C}^*\models \bt a\leq  \bt \bt a\ $ iff $\ \mathbb{C}^*\models \bt a\leq  \wt \bt a\ $ iff $\ \mathbb{C}\models\; \npcon^{-1}\circ \npcon \subseteq  \npcon^{-1} $

      \item      $\mathbb{C}\models\; \npcon\circ \npcon \subseteq  \npcon\ $ iff  $\ \mathbb{C}^*\models \bt a\leq  \bt \wt a\ $ iff $\ \mathbb{C}^*\models \wt a\leq  \wt \bt a $ iff $\npcon^{-1} \circ \npcon^{-1} \subseteq  \npcon^{-1} $
      \item      $\mathbb{C}\models\; \npcon\circ \npcon \subseteq  \npcon^{-1} $ iff  $\ \mathbb{C}^*\models \wt a\leq  \bt \bt a\ $ iff $\ \mathbb{C}^*\models \bt a\leq  \wt \wt a $ iff $\npcon^{-1} \circ \npcon^{-1} \subseteq  \npcon$
      \end{enumerate}
      \item if $\mathbb{C}$ is a lattice-based precontact algebra, then
\begin{enumerate}[label=(\roman*)]
      \item $\mathbb{C}\models\  \npcon \subseteq\ \npcon^{-1}\ $  iff $\ \mathbb{C}\models\  \npcon^{-1} \subseteq\ \npcon\ $  iff $\ \mathbb{C}^*\models \wt a\leq  \bt a\ $ iff $\ \mathbb{C}^*\models  a\leq  \wt \wt a\ $ iff $\ \mathbb{C}^*\models \bt a\leq  \wt a\ $ iff $\ \mathbb{C}^*\models  a\leq  \bt\bt a$.

      \end{enumerate}
\item if $\mathbb{C}$ is a $\wedge$-SL based precontact algebra, then
\begin{enumerate}[label=(\roman*)]

\item  $\mathbb{C}\models\;  \mathrm{(CMO)} $ iff $\ \mathbb{C}^*\models  \wt a\leq \wt( a \wedge \wt a)$.
\end{enumerate}

\item if $\mathbb{C}$ is a DL-based precontact algebra, then
\begin{enumerate}[label=(\roman*)]
\item $\mathbb{C}\models\;  \mathrm{(ALT)}^\wt $ iff $\ \mathbb{C}^*\models \top \leq \wt a\vee \wt \wt a $.
\item $\mathbb{C}\models\;  \mathrm{(ALT)}^\bt $ iff $\ \mathbb{C}^*\models \top \leq \bt a\vee \bt \bt a $.

\end{enumerate}
\item if $\mathbb{C}$ is a Heyting algebra-based\footnote{A {\em Heyting algebra} is a distributive lattice $A$ endowed with a binary operation $\rightarrow$ s.t.~$a\wedge b\leq c$ iff $b\leq a\rightarrow c$ for all $a, b, c\in A$.} precontact algebra, then
\begin{enumerate}[label=(\roman*)]
\item $\mathbb{C}\models\;  \mathrm{(\vee\wedge)} $ iff $\ \mathbb{C}^*\models \wt(a\vee b)\wedge (b\rightarrow \wt a) \leq  a\rightarrow  \bt b $.
\end{enumerate}
\end{enumerate}
\end{proposition}

\begin{proof}
In what follows, variables $a, b, c\ldots$ range in $A$, and $k$ and $h$ range in $K(A^\delta)$.  Also, for the sake of a more concise and readable presentation, we will write conjunctions of  relational atoms as chains; for instance, $a\leq b \ \& \ b\npcon c$ will be written as $a\leq b\npcon c$.\\
1(i)
    \begin{center}
    \begin{tabular}{ccll}
          $\forall a( \wt a\leq a)$ &
       iff  & $\forall a \forall b  (b \leq \wt a\Rightarrow  b \leq a )$ & Proposition \ref{prop:background can-ext}.1(ii)\\
      & iff  & $\forall a \forall b   (a\npcon b\Rightarrow  b \leq a )$ & Lemma \ref{lem: diamond-output equivalence}.2(ii)\\
     &  iff  & $\npcon \ \subseteq \ \leq $ & \\
       \end{tabular}
       \end{center}
1(ii)
\begin{center}
    \begin{tabular}{ccll}
          $\forall a(  a\leq \wt a)$ &
       iff  & $\forall a \forall b   (b \leq  a\Rightarrow  b \leq \wt a )$ & Proposition
 \ref{prop:background can-ext}.1(ii)\\
      & iff  & $\forall a \forall b   (b \leq a\Rightarrow a\npcon b )$ & Lemma \ref{lem: diamond-output equivalence}.2(ii)\\
      & iff  & $\leq \ \subseteq \npcon $ & \\
       \end{tabular}
       \end{center}
 1(v)
\begin{center}
    \begin{tabular}{ccll}
          $\forall a(a\wedge \wt a\leq \bot)$ &
       iff  & $\forall a \forall b (b \leq a\wedge \wt a  \Rightarrow  b \leq \bot )$ & Proposition
 \ref{prop:background can-ext}.1(ii)\\
       &iff  & $\forall a \forall b ((b \leq a \ \&\ b\leq \wt a)  \Rightarrow  b \leq \bot )$ & \\
      & iff  & $ \forall b (b\leq \wt b  \Rightarrow  b \leq \bot )$ & $(\ast)$ \\

      & iff  & $ \forall b  (b \npcon b \Rightarrow  b \leq \bot)$ & Lemma \ref{lem: diamond-output equivalence}.2(ii)\\
       \end{tabular}
       \end{center}
 As to the equivalence marked with $(\ast)$, the top-to-bottom direction obtains by instantiating $a: = b$, and the converse one follows from the antitonicity of $\wt$.\\

2. The proofs of this item will make use of the fact that $\mathbb{C}$ being a precontact algebra implies that  $\wt$ on $\mathbb{C}^\ast$ is normal (cf.~Corollary \ref{cor: charact monotone reg norm}.2(iii)), i.e.~$\wt$ is {\em finitely} join-reversing; then, by basic properties of the canonical extensions of slanted algebras,  $\wt^{\pi}$ is {\em completely} join-reversing (cf.~\cite[Lemma 3.5]{de2020subordination}). Hence, using this latter property in combination with denseness (in what follows, we will flag this out as `denseness +'), a term such as $\wt o$, for any $o\in O(A^\delta)$, can be equivalently rewritten as follows:
\[\wt o = \wt \bigvee \{h\in K(A^\delta)\mid h\leq o\} = \bigwedge \{\wt h\mid h\in K(A^\delta)\ \&\ h\leq o \}\]
(i)
\begin{center}
    \begin{tabular}{cll}
         & $\forall a( \wt a\leq \wt \wt a)$ & \\
       iff  & $\forall a \forall k  \forall h (k \leq \wt a\ \&\ h \leq \wt a \Rightarrow  k \leq \wt h )$ & denseness +\\
       iff  & $\forall a \forall k  \forall h (\exists b  (a\npcon b\ \&\ k \leq b)\ \&\ \exists c(a\npcon c\ \&\ h\leq c) \Rightarrow  k \leq \wt h )$ & Lemma \ref{lem: diamond-output equivalence extended}.3(iv)\\
       iff  & $\forall a \forall b \forall c \forall k \forall h ((a\npcon b\ \&\ k \leq b\ \&\ a\npcon c\ \&\ h\leq c) \Rightarrow  k \leq \wt h )$ & \\
       iff  & $\forall a \forall b  \forall c (a\npcon b\ \&\ a\npcon c \Rightarrow  b \leq \wt c)$ & $(\ast)$\\
       iff  & $\forall a \forall b  \forall c (a\npcon b\ \&\ a\npcon c \Rightarrow  c\npcon b)$ & Lemma \ref{lem: diamond-output equivalence}.2(ii)\\
       iff & $ \npcon^{-1}\circ \npcon \subseteq   \npcon   $    &\\
       \end{tabular}
       \end{center}
Let us show the equivalence marked with $(\ast)$. The top-to-bottom direction obtains by instantiating $k: = b\in A\subseteq K(A^\delta)$ and $h: = c \in A\subseteq K(A^\delta)$; conversely, let $a, b, c\in A$ and $h, k\in K(A^\delta)$ s.t.~$a\npcon b\geq k$ and $a\npcon c\geq h$. Then  the assumption and the antitonicity of $\wt$ imply that $k\leq b\leq \wt c\leq \wt h$, as required.
\\

(ii)
\begin{center}
    \begin{tabular}{cll}
         & $\forall a( \bt a\leq \bt \bt a)$ & \\
       iff  & $\forall a \forall k  \forall h (k \leq \bt a\ \&\ h \leq \bt a \Rightarrow  k \leq \bt h )$ & denseness +\\
       iff  & $\forall a \forall k  \forall h (\exists b  (b\npcon a\ \&\ k \leq b)\ \&\ \exists c(c\npcon a\ \&\ h\leq c) \Rightarrow  k \leq \bt h )$ & Lemma \ref{lem: diamond-output equivalence extended}.4(iv)\\
       iff  & $\forall a \forall b \forall c \forall k \forall h ((a\npcon b\ \&\ k \leq b\ \&\ a\npcon c\ \&\ h\leq c) \Rightarrow  k \leq \bt h )$ & \\
       iff  & $\forall a \forall b  \forall c (a\npcon b\ \&\ a\npcon c \Rightarrow  b \leq \bt c)$ & $(\ast)$\\
       iff  & $\forall a \forall b  \forall c (a\npcon b\ \&\ a\npcon c \Rightarrow  b\npcon c)$ & Lemma \ref{lem: diamond-output equivalence}.2(ii)\\
       iff & $ \npcon^{-1}\circ \npcon \subseteq   \npcon   $    &\\
       \end{tabular}
       \end{center}
3(i) This proof makes use of the fact that $\mathbb{C}$ being a lattice-based precontact algebra implies that $\mathbb{C}^\ast$ is not only normal but also {\em tense} (cf.~Corollary \ref{cor: defined+si+wo imply tense hence normal}); therefore, $\wt^\pi$ and $\bt^\pi$ form an adjoint pair (cf.~Lemma \ref{lemma:tense and lifted adjunction}).
\begin{center}
    \begin{tabular}{ccll}
           $\forall a( a\leq \wt \wt a)$ &
     iff    & $\forall a( \wt a\leq \bt a)$ & adjunction\\
       &iff  & $\forall a \forall b   (b \leq \wt a\Rightarrow  b \leq \bt a )$ & Proposition \ref{prop:background can-ext}.1(ii)\\
       &iff  & $\forall a \forall b   (a\npcon b\Rightarrow  b \npcon a )$ & Lemma \ref{lem: diamond-output equivalence}.2(ii)(iii)  \\
       &iff &  $  \npcon \subseteq\ \npcon^{-1}\ $   &\\
      & iff &  $  \npcon^{-1} \subseteq\ (\npcon^{-1})^{-1} = \npcon $   &\\
       &iff  & $\forall a \forall b   (a \leq \bt b\Rightarrow  a \leq \wt b )$ & \\

      & iff  & $ \forall b   (\bt b\leq \wt b )$ & Proposition \ref{prop:background can-ext}.1(ii) \\
     & iff  & $ \forall b   ( b\leq \bt\bt b )$ & adjunction \\
       \end{tabular}
       \end{center}
4(i)
\begin{center}
    \begin{tabular}{cll}
         & $\forall a( \wt a\leq \wt( a \wedge \wt a))$ & \\
       iff  & $\forall a \forall k  \forall h(k \leq \wt a\ \&\ h \leq a\wedge \wt a  \Rightarrow  k \leq \wt h )$ & denseness +\\
       iff  & $\forall a \forall k  \forall h((k \leq \wt a\ \&\ h \leq a\ \&\ h\leq \wt a)  \Rightarrow  k \leq \wt h )$ & \\
       iff  & $\forall a \forall k  \forall h ((\exists b (a\npcon b\ \&\ k \leq b) \ \&\ h\leq a\ \&\ \exists c (a\npcon c\ \&\ h\leq c) )\Rightarrow  k \leq \wt h )$ & Lemma \ref{lem: diamond-output equivalence extended}.3(iv)\\
       iff  & $\forall a\forall b\forall c \forall k  \forall h ((a\npcon b\ \&\ k \leq b \ \&\ h\leq a\wedge c\ \&\  a\npcon c) \Rightarrow  k \leq \wt h )$ & \\
       iff  & $\forall a \forall b \forall c  (a\npcon b\ \&\ a\npcon c \Rightarrow  b \leq \wt (a\wedge c ))$ & $(\ast)$ \\
       iff  & $\forall a \forall b  \forall c (a\npcon b\ \&\ a\npcon c \Rightarrow   (a\wedge c)\npcon b)$ & Lemma \ref{lem: diamond-output equivalence}.2(ii)\\
       \end{tabular}
       \end{center}
Let us show the equivalence marked with $(\ast)$. The top-to-bottom direction obtains by instantiating $k: = b\in A\subseteq K(A^\delta)$ and $h: = a\wedge c \in A \subseteq K(A^\delta)$; conversely, let $a, b, c\in A$ and $h, k\in K(A^\delta)$ s.t.~$a\npcon b\geq k$ and $a\npcon c $ and $a\wedge c\geq h$. Then  the assumption and the antitonicity of $\wt$ imply that $k\leq b\leq \wt (a\wedge c)\leq \wt h$, as required.\\

5. The proofs of this item  make use of the fact that the precontact algebra  $\mathbb{C}$ is based on a distributive lattice $A$, and hence $A^\delta$ is {\em completely distributive}.  Hence, using this latter property in combination with denseness and, as discussed above, $\wt^{\pi}$ being completely join-reversing (in what follows, we will flag this out as `denseness ++'), a term such as $\wt a\vee \wt o$, for any $a\in A$ and  $o\in O(A^\delta)$, can be equivalently rewritten as follows:
\begin{center}
    \begin{tabular}{cc}
        $\wt a \vee \wt o$ &= $ \wt a \vee \wt \bigvee \{h\in K(A^\delta)\mid h\leq o\}$\\
         & $ = \wt a \vee \bigwedge \{\wt h\mid h\in K(A^\delta)\ \&\ h\leq o \}$\\
         & $ =\bigwedge \{\wt a \vee \wt h\mid h\in K(A^\delta)\ \&\ h\leq o \}$
    \end{tabular}
\end{center}

(i)
\begin{center}
    \begin{tabular}{cll}
         & $\forall a( \top \leq \wt a\vee \wt \wt a)$ & \\
       iff  & $\forall a \forall k  (k\leq \wt a \Rightarrow \top \leq \wt a \vee \wt k )$ & denseness ++\\
       iff  & $\forall a \forall k (\exists b (k\leq b \ \&\ a\npcon b) \Rightarrow \top \leq \wt a \vee \wt k )$ & Lemma \ref{lem: diamond-output equivalence extended}.3(iv)\\
      iff  & $\forall a \forall b  (a\npcon b \Rightarrow \top \leq \wt a \vee \wt b )$ & $(\ast)$\\
      iff  & $\forall a \forall b  (a\npcon b \Rightarrow \exists c\exists d(\top \leq c \vee d \ \&\ c\leq \wt a \ \&\ d\leq \wt b)$ & Proposition \ref{prop: compactness and existential ackermann}.2(iii)\\
      iff  & $\forall a \forall b  (a\npcon b \Rightarrow \exists c\exists d(\top \leq c \vee d \ \&\ a\npcon c \ \&\ b\npcon d)$ & Lemma \ref{lem: diamond-output equivalence}.2(ii)\\
       \end{tabular}
       \end{center}
(ii)
\begin{center}
    \begin{tabular}{cll}
         & $\forall a( \top \leq \bt a\vee \bt \bt a)$ & \\
       iff  & $\forall a \forall k  (k\leq \bt a \Rightarrow \top \leq \bt a \vee \bt k )$ & denseness ++\\
       iff  & $\forall a \forall k (\exists b (k\leq b \ \&\ b\npcon a) \Rightarrow \top \leq \bt a \vee \bt k )$ & Lemma \ref{lem: diamond-output equivalence extended}.4(iv)\\
      iff  & $\forall a \forall b  (b\npcon a \Rightarrow \top \leq \bt a \vee \bt b )$ & $(\ast)$\\
      iff  & $\forall a \forall b  (b\npcon a \Rightarrow \exists c\exists d(\top \leq c \vee d \ \&\ c\leq \bt a \ \&\ d\leq \bt b)$ & Proposition \ref{prop: compactness and existential ackermann}.2(iii)\\
      iff  & $\forall a \forall b  (b\npcon a \Rightarrow \exists c\exists d(\top \leq c \vee d \ \&\ c\npcon a \ \&\ d\npcon b)$ & Lemma \ref{lem: diamond-output equivalence}.2(ii)\\
       \end{tabular}
       \end{center}
6(i)
\begin{center}
    \begin{tabular}{cll}
         & $\forall a\forall b( \wt(a\vee b)\wedge (b\rightarrow \wt a) \leq  a\rightarrow \bt b)$ & \\
         iff & $\forall a\forall b \forall c(c\leq  \wt(a\vee b)\wedge (b\rightarrow \wt a) \Rightarrow x\leq  a\rightarrow \bt b)$ & Proposition \ref{prop:background can-ext}.1(ii)\\
         iff & $\forall a\forall b \forall c(c\leq  \wt(a\vee b) \ \&\  c\leq  b\rightarrow \wt a \Rightarrow c\leq  a\rightarrow \bt b)$\\
         iff & $\forall a\forall b \forall c(c\leq  \wt(a\vee b) \ \&\  b\wedge c\leq  \wt a \Rightarrow c\wedge a\leq   \bt b)$ & Heyting algebra\\
         iff & $\forall a\forall b \forall c((a\vee b)\npcon c \ \&\  a\npcon (b\wedge c)   \Rightarrow (c\wedge a)\npcon    b)$ & Lemma \ref{lem: diamond-output equivalence}.2(ii)\\
         \end{tabular}
       \end{center}
\end{proof}
The proof of the following proposition proceeds similarly to that of the proposition above, and is omitted.
\begin{proposition}
\label{prop:postcontact correspondence}
For any    dual proto-precontact algebra $\mathbb{D} = (A, \ppcon)$,
\begin{enumerate}
\item if $\mathbb{D}$ is antitone, then
\begin{enumerate}[label=(\roman*)]
   \item  $\mathbb{D}\models\; \leq\ \subseteq\ \nppcon\ $ iff $\ \mathbb{D}^*\models\tw a\leq  a$.
\item $\mathbb{D}\models\;  \nppcon\ \subseteq\ \leq\ $ iff $\ \mathbb{D}^*\models  a\leq \tw a$.
  \item  $\mathbb{D}\models\; \geq\ \subseteq\ \nppcon\ $ iff $\ \mathbb{D}^*\models \tb a\leq  a$.
\item $\mathbb{D}\models\;  \nppcon\ \subseteq\ \geq\ $ iff $\ \mathbb{D}^*\models  a\leq \tb a$.
\item $\mathbb{D}\models\;  \mathrm{(SFN)^\tw}$ iff $\ \mathbb{D}^*\models  a\wedge \tw a\leq \bot$.
 \item $\mathbb{D}\models\;  \mathrm{(SR)}$ iff $\ \mathbb{D}^*\models \top \leq  a\vee \tw a $.

\end{enumerate}
\item if $\mathbb{D}$ is a dual precontact algebra, then
\begin{enumerate}[label=(\roman*)]
\item      $\mathbb{D}\models\; \nppcon\circ \nppcon^{-1} \subseteq   \nppcon \ $ iff $\ \mathbb{D}^*\models \tw \tw a\leq  \tw a\ $ iff     $\ \mathbb{D}^*\models \tb \tw a\leq  \tw a\ $ iff
      $\ \mathbb{D}\models\; \nppcon\circ \nppcon^{-1} \subseteq   \nppcon^{-1}  $
     \item      $\mathbb{D}\models\; \nppcon^{-1}\circ \nppcon \subseteq  \nppcon\ $ iff  $\ \mathbb{D}^*\models \tw \tb a\leq \tb a\ $ iff $\ \mathbb{D}^*\models \tb \tb a\leq  \tb a\ $ iff $\ \mathbb{D}\models\; \nppcon^{-1}\circ \nppcon \subseteq  \nppcon^{-1} $
      \item      $\mathbb{D}\models\; \nppcon\circ \nppcon \subseteq  \nppcon\ $ iff  $\ \mathbb{D}^*\models \tw \tb a\leq  \tw a\ $ iff $\ \mathbb{D}^*\models \tb \tw a\leq  \tb a $ iff $\nppcon^{-1} \circ \nppcon^{-1} \subseteq  \nppcon^{-1} $
      \item      $\mathbb{D}\models\; \nppcon\circ \nppcon \subseteq  \nppcon^{-1} $ iff  $\ \mathbb{D}^*\models \tw \tw a\leq   \tb a\ $ iff $\ \mathbb{D}^*\models \tb \tb a\leq  \tw a $ iff $\nppcon^{-1} \circ \nppcon^{-1} \subseteq  \nppcon$

\item $\mathbb{D}\models\  \nppcon \subseteq\ \nppcon^{-1}\ $  iff $\ \mathbb{D}\models\  \nppcon^{-1} \subseteq\ \nppcon\ $  iff $\ \mathbb{D}^*\models \tw a\leq  \tb a\ $ iff $\ \mathbb{D}^*\models  \tb \tb a\leq   a\ $ iff $\ \mathbb{D}^*\models \tb a\leq  \tw a\ $ iff $\ \mathbb{D}^*\models \tw \tw a\leq   a$.
      \end{enumerate}

\item if $\mathbb{D}$ is a $\vee$-SL based dual precontact algebra, then
\begin{enumerate}[label=(\roman*)]

\item  $\mathbb{D}\models\; \mathrm{(CMO)}^\tw $ iff $\ \mathbb{D}^*\models  \tw (a \vee \tw a) \leq \tw a$.

\end{enumerate}

\item if $\mathbb{D}$ is a DL-based dual precontact algebra, then
\begin{enumerate}[label=(\roman*)]

\item $\mathbb{D}\models\;  \mathrm{(ALT)^\tw} $ iff $\ \mathbb{D}^*\models \tw \tw a\wedge \tw a\leq \bot $.
\item $\mathbb{D}\models\;  \mathrm{(ALT)^\tb} $ iff $\ \mathbb{D}^*\models \tb \tb a\wedge \tb a \leq \bot $.
\end{enumerate}
\end{enumerate}
\end{proposition}

We conclude this section by showing that the modal characterization results above  naturally extend to relational algebras endowed with more than one relation. Two such examples arise when considering the interaction between obligation and permission in input/output logic and also  in the context of the study of subordination and precontact algebras.

A {\em bi-subordination  algebra} \cite{celani2021bounded} is a structure $\mathbb{B} = (A, \prec_1, \prec_2)$ such that $(A, \prec_i)$ is a subordination algebra for $1\leq i\leq 2$. The following conditions on DL-based bi-subordination algebras have been considered in the literature (see also Section \ref{ssec:dunn-saxioms} for an expanded discussion on these conditions):
\begin{center}
\begin{tabular}{rlrl}
 (P1) & $ c \prec_1 a \vee b \Rightarrow \forall d [a \prec_2 d \Rightarrow \exists e (e \prec_1 b\ \&\ c \leq  e \vee d)]$\\
(P2) & $  a \wedge b \prec_2 c \Rightarrow \forall d[d \prec_1 a\Rightarrow \exists e (b \prec_2 e\ \&\ d \wedge e \leq c)]$.
\end{tabular}
\end{center}
A {\em (proto-) subordination precontact algebra}, abbreviated as  (proto) sp-algebra, is a structure $\mathbb{H} = (A, \prec, \pcon)$ such that $(A, \prec)$ is a (proto-)subordination algebra, and $(A, \pcon)$ is a (proto-)precontact  algebra\footnote{\label{footn:sdp}Likewise, we can define {\em (proto) sdp-algebras} as structures $\mathbb{K} = (A, \prec, \ppcon)$ such that $(A, \prec)$ is a (proto-)subordination algebra, and $(A, \ppcon)$ is a dual (proto-)precontact  algebra}. We will consider the following conditions on sp-algebras, some of which are the algebraic counterparts of  well known conditions on the relationship between norms and permissions:\footnote{Conditions (CS1) and (CS2) crop up in the literature on subordination algebras as the defining conditions of  pre-contact subordination lattices (cf.~\cite[Definition 3.2.1]{de2020subordination}). Interestingly,  (CS2) is the algebraic counterpart of a closure rule known as (AND)$^{\downarrow}$ in the input/output logic literature. }
\begin{center}
\begin{tabular}{rlrl}
(CS1) & $a \prec x\vee y\ \&\ a \npcon x \Rightarrow a \prec   y$ &
(CS2) & $a \npcon x\wedge y\ \&\ a \prec x \Rightarrow a  \npcon   y$.\\
(CT)$^\downarrow$ & $a\prec b\ \&\ a\npcon b\wedge c \Rightarrow a\wedge b\npcon c$ & (CT)$^\wt$ & $a\prec b\ \&\ a\wedge b\npcon c\Rightarrow a\npcon c$.
\\

(INC) & $a\prec b\Rightarrow a\npcon b$\\
\end{tabular}
\end{center}
Slanted algebras of the appropriate modal similarity type can be associated both with bi-subordination algebras and with sp-algebras  in the obvious way. The modal axioms characterizing conditions (P1) and (P2) in item 3 of the next proposition have been used  in \cite{dunn1995positive} to completely axiomatize the positive (i.e.~negation-free) fragment of basic classical normal modal logic.

\begin{proposition}
\label{prop:poly correspondence}
For any   bi-subordination algebra $\mathbb{B} = (A, \prec_1, \prec_2)$,
\begin{enumerate}
\item  $\mathbb{B}\models \prec_1\ \subseteq \prec_2$ iff
 $\mathbb{B}^\ast\models \Diamond_2 a\leq \Diamond_1 a$.
\item  $\mathbb{B}\models \prec_2\circ \prec_1\ \subseteq \ \leq$ iff $\mathbb{B}^\ast\models \blacksquare_2 a\leq \Diamond_1 a$.
\item  if $\mathbb{B}$ is DL-based, then
\begin{enumerate}[label=(\roman*)]
\item   $\mathbb{B}\models \mathrm{(P1)}$ iff $\mathbb{B}^\ast\models \blacksquare_1 (a\vee b)\leq \Diamond_2 a\vee \blacksquare_1 b $.

\item $\mathbb{B}\models \mathrm{(P2)}$ iff $\mathbb{B}^\ast\models \blacksquare_1 a\wedge \Diamond_2 b \leq \Diamond_2 (a\wedge b)$.
\end{enumerate}
\item $\mathbb{B}\models \prec_2  \circ\prec_1\ \subseteq\ \prec_1  \circ\prec_2$ iff $\mathbb{B}^\ast\models \Diamond_1\blacksquare_2  a\leq \blacksquare_2\Diamond_1 a $.
\item $\mathbb{B}\models \prec_1 \circ \prec_2 \ \subseteq\ \prec_1 $ iff $\mathbb{B}^\ast\models \Diamond_1 a \leq \Diamond_2\Diamond_1 a $.
\item $\mathbb{B}\models \forall a\forall b ( a \prec_1 b \ \& \ a \prec_2 b  \Rightarrow b =\top) $ iff $\mathbb{B}^\ast\models \Diamond_1 a  \vee \Diamond_2 a = \top $.
\item $\mathbb{B}\models  \forall b (\forall a (a \prec_1 b) \Rightarrow \forall a (a \prec_2 b)) $ iff $\mathbb{B}^\ast\models \Diamond_2 \top \leq \Diamond_1 \top$.
\item $\mathbb{B} \models \forall a (\exists b ( a \prec_1 b\ \&\ b \prec_2 \bot) \Rightarrow a = \bot)$ iff $\mathbb{B}^\ast \models \blacksquare_1 \blacksquare_2 \bot= \bot$.
\end{enumerate}
\end{proposition}
\begin{proof}
In what follows, variables $a, b, c\ldots$ range in $A$, $k$ and  $o$ in $K(A^\delta)$ and $O(A^\delta)$, respectively.  
In the following calculations, we omit the references to the various properties, since the justifications are similar to those given in the proof of the propositions above.

\begin{center}
    \begin{tabular}{cll}
  $\forall a(\Diamond_2 a\leq \Diamond_1 a )$ &
 iff & $\forall a\forall b(\Diamond_1 a\leq b\Rightarrow  \Diamond_2 a \leq b )$\\
 &iff & $\forall a\forall b( a\prec_1 b\Rightarrow   a \prec_2 b )$\\
 &iff & $\prec_1 \ \subseteq  \   \prec_2 $\\
 \end{tabular}
    \end{center}

\begin{center}
    \begin{tabular}{cll}
  $\forall a(\blacksquare_2 a\leq \Diamond_1 a )$ &
 iff & $\forall a\forall k\forall o[(k\leq \blacksquare_2 a\ \&\ \Diamond_1 a\leq o)\Rightarrow  k \leq o ]$\\
 &iff & $\forall a\forall b\forall c\forall k\forall o[(k\leq b\ \&\ b\prec_2 a\ \&\  a\prec_1 c\  \&\ c\leq o)\Rightarrow  k \leq o ]$\\
& iff & $\forall b\forall c[\exists a (b\prec_2 a\ \&\  a\prec_1 c)\Rightarrow  b \leq c ]$\\
 &iff & $\prec_2\circ\prec_1\ \subseteq\ \leq $
 \end{tabular}
    \end{center}

\begin{center}
    \begin{tabular}{cll}
         $\forall a\forall b(\blacksquare_1 (a\vee b)\leq \Diamond_2 a\vee \blacksquare_1 b )$ &
         iff & $\forall a\forall b\forall k\forall o[(k\leq \blacksquare_1 (a\vee b)\ \& \ \Diamond_2 a\leq o)\Rightarrow k\leq o\vee \blacksquare_1 b ]$\\
         &iff & $\forall a\forall b\forall c\forall d\forall k\forall o[(k\leq c\ \& \ c\prec_1 (a\vee b)\ \& \  a\prec_2 d\ \&\ d\leq o)$\\
         & & $\Rightarrow k\leq o\vee \blacksquare_1 b ]$\\
         &iff & $\forall a\forall b\forall c\forall d[( c\prec_1 (a\vee b)\ \& \  a\prec_2 d)\Rightarrow c\leq d\vee \blacksquare_1 b ]$\\
         &iff & $\forall a\forall b\forall c[ c\prec_1 (a\vee b) \Rightarrow \forall d( a\prec_2 d\Rightarrow c\leq d\vee \blacksquare_1 b )]$\\
         &iff & $\forall a\forall b\forall c[ c\prec_1 (a\vee b) \Rightarrow \forall d[ a\prec_2 d\Rightarrow \exists e(e\prec_1 b \ \&\ c\leq d\vee e )]]$\\
\end{tabular}
         \end{center}

\begin{center}
    \begin{tabular}{cll}
          $\forall a\forall b(\blacksquare_1 a\wedge\Diamond_2  b\leq \Diamond_2 (a\wedge b))$ &
         iff & $\forall a\forall b\forall k\forall o[(k\leq \blacksquare_1 a \ \&\ \Diamond_2 (a\wedge b)\leq o)\Rightarrow k\wedge \Diamond_2 b \leq o]$\\
         &iff & $\forall a\forall b\forall c \forall d\forall k\forall o[(k\leq d\ \&\  d\prec_1 a \ \&\  a\wedge b\prec_2 c \ \&\ c\leq o)$\\
         &&$\Rightarrow  k\wedge \Diamond_2 b \leq o]$\\
         &iff & $\forall a\forall b\forall c \forall d[( d\prec_1 a \ \&\  a\wedge b\prec_2 c)\Rightarrow d\wedge \Diamond_2 b \leq c]$\\
         &iff & $\forall a\forall b\forall c [ a\wedge b\prec_2 c\Rightarrow \forall d[d\prec_1 a \Rightarrow \exists e (b\prec_2 e \ \& \ d\wedge e \leq c)]]$\\
         \end{tabular}
         \end{center}

         \begin{center}
    \begin{tabular}{cll}
  $\forall a(\Diamond_1\blacksquare_2  a\leq \blacksquare_2\Diamond_1 a )$ &
 iff & $\forall a\forall k\forall o[(k\leq \blacksquare_2 a \ \&\ \Diamond_1 a\leq o)\Rightarrow  \Diamond_1 k\leq  \blacksquare_2 o ]$ \\
  &iff & $\forall a\forall b\forall c\forall k\forall o[(k\leq b\prec_2 a \ \&\  a \prec_1 c\leq o)\Rightarrow  \Diamond_1 k\leq  \blacksquare_2 o ]$\\
  &iff & $\forall a\forall b\forall c\forall k\forall o[(b\prec_2 a \ \&\  a \prec_1 c)\Rightarrow  \Diamond_1 b\leq  \blacksquare_2 c ]$\\
  &iff & $\forall b\forall c[\exists a(b\prec_2 a \ \&\  a \prec_1 c)\Rightarrow \exists e (  b\prec_1 e\ \&\  e\prec_2 c )]$\\
  &iff & $\prec_2  \circ\prec_1\ \subseteq \prec_1  \circ\prec_2$\\
\end{tabular}
    \end{center}

\begin{center}
    \begin{tabular}{cll}
         $\forall a (\Diamond_1 a \leq \Diamond_2\Diamond_1 a)$  &
         iff & $\forall a \forall c (\Diamond_2\Diamond_1 a  \leq c \Rightarrow \Diamond_1 a\leq c) $ \\
        & iff & $\forall a \forall c (\exists b ( a \prec_1 b \ \& \ \Diamond_2 b \leq c) \Rightarrow \Diamond_1 a \leq c) $\\
         &iff  & $\forall a  \forall c (\exists b( a \prec_1 b \ \& \ b \prec_2 c) \Rightarrow a \prec_1 c) $\\
         & iff & $\prec_1\circ \prec_2\ \subseteq \ \prec_1$
    \end{tabular}
\end{center}

\begin{center}
    \begin{tabular}{cll}
     $\forall a(\top \leq \Diamond_1 a \vee \Diamond_2 a)$ &
    iff& $\forall a\forall b (\Diamond_1 a \vee \Diamond_2 a \leq b \Rightarrow \top \leq b) $ \\
    &iff & $\forall a\forall b (\Diamond_1 a \leq b \ \& \ \Diamond_2 a \leq b \Rightarrow \top \leq b) $\\
   & iff &  $\forall a\forall b ( a \prec_1 b \ \& \  a \prec_2 b \Rightarrow  b=\top) $\\
     \end{tabular}
\end{center}

\begin{center}
    \begin{tabular}{cll}
          $\Diamond_2 \top \leq \Diamond_1 \top$  &
         iff & $\forall b(\Diamond_1\top \leq b \Rightarrow \Diamond_2 \top \leq b)$ \\
         &iff & $\forall b(\top \prec_1  b \Rightarrow \top \prec_2  b )$ \\
         & iff & $\forall b(\forall a (a \prec_1  b) \Rightarrow \forall a (a \prec_2  b ))$ \\
    \end{tabular}
\end{center}

\begin{center}
    \begin{tabular}{cll}
          $\blacksquare_1\blacksquare_2 \bot = \bot$  &
         iff & $\forall a ( a \leq \blacksquare_1\blacksquare_2 \bot \Rightarrow a \leq \bot) $\\
         &iff &  $\forall a ( \Diamond_2\Diamond_1 a \leq  \bot \Rightarrow a \leq \bot) $\\
         &iff &  $\forall a ( \exists b (a \prec_1 b \ \& \ b \prec_2 \bot) \leq  \bot \Rightarrow a \leq \bot) $
    \end{tabular}
\end{center}
\end{proof}

\begin{proposition}
\label{prop:poly correspondence permission}
For any   sp-algebra $\mathbb{H} = (A, \prec, \pcon)$ and any sdp-algebra $\mathbb{K} = (A, \prec, \ppcon)$,
\begin{enumerate}
\item $\mathbb{H} \models \mathrm{(INC)}$ iff
 $\mathbb{H}^\ast\models \blacksquare
 a\leq \wt  a$.
\item if $\mathbb{H}$ is $\wedge$-SL-based, then
\begin{enumerate}[label=(\roman*)]
\item $\mathbb{H} \models \mathrm{(CT)}^\wt$ iff
 $\mathbb{H}^\ast\models \wt(a\wedge \Diamond a)\leq \wt a$.
 \end{enumerate}
\item if $\mathbb{H}$ is Heyting algebra-based, then
\begin{enumerate}[label=(\roman*)]
\item $\mathbb{H} \models \mathrm{(CS2)}$ iff $\mathbb{H}^\ast\models \Diamond a\rightarrow \wt a\leq \wt a$
\item $\mathbb{H} \models \mathrm{(CT)}^\downarrow$ iff $\mathbb{H}^\ast\models\blacksquare b\wedge \bt(b\wedge c)\leq b\rightarrow \bt c$.
\end{enumerate}
\item if $\mathbb{H}$ is co-Heyting algebra-based\footnote{A {\em co-Heyting algebra} is a distributive lattice $A$ endowed with a binary operation $\pdla$ s.t.~$a\pdla b\leq c$ iff $ a\leq  c\vee b$ for all $a, b, c\in A$.}, then
\begin{enumerate}[label=(\roman*)]
\item $\mathbb{H} \models \mathrm{(CS1)}$ iff $\mathbb{H}^\ast\models \Diamond a \leq \Diamond a\pdla \wt a$.
\end{enumerate}
\item if $\mathbb{K}$ is co-Heyting algebra-based, then
\begin{enumerate}[label=(\roman*)]
\item $\mathbb{K} \models \mathrm{(CT)^\tw}$ iff $\mathbb{K}^\ast\models \tw a \leq \Diamond(\tw a\pdla  a)$.
\end{enumerate}
\end{enumerate}
\end{proposition}
\begin{proof}
$\quad$
1. \begin{center}
    \begin{tabular}{cll}
    $\forall a (\blacksquare a\leq \bt a)$ & iff &$\forall a\forall b(b\leq \blacksquare a\Rightarrow b\leq \bt a )$\\
    & iff &$\forall a\forall b(b\prec   a\Rightarrow b\npcon a )$\\
  \end{tabular}
\end{center}
2(i) \begin{center}
    \begin{tabular}{cll}
$\forall a(\wt(a\wedge \Diamond a)\leq \wt a)$ & iff & $\forall a \forall c (c\leq \wt(a\wedge \Diamond a)\Rightarrow c\leq \wt a )$\\

& iff & $\forall a \forall c (\exists b ( \Diamond a \leq b \ \& \ c\leq \wt(a\wedge b))\Rightarrow c\leq \wt a )$\\
& iff & $\forall a \forall b \forall c ( a \prec b \ \& \ (a\wedge b)\npcon c\Rightarrow a \npcon c )$\\
   \end{tabular}
\end{center}

3(i) \begin{center}
    \begin{tabular}{cll}
    $\forall a (\Diamond a\rightarrow \wt a\leq  \wt a)$ & iff & $\forall  a \forall c(c\leq \Diamond a\rightarrow \wt a\Rightarrow c\leq \wt a)$\\
    & iff & $\forall  a \forall c(\Diamond a \wedge c\leq \wt a \Rightarrow c\leq \wt a)$\\
     & iff & $\forall  a \forall c(\exists b(\Diamond a\leq b \ \& \ b \wedge c\leq \wt a) \Rightarrow c\leq \wt a)$\\
     & iff & $\forall  a \forall b\forall c(\Diamond a\leq b \ \& \ b \wedge c\leq \wt a \Rightarrow c\leq \wt a)$\\
     & iff & $\forall  a \forall b\forall c(a\prec b \ \&\ a \npcon  b \wedge c \Rightarrow a\npcon c)$\\
   \end{tabular}
\end{center}

(ii)
\begin{center}
    \begin{tabular}{cll}
    $\forall b\forall c(\blacksquare b\wedge \bt(b\wedge c)\leq b\rightarrow \bt c)$ & iff & $\forall a\forall b\forall c(a\leq \blacksquare b\wedge \bt(b\wedge c)\Rightarrow a\leq b\rightarrow \bt c)$\\
    & iff & $\forall a\forall b\forall c(a\leq \blacksquare b \ \& \ a\leq \bt(b\wedge c)\Rightarrow (a\wedge b)\leq \bt c)$\\
    & iff & $\forall a\forall b\forall c(a\prec b \ \& \ a\npcon (b\wedge c)\Rightarrow (a\wedge b)\npcon c)$\\
    \end{tabular}
\end{center}

4(i)

\begin{center}
    \begin{tabular}{cll}
    $\forall a (\Diamond a \leq \Diamond a\pdla \wt a)$ & iff & $ \forall a \forall c ( \Diamond a\pdla \wt a\leq c \Rightarrow \Diamond a \leq c )$\\
    & iff & $ \forall a \forall c ( \Diamond a\leq c \vee \wt a\Rightarrow \Diamond a \leq c )$\\
    & iff & $ \forall a \forall c ( \exists b (b\leq \wt a \ \& \ \Diamond a\leq c \vee b)\Rightarrow \Diamond a \leq c )$\\
    & iff & $ \forall a \forall b\forall c ( b\leq \wt a \ \& \ \Diamond a\leq c \vee b\Rightarrow \Diamond a \leq c )$\\
    & iff & $ \forall a \forall b\forall c ( a \npcon b \ \& \ a\prec b \vee c\Rightarrow a \prec c )$\\
     \end{tabular}
\end{center}

5(i)

\begin{center}
    \begin{tabular}{cll}
    $\forall a (\tw a \leq \Diamond(\tw a\pdla a)$ & iff & $\forall a \forall c (\Diamond(\tw a\pdla a)\leq c \Rightarrow \tw a\leq c)$\\
    & iff & $\forall a \forall c (\exists b(\Diamond(b\pdla a)\leq c\ \& \ \tw a\leq b) \Rightarrow \tw a\leq c)$\\
     & iff & $\forall a\forall b \forall c (\Diamond(b\pdla a)\leq c\ \& \ \tw a\leq b \Rightarrow \tw a\leq c)$\\
     & iff & $\forall a\forall b \forall c ((b\pdla a)\prec c\ \& \ a\nppcon b \Rightarrow  a\nppcon c)$\\
       \end{tabular}
\end{center}
\end{proof}
\section{Applications}
\label{sec: applications}
In the present section, we discuss four independent but connected applications of the characterization results of the previous section. 
\subsection{Characterizing output operators}
\label{ssec: characterizing output}

The output operators $out^N_i$ for $1\leq i\leq 6$ associated with a given input/output logic $\mathbb{L} = (\mathcal{L}, N)$ can be given semantic counterparts in the environment of proto-subordination algebras as follows: for every proto-subordination algebra $\mathbb{S} = (A, \prec)$,  we let $\mathbb{S}_i: = (A, \prec_i)$ where ${\prec_i}\subseteq A\times A$ is the smallest extension   of $\prec$ which satisfies the properties indicated in the following table:\footnote{While (the purely logical counterparts of) outputs 1 to 4 are  well known in the literature of input/output logic (cf.~\cite{Makinson00}), the last two outputs are novel.}
\begin{center}
	\begin{tabular}{ l l}
		\hline
		$\prec_i$ & Properties   \\
		\hline	
		$\prec_{1}$  & ($\top$), (SI), (WO), (AND)   \\
		$\prec_{2}$  & ($\top$), (SI), (WO), (AND), (OR)  \\
		$\prec_{3}$  & ($\top$), (SI), (WO), (AND), (CT)  \\
		$\prec_{4}$  & ($\top$), (SI), (WO), (AND), (OR), (CT) \\
		\hline
$\prec_{5}$  & ($\bot$), (SI), (WO), (OR)   \\
		$\prec_{6}$  & ($\bot$), (SI), (WO), (OR), (DCT)  \\
  \hline
	\end{tabular}
\end{center}

Then, for each $1\leq i\leq 4$ and every  $k \in K(A^\delta)$,\footnote{Recall that, by definition, if $k \in K(A^\delta)$, then $k = \bigwedge D$ for some nonempty down-directed $D\subseteq A$.}
\[\Diamond_i^{\sigma}k: = \bigwedge\{ \bigwedge {\prec_i}[a]\mid a\in A \text{ and } k\leq a\}\] encodes the algebraic  counterpart of $out^N_i(\Gamma)$ for any $\Gamma\subseteq \mathrm{Fm}$, and the characteristic properties of $\Diamond_i$ for each $1\leq i\leq 4$ are those  identified in Lemma  \ref{lem: axioms to relation}, Corollary \ref{cor: charact monotone reg norm}, and  Proposition \ref{prop:characteriz}. For any defined proto-subordination algebra $\mathbb{S} = (A, \prec)$, let $\mathbb{S}_i^\ast: = (A, \Diamond_i, \blacksquare_i)$ denote the slanted algebras associated with $\mathbb{S}_i = (A, \prec_i)$ for each $1\leq i\leq 4$.

\begin{proposition}
\label{prop:output}
For any defined proto-subordination algebra $\mathbb{S} = (A, \prec)$,
\begin{enumerate}
    \item $\Diamond_1$ is the largest monotone map dominated by $\Diamond$ (i.e.~pointwise-smaller than or equal to $\Diamond$). 
    \item $\Diamond_2$ is the largest regular map dominated by $\Diamond$. 
    \item $\Diamond_3$ is the largest monotone map satisfying $\Diamond_3 a\leq \Diamond_3(a\wedge \Diamond_3 a)$ dominated by $\Diamond$. 
    \item $\Diamond_4$ is the largest regular map satisfying $\Diamond_4 a\leq \Diamond_4(a\wedge \Diamond_4 a)$ dominated by $\Diamond$. 
    \item $\Diamond_5$ is the largest normal map dominated by $\Diamond$.
    \item $\Diamond_6$ is the largest normal map dominated by $\Diamond$, the adjoint of which satisfies  $\blacksquare_6 (a \vee \blacksquare_6 a) \leq \blacksquare_6 a$.
\end{enumerate}
\end{proposition}

\begin{proof}
 By Lemma \ref{lem: axioms to relation} and Proposition \ref{prop:characteriz}, the properties stated in each item of the statement hold for $\Diamond_i$ and $\blacksquare_i$.  To complete the proof, we need to argue for $\Diamond_i$ being the largest such map 
 By Lemma \ref{lem: diamond-output equivalence} 1.(ii), $a\prec_i b$ iff $\Diamond_i a\leq b$  for  all $a, b\in A$ and $1\leq i\leq 4$. Any $f: A\to A^\delta$  s.t.~$f(a)\in K(A^\delta)$ for every $a\in A$ induces a proto-subordination relation  $\prec_f\subseteq A\times A$ defined as $a\prec_f b$ iff $f(a)\leq b$. Clearly, if $f(a)\leq f' (a)$ for every $a\in A$, then ${\prec}_{f'}\subseteq {\prec_f}$. Moreover, if $f(a)< f' (a)$, then, by denseness, $f(a)\leq b$ for some $b\in A$ s.t.~$f'(a)\nleq b$, hence ${\prec}_{f'}\subset {\prec_f}$.

 If $\Diamond_i$ is not the largest map endowed with the properties mentioned in the statement and dominated by $\Diamond$, then  a map $f$ exists which is endowed with these properties such that  $\Diamond_i a \leq f(a) \leq \Diamond a$ for all $a \in A$, and  $\Diamond_i b < f(b)$ for some $b \in A$. Then, by the argument in the previous paragraph,  ${\prec} = {\prec_\Diamond} \subseteq {\prec_f} \subset {\prec}_{\Diamond_i}=\prec_i$. As $f$ is endowed with the the properties mentioned in the statement,  $\prec_f$ is an extension of $\prec$ which enjoys the required properties, and is strictly contained in $\prec_i$. Hence, $\prec_i$ is not the smallest such extension.
\end{proof}

\begin{remark}
   The modal characterizations stated in the proposition  above are similar to the characterization of $out_2$-$out_4$shown in \cite[Observation 4 in Section 4.3]{Makinson00}. 
   However, while characterizations of this type are regarded in \cite{Makinson00} as ``interesting curiosities more than
useful tools'', in the context of the present formal framework, these modal characterizations can be made systematic and   acquire meaning and use.
\end{remark}

 For any selfextensional logic  $\mathcal{L}$ and any permission system $P$, the  output operators associated with 
 $P_i^c$ for $1\leq i\leq 4$ can be given semantic counterparts in the environment of proto sp-algebras (cf.~Section \ref{sec: modal charact}) as follows: for every proto sp-algebra $\mathbb{H} = (A, \prec, \pcon)$,  we let $\mathbb{H}_i: = (A, \prec_i, \pcon_{i})$, where $\prec_i\ \subseteq \ A\times A$ is defined  as discussed at the beginning of the present section, and ${\pcon_i}\subseteq A\times A$ is the relative complement in $A\times A$ of the smallest extension   of $ \npcon$ satisfying the properties indicated in the following table:

\begin{center}
	\begin{tabular}{ l l}
		\hline
		$\mathcal{C}_i$ & Properties   \\
		\hline	
		$\mathcal{C}_{1}$  & $ \mathrm{(\top)^\wt, (SI)^\wt, (WO)^\wt, (AND)^\wt}$   \\
		$\mathcal{C}_{2}$  & $\mathrm{ (\top)^\wt, (SI)^\wt, (WO)^\wt, (AND)^\wt, (OR)^\wt}$  \\
		$\mathcal{C}_{3}$  & $\mathrm{ (\top)^\wt, (SI)^\wt, (WO)^\wt, (AND)^\wt, (CT)^\wt}$  \\
		$\mathcal{C}_{4}$  & $\mathrm{ (\top), (SI)^\wt, (WO)^\wt,(AND)^\wt, (OR)^\wt, (CT)^\wt}$ \\
		\hline
	\end{tabular}
\end{center}
Then, for each $1\leq i\leq 4$ and every  $k \in K(A^\delta)$, \[\Diamond_i^{\sigma}k: = \bigwedge\{ \bigwedge {\prec_i}[a]\mid a\in A \text{ and } k\leq a\}\ \text{ and }\ \wt_i^\pi k:= \bigvee\{ \bigvee (\pcon_i [a])^c\mid a\in A\mbox{ and } k\leq a\}\] respectively encode the algebraic  counterparts of the output operators associated with a given normative and permission system  on a given selfextensional logic $\mathcal{L}$, and the characteristic properties of $\Diamond_i$ and $\wt_i$ for each $1\leq i\leq 4$ are those  identified in Lemma  \ref{lem: axioms to relation}, Corollary \ref{cor: charact monotone reg norm}, and  Proposition \ref{prop:characteriz}. For any defined\footnote{A proto sp-algebra $\mathbb{H} = (A, \prec, \pcon)$ is {\em defined} if both $(A, \prec)$ and $(A, \pcon)$ are.} proto sp-algebra $\mathbb{H} = (A, \prec, \pcon)$, let $\mathbb{H}_i^\ast: = (A, \Diamond_i, \wt_i)$ denote the slanted algebra associated with $\mathbb{H}_i = (A, \prec_i, \pcon_i)$ for each $1\leq i\leq 4$.

\begin{proposition}
\label{prop:output contact}
For any defined proto sp-algebra $\mathbb{H} = (A, \prec, \pcon)$,
\begin{enumerate}
    \item $\wt_1$ is the smallest antitone map dominating $\wt$.
    \item $\wt_2$ is the smallest regular map dominating $\wt$.
    \item $\wt_3$ is the smallest antitone map satisfying $\wt_3(a\wedge \Diamond_3 a)\leq \wt_3 a$ dominating $\wt$.
    \item $\wt_4$ is the smallest regular map satisfying  $\wt_4(a\wedge \Diamond_4 a)\leq \wt_4 a$ dominating $\wt$.
\end{enumerate}
\end{proposition}
\begin{proof}
 By Lemma \ref{lem: axioms to relation} and Proposition \ref{prop:poly correspondence permission}, the properties mentioned in each item of the statement hold for $\wt_i$.  To complete the proof, we need to argue for $\wt_i$ being the smallest such map.  By Lemma \ref{lem: diamond-output equivalence} 2.(ii), $a\npcon_i b$ iff $b \leq \wt_i a$  for  all $a, b\in A$ and $1\leq i\leq 4$. Any $g: A\to A^\delta$  s.t.~$g(a)\in O(A^\delta)$ for every $a\in A$ induces (the relative complement of) a proto-precontact relation  $\npcon_g\subseteq A\times A$ defined as $a\npcon_g b$ iff $b\leq g(a)$. Clearly, if $g'(a)\leq g (a)$ for every $a\in A$, then ${\npcon}_{g'}\subseteq \npcon_{g}$. Moreover, if $g'(a)< g(a)$, then, by denseness, $b\leq g(a)$ for some $b\in A$ s.t.~$b\nleq g'(a)$, hence ${\npcon}_{g'}\subset {\npcon_{g}}$.

 If $\wt_i$ is not the smallest map endowed with the properties mentioned in the statement dominating $\wt$, then  a map $g$ exists which is endowed with these properties such that  $\wt a \leq g(a)\leq  \wt_i a$ for all $a \in A$, and  $\wt_i b < g(b)$ for some $b \in A$. Then, by the argument in the previous paragraph,  ${\npcon} = {\npcon_\wt} \subseteq {\npcon_g} \subset {\npcon}_{\wt_i}=\npcon_i $. As $g$ is endowed with the the properties mentioned in the statement,  $\npcon_g$ is an extension of $\npcon$ which enjoys the required properties mentioned in the table, and is strictly contained in $\npcon_i$. Hence, $\npcon_i$ is not the smallest such extension.
\end{proof}

Likewise, for dual  permission systems $D$, the  output operator associated with 
 $D_i^c$ for $1\leq i\leq 4$ can be given semantic counterpart in the environment of proto sdp-algebras (cf.~Footnote \ref{footn:sdp}) as follows: for every proto sdp-algebra $\mathbb{K} = (A, \prec, \ppcon)$,  we let $\mathbb{K}_i: = (A, \prec_i, \ppcon_{i})$, where $\prec_i\ \subseteq \ A\times A$ is defined  as discussed at the beginning of the present section, and ${\ppcon_i}\subseteq A\times A$ is the relative complement in $A\times A$ of the smallest extension   of $ \nppcon$ satisfying the properties indicated in the following table:

\begin{center}
	\begin{tabular}{ l l}
		\hline
		$\mathcal{D}_i$ & Properties   \\
		\hline	
		$\mathcal{D}_{1}$  & $ \mathrm{(\top)^\tw, (SI)^\tw, (WO)^\tw, (AND)^\tw}$   \\
		$\mathcal{D}_{2}$  & $\mathrm{ (\top)^\tw, (SI)^\tw, (WO)^\tw, (AND)^\tw, (OR)^\tw}$  \\
		$\mathcal{D}_{3}$  & $\mathrm{ (\top)^\tw, (SI)^\tw, (WO)^\tw, (AND)^\tw, (CT)^\tw}$  \\
		$\mathcal{D}_{4}$  & $\mathrm{ (\top), (SI)^\tw, (WO)^\tw,(AND)^\tw, (OR)^\tw, (CT)^\tw}$ \\
		\hline
	\end{tabular}
\end{center}
Then, for each $1\leq i\leq 4$, every  $k \in K(A^\delta)$, and every  $o \in O(A^\delta)$, \[\Diamond_i^{\sigma}k: = \bigwedge\{ \bigwedge {\prec_i}[a]\mid a\in A \text{ and } k\leq a\}\ \text{ and }\ \tw_i^\pi o:= \bigwedge\{ \bigvee (\ppcon_i [a])^c\mid a\in A\mbox{ and } a\leq o\}\] respectively encode the algebraic  counterparts of the output operators associated with a given normative and dual permission system  on a given selfextensional logic $\mathcal{L}$, and the characteristic properties of $\Diamond_i$ and $\tw_i$ for each $1\leq i\leq 4$ are those  identified in Lemma  \ref{lem: axioms to relation}, Corollary \ref{cor: charact monotone reg norm}, and  Proposition \ref{prop:characteriz}. For any defined\footnote{A proto sdp-algebra $\mathbb{K} = (A, \prec, \ppcon)$ is {\em defined} if both $(A, \prec)$ and $(A, \ppcon)$ are.} proto sdp-algebra $\mathbb{K} = (A, \prec, \ppcon)$, let $\mathbb{K}_i^\ast: = (A, \Diamond_i, \tw_i)$ denote the slanted algebra associated with $\mathbb{K}_i = (A, \prec_i, \ppcon_i)$ for each $1\leq i\leq 4$.

\begin{proposition}
\label{prop:output dual contact}
For any defined proto sdp-algebra $\mathbb{K} = (A, \prec, \ppcon)$,
\begin{enumerate}
    \item $\tw_1$ is the largest antitone map dominated by $\tw$.
    \item $\tw_2$ is the largest regular map dominated by $\tw$.
    \item $\tw_3$ is the largest antitone map satisfying $\tw_3 a\leq \Diamond_3 (\tw_3 a\pdla   a)$ dominated by $\tw$.
    \item $\tw_4$ is the largest regular map satisfying  $\tw_4 a\leq \Diamond_4 (\tw_4 a\pdla   a)$ dominated by $\tw$.
\end{enumerate}
\end{proposition}
\begin{proof}
 By Lemma \ref{lem: axioms to relation} and Proposition \ref{prop:poly correspondence permission}, the properties mentioned in each item of the statement hold for $\tw_i$.  To complete the proof, we need to argue for $\tw_i$ being the smallest such map.  By Lemma \ref{lem: diamond-output equivalence} 3.(ii), $a\nppcon_i b$ iff $\tw_i a\leq b$  for  all $a, b\in A$ and $1\leq i\leq 4$. Any $f: A\to A^\delta$  s.t.~$f(a)\in K(A^\delta)$ for every $a\in A$ induces (the relative complement of) a dual proto-precontact relation  $\nppcon_f\subseteq A\times A$ defined as $a\nppcon_f b$ iff $f(a)\leq b$. Clearly, if $f(a)\leq f' (a)$ for every $a\in A$, then ${\nppcon}_{f'}\subseteq \nppcon_{f}$. Moreover, if $f'(a)< f(a)$, then, by denseness, $f'(a)\leq b$ for some $b\in A$ s.t.~$f(a)\nleq b$, hence ${\nppcon}_{f'}\subset {\nppcon_{f}}$.

 If $\tw_i$ is not the largest map endowed with the properties mentioned in the statement and dominated by $\tw$, then  a map $f$ exists which is endowed with these properties such that  $\tw_i a \leq f(a) \leq \tw a$ for all $a \in A$, and  $\tw_i b < f(b)$ for some $b \in A$. Then, by the argument in the previous paragraph,  ${\nppcon} = {\nppcon_\tw} \subseteq {\nppcon_f} \subset {\nppcon}_{\tw_i}=\nppcon_i$. As $f$ is endowed with the the properties mentioned in the statement,  $\nppcon_f$ is an extension of $\nppcon$ which enjoys the required properties, and is strictly contained in $\nppcon_i$. Hence, $\nppcon_i$ is not the smallest such extension.
\end{proof}

\subsection{Algebraizing static positive permissions}
\label{ssec:algebraizing static permission}
Static positive permission has been introduced in \cite{Makinson03} as a stronger notion of conditional permission, introduced to mark the difference between permissions which derive their status from the fact that there is no norm forbidding them, and those that are or derive from e.g.~{\em civil rights}.  The intuitive idea is that if $\prec$ and $\pcon$ represent a set of  norms and a set of (explicit) permissions respectively, then every positive permission generated by $\prec$ and $\pcon$ is a member of the {\em normative} system generated by closing $\prec$ and some element in $\pcon$ under some Horn-type condition. In \cite{dedomenico2024obligations}, this notion has been generalized to normative and permission systems on selfextensional logics.
Throughout this section, we assume that $\mathbb{H} = (A, \prec, \pcon)$ is an sp-algebra  s.t.~$\pcon\subseteq \pcon_{\prec}$.
\begin{definition}
For any sp-algebra $\mathbb{H} = (A, \prec, \pcon)$ s.t.~$\pcon\subseteq \pcon_{\prec}$ 
%
and  any Horn-type condition $\mathrm{(R)}$ in the first order language of $\mathbb{H}$,  we let
\begin{center}
	$ S^{\mathrm{(R)}}(\prec, \pcon) :=\begin{cases} \bigcup\{ \prec^{\mathrm{(R)}}_{(a, b)} \mid (a, b)\in \pcon\} & \text{if } \pcon\neq \varnothing\\
 \prec^{\mathrm{(R)}} & \text{ otherwise},
 \end{cases}$
 \end{center}
where $\prec^{\mathrm{(R)}}_{(a, b)}$	is the smallest superset of $\prec\cup\{(a, b)\}$ satisfying condition $\mathrm{(R)}$.
	
\end{definition}

 For any selfextensional logic $\mathcal{L}$, any $A\in \mathsf{Alg}(\mathcal{L})$ and any  $S\subseteq A$, we let $ F_{\mathcal{L}}(S)$ denote the $\mathcal{L}$-filter generated by $S$.
\begin{definition}
\label{def:cross-coherent}
  For any selfextensional logic $\mathcal{L}$,   an sp-algebra $\mathbb{H} = (A, \prec, \pcon)$ s.t.~$A\in \mathsf{Alg}(\mathcal{L})$ and $\pcon\subseteq \pcon_{\prec}$ is
   \begin{enumerate}
   \item $\mathrm{(R)}$-{\em cross-incoherent} if $a\prec b$ and $(a, c) \in S^{\mathrm{(R)}}(\mathcal{C},\prec)$ for some $a, b, c \in A$  s.t.~$F_{\mathcal{L}}(a) \neq A$ and $F_{\mathcal{L}}( b, c) = A$. If $\mathbb{H}$ is not $\mathrm{(R)}$-cross-incoherent, then it is   $\mathrm{(R)}$-{\em cross-coherent}.
   \item ${\mathrm{(R)}}$-\emph{updirected} if for all $a, b, c, d \in A$ s.t.~$a\pcon b$ and $c\pcon d$, some $e, f\in A$  exist s.t.~$e\pcon f$ s.t.$\prec_{(a,b)}^{\mathrm{(R)}} \cup \prec_{(c,d)}^{\mathrm{(R)}} \subseteq \prec_{(e,f)}^{\mathrm{(R)}}$.
   \end{enumerate}
\end{definition}

 Any $\mathbb{H} = (A, \prec, \pcon)$  s.t.~$\pcon$ is a singleton set is updirected. Moreover, if $\mathbb{H}$ is lattice-based, $(A, \prec)\models \mathrm{(WO)+(SI)}$, and
    $\pcon = \{(a,b),(c,d),(a\vee c,b\wedge d)\}$,  then  $\mathbb{H}$ is updirected.

For any Horn-type condition $\mathrm{(R)}$ and any sp-algebra $\mathbb{H} = (A, \prec, \pcon)$, let $\mathbb{S}_\mathbb{H}^{\mathrm{(R)}}: = (A, S^{\mathrm{(R)}}(\prec, \pcon))$ be the proto-subordination algebra associated with  $\mathbb{H}$.
\begin{proposition}
For any sp-algebra $\mathbb{H} = (A, \prec, \pcon)$ s.t.~$\pcon\subseteq \pcon_{\prec}$,
\begin{enumerate}
    \item $\mathbb{S}_\mathbb{H}^{\mathrm{(X)}}\models\mathrm{(X)}$ for   any $(\mathrm{X})\in \{\mathrm{(WO)}, \mathrm{(SI)}\}$.
    \item If $\mathbb{H}$ is lattice-based and ${\mathrm{(X)}}$-updirected, then $\mathbb{S}_\mathbb{H}^{\mathrm{(X)}}\models\mathrm{(X)}$ for   any $(\mathrm{X})\in \{\mathrm{(AND)}, \mathrm{(OR)}\}$.
\end{enumerate}
\end{proposition}
\begin{proof}
    1. Let us show the statement for $\mathrm{(X)} = \mathrm{(WO)}$. If $\pcon = \varnothing$, the statement follows straightforwardly from the definition.  If $\pcon\neq \varnothing$, let $a, b, c\in A$ s.t.~$a\leq b$ and $b S^{\mathrm{(WO)}}(\prec, \pcon)) c$. Hence, $b\prec_{(d, e)}^{\mathrm{(WO)}}c$ for some $d, e\in A$ s.t.~$d\pcon e$. Since, by definition, $(A, \prec_{(d, e)}^{\mathrm{(WO)}})\models \mathrm{(WO)}$, we conclude $a \prec_{(d, e)}^{\mathrm{(WO)}} c$, and hence $aS^{\mathrm{(WO)}}(\prec, \pcon)) c$, as required. The proof of the statement for $\mathrm{(X)} = \mathrm{(SI)}$ is analogous and is omitted.

    2.  Let us show the statement for $\mathrm{(X)} = \mathrm{(AND)}$. If $\pcon = \varnothing$, the statement follows straightforwardly from the definition.  If $\pcon\neq \varnothing$, let $a, b, c\in A$ s.t.~$aS^{\mathrm{(AND)}}(\prec, \pcon) b$ and $a S^{\mathrm{(AND)}}(\prec, \pcon) c$. Hence, $a\prec_{(d, e)}^{\mathrm{(AND)}}b$ and $a\prec_{(d', e')}^{\mathrm{(AND)}}c$ for some $d, e, d', e'\in A$ s.t.~$d\pcon e$ and $d'\pcon e'$. Since $\mathbb{H}$ is ${\mathrm{(AND)}}$-updirected, some $d'', e''\in A$ exist s.t.~$a\prec_{(d'', e'')}^{\mathrm{(AND)}}b$ and $a\prec_{(d'', e'')}^{\mathrm{(AND)}}c$.  By definition, $(A, \prec_{(d'', e'')}^{\mathrm{(AND)}})\models \mathrm{(AND)}$, hence we conclude $a \prec_{(d'', e'')}^{\mathrm{(AND)}} (b\wedge c)$, and hence $aS^{\mathrm{(AND)}}(\prec, \pcon) (b\wedge c)$, as required. The proof of the statement for $\mathrm{(X)} = \mathrm{(OR)}$ is analogous and is omitted.
\end{proof}

The proposition above implies that,  if $\mathbb{H} = (A, \prec, \pcon)$ is bounded lattice-based and $\mathrm{(X)}$-updirected for every $\mathrm{(X)}\in \{\mathrm{(SI),(WO),(AND),(OR)}\}$, then $S^{\mathrm{(SI)},\mathrm{(WO)},\mathrm{(AND)},\mathrm{(OR)}}(\prec, \pcon)$ (abbreviated as $S(\prec, \pcon)$) is a subordination relation on $A$, which can be associated with a normal slanted operator $\blacksquare_s: A\to A^{\delta}$   as described in Section \ref{sec: proto and slanted}. With the added expressivity of this operator,   the notion of cross-coherence can be modally characterized as follows in sp-algebras as above, which in addition are based on Heyting-algebras:

\begin{center}
    \begin{tabular}{cll}
     & $\forall a \forall b \forall c(a S(\prec, \pcon) b \ \&\ a\prec c \ \&\ b\wedge c \leq \bot \Rightarrow a \leq \bot)$ \\
       iff  & $\forall a \forall b \forall c(a \leq \blacksquare_s b \ \&\ a\leq \blacksquare c \ \&\ b\wedge c \leq \bot \Rightarrow a \leq \bot)$ \\
       iff  & $\forall a \forall b \forall c(a \leq \blacksquare_s b \wedge \blacksquare c \ \&\ b\wedge c \leq \bot \Rightarrow a \leq \bot)$ \\
         iff & $ \forall b \forall c(b \wedge c \leq \bot \Rightarrow \forall a( a \leq \blacksquare_s b \wedge \blacksquare c  \Rightarrow a \leq \bot))$ \\
         iff & $ \forall b \forall c(b \wedge c \leq \bot \Rightarrow  \blacksquare_s b \wedge \blacksquare c  \leq \bot)$ \\
         iff  & $ \forall b \forall c( b\leq c\rightarrow\bot \Rightarrow  \blacksquare_s b \wedge \blacksquare c  \leq \bot)$ \\
         iff  & $  \forall c(\blacksquare_s(c\rightarrow \bot) \wedge \blacksquare c  \leq \bot)$. & $(\ast)$ \\

    \end{tabular}
\end{center}
Let us show the equivalence marked with $(\ast)$: from top to bottom, it is enough to instantiate $b: = c\rightarrow \bot$; conversely, by the monotonicity of $\wedge$ and $\blacksquare_s$,  If $b\leq c\rightarrow \bot$ then  $\blacksquare_s b\wedge \blacksquare c\leq \blacksquare_s (c\rightarrow \bot)\wedge \blacksquare c\leq\bot$.

\subsection{Dual characterization of conditions on relational algebras}
\label{ssec: dual
 charact}
 In \cite{celani2020subordination}, Celani introduces an  expansion of Priestley's duality for bounded distributive lattices to
 {\em subordination lattices}, i.e.~tuples $\mathbb{S} =(A, \prec)$ such that $A$ is a distributive lattice and ${\prec}\subseteq A\times A$ is a subordination relation.\footnote{In the terminology of the present paper, subordination lattices correspond to distributive lattice-based subordination algebras  (cf.~Definition \ref{def: subordination algebra}).} The  dual structure of  any subordination lattice $\mathbb{S} =(A, \prec)$ is referred to as the {\em (Priestley) subordination space} of $\mathbb{S}$, and is defined as  $\mathbb{S}_*: = (X(A), R_\prec)$, where $X(A)$ is (the Priestley space dual to $A$, based on) the poset of prime filters of $A$ ordered by inclusion,   and $R_\prec \subseteq X(A) \times X(A)$ is defined as follows: for all prime filters $P, Q$ of $A$,
\[
(P,Q) \in R_\prec \quad \text{ iff }\quad {\prec}[P]: =\{x\in A\mid \exists a(a\in P\ \&\ a\prec x)\} \subseteq Q.
\]

Up to isomorphism, we can equivalently define the subordination space of $\mathbb{S}$ as follows:
\begin{definition}
 \label{def: subordination space in adelta}
 The  {\em subordination space} associated with a subordination lattice $\mathbb{S} =(A, \prec)$ is $\mathbb{S}_*: = ((\jty(A^\delta), \sqsubseteq), R_\prec)$, where $(\jty(A^\delta), \sqsubseteq)$ is the poset of the completely join-irreducible elements of $A^\delta$ s.t.~$j\sqsubseteq k$ iff $j\geq k$, and $R_{\prec}$ is a binary relation on $ \jty(A^\delta)$ such that $  R_{\prec}(j, i)$ iff $i\leq \Diamond j$ iff $\blacksquare \kappa (i)\leq \kappa(j)$, where $\kappa (i): = \bigvee\{i'\in \jty(A^\delta)\mid i\nleq i'\}$.
\end{definition}
\begin{lemma}
For any subordination lattice $\mathbb{S} =(A, \prec)$, the subordination spaces $\mathbb{S}_*$ given according to the two definitions   above are isomorphic.
\end{lemma}
\begin{proof}
 As is well known, in the canonical extension $A^\delta$ of any  distributive lattice $A$, the set $\jty(A^\delta)$ of the completely join-irreducible elements of $A^\delta$  coincides with the set of its completely join-prime elements, which are in dual order-isomorphism with the poset of prime filters of $A$ ordered by inclusion. Specifically, if $P\subseteq A$ is a prime filter, then $j_P: = \bigwedge P\in K(A^\delta)$ is a completely join-prime element of $A^\delta$; conversely, if $j$ is a completely join-prime element of $A^\delta$, then  $P_j: = \{a\in A\mid j\leq a\}$  is a prime filter of $A$. Clearly, $j = \bigwedge P_j = j_{P_j}$ for any $j\in \jty(A^\delta)$; moreover, it is easy to show, by applying compactness, that $P_{j_{P}} = \{a\in A\mid \bigwedge P\leq a\} = P$ for any prime filter $P$ of $A$.

 To complete the proof and show that the two relations $R_{\prec}$ can be identified modulo the identifications above, it is enough to show that  ${\prec}[P]\subseteq Q$ iff $\bigwedge Q\leq \bigwedge {\prec}[P]$ for all prime filters $P$ and $Q$ of $A$. Clearly, ${\prec}[P]\subseteq Q$ implies  $\bigwedge Q\leq \bigwedge {\prec}[P]$. Conversely, if $b\in {\prec} [P]$, then  $\bigwedge Q\leq \bigwedge {\prec}[P]\leq b$, hence, by compactness and $Q$ being an up-set, $b\in Q$, as required.
\end{proof}

Likewise, we can consider {\em  precontact lattices} (resp.~{\em dual precontact lattices}) as  precontact algebras (resp.~dual precontact algebras) based on bounded distributive lattices,\footnote{ Analogously, we can define e.g.~{\em bi-subordination lattices}, or {\em sp-lattices} and  their corresponding spaces in the obvious way. We omit the details since they are straightforward.} and define
\begin{definition}
 \label{def: precontact space in adelta}
 The  {\em precontact space} associated with a precontact lattice $\mathbb{C} =(A, \pcon)$ is $\mathbb{C}_*: = ((\jty(A^\delta), \sqsubseteq), R_{\npcon})$, where $(\jty(A^\delta), \sqsubseteq)$ is the poset of the completely join-irreducible elements of $A^\delta$ s.t.~$j\sqsubseteq k$ iff $j\geq k$, and $R_{\npcon}$ is a binary relation on $ \jty(A^\delta)$ such that $  R_{\npcon}(j, i)$  iff $\wt  j\leq \kappa(i)$ iff  $\bt  i\leq \kappa(j)$, where $\kappa (i): = \bigvee\{i'\in \jty(A^\delta)\mid i\nleq i'\}$.\footnote{\label{footn:kappa} From the definition of $\kappa$ and the join primeness of every $j\in J^{\infty}(A^\delta)$, it immediately follows that $j\nleq u$ iff $u\leq \kappa(j)$ for every $u \in A^\delta$ and  every $j \in J^{\infty}(A^\delta)$.}

 The  {\em dual precontact space} associated with a dual precontact lattice $\mathbb{D} =(A, \ppcon)$ is $\mathbb{D}_*: = ((\jty(A^\delta), \sqsubseteq), R_{\nppcon})$, where $(\jty(A^\delta), \sqsubseteq)$ is the poset of the completely join-irreducible elements of $A^\delta$ s.t.~$j\sqsubseteq k$ iff $j\geq k$, and $R_{\nppcon}$ is a binary relation on $ \jty(A^\delta)$ such that $  R_{\nppcon}(j, i)$  iff $i\leq \tw \kappa(j)$ iff  $ j\leq \tb \kappa (i)$. 
\end{definition}
In \cite{celani2020subordination}, some properties of subordination lattices are dually characterized in terms  properties of their associated subordination spaces. In the following proposition, we obtain  these results as consequences of the dual characterizations in  Proposition \ref{prop:characteriz}, slanted canonicity \cite{de2021slanted}, and correspondence theory for (standard) distributive  modal logic \cite{conradie2012algorithmic}.
\begin{proposition}
\emph{(cf.~\cite{celani2020subordination}, Theorem 5.7)}
\label{prop: celani}
For any subordination lattice $\mathbb{S}$,

\begin{enumerate}[label=(\roman*)]
    \item $\mathbb{S}\models {\prec} \subseteq {\leq}\quad$ iff $\quad R_{\prec}$ is reflexive;

    \item $\mathbb{S}\models \mathrm{(D)}\quad $ iff $\quad R_{\prec}$ is transitive, i.e.~$R_{\prec}\circ R_{\prec}\subseteq R_{\prec}$; \item $\mathbb{S}\models \mathrm{(T)}\quad $ iff $\quad R_{\prec}$ is dense, i.e.~$R_{\prec}\subseteq R_{\prec}\circ R_{\prec}$;
\end{enumerate}
\end{proposition}
\begin{proof}
 (i) By Proposition \ref{prop:characteriz}.1(i), $\mathbb{S}\models {\prec} \subseteq {\leq}$ iff 
 $\mathbb{S}^{*}\models a\leq \Diamond a$; the inequality $a\leq \Diamond a$ is analytic inductive (cf.~\cite[Definition 55]{greco2018unified}), and hence {\em slanted} canonical by \cite[Theorem 4.1]{de2021slanted}. Hence, from $\mathbb{S}^{*}\models a\leq \Diamond a$ it follows that $(\mathbb{S}^{*})^\delta\models a\leq \Diamond a$, where $(\mathbb{S}^{*})^\delta$  is a {\em standard} (perfect) distributive modal algebra. By algorithmic correspondence theory for distributive modal logic (cf.~\cite[Theorems 8.1 and 9.8]{conradie2012algorithmic}), $(\mathbb{S}^{*})^\delta\models a\leq \Diamond a$ iff $(\mathbb{S}^{*})^\delta\models \forall \nomj (\nomj \leq \Diamond \nomj)$ where $\nomj$ ranges in the set $\jty((\mathbb{S}^{*})^\delta)$. By Definition \ref{def: subordination space in adelta}, this is equivalent to $R_{\prec}$ being reflexive.

(ii) By Proposition  \ref{prop:characteriz}.1(iv), $\mathbb{S}\models \mathrm{(D)}$ iff  $\mathbb{S}^{*}\models \Diamond \Diamond a\leq \Diamond a$; the inequality $\Diamond \Diamond a\leq \Diamond a$ is analytic inductive (cf.~\cite[Definition 55]{greco2018unified}), and hence {\em slanted} canonical by \cite[Theorem 4.1]{de2021slanted}. Hence, from $\mathbb{S}^{*}\models \Diamond \Diamond a\leq \Diamond a$ it follows that $(\mathbb{S}^{*})^\delta\models \Diamond \Diamond a\leq \Diamond a$, where $(\mathbb{S}^{*})^\delta$ is a {\em standard} (perfect) distributive modal algebra. By algorithmic correspondence theory for distributive modal logic (cf.~\cite[Theorems 8.1 and 9.8]{conradie2012algorithmic}), $(\mathbb{S}^{*})^\delta\models \Diamond \Diamond a\leq \Diamond a$ iff $(\mathbb{S}^{*})^\delta\models \forall \nomj (\Diamond \Diamond \nomj \leq \Diamond \nomj)$ where $\nomj$ ranges in  $\jty((\mathbb{S}^{*})^\delta)$. The following chain of equivalences holds in any (perfect) algebra $A^\delta$:

\begin{center}
    \begin{tabular}{clll}
    $\forall \nomj (\Diamond\Diamond \nomj\leq \Diamond\nomj)$ & iff & $\forall \nomj\forall \nomk (\nomk\leq \Diamond\Diamond \nomj \Rightarrow \nomk\leq \Diamond\nomj)$ & ($\ast\ast$)\\
    & iff & $\forall \nomj\forall \nomk (\exists \nomi (\nomi\leq \Diamond \nomj \ \& \ \nomk\leq \Diamond\nomi) \Rightarrow \nomk\leq \Diamond\nomj)$ & ($\ast$)\\
    & iff & $\forall \nomj\forall \nomk\forall  \nomi ( (\nomi\leq \Diamond \nomj \ \& \ \nomk\leq \Diamond\nomi) \Rightarrow \nomk\leq \Diamond\nomj)$\\
    \end{tabular}
\end{center}
The equivalence marked with ($\ast\ast$) is due to the fact that canonical extensions of distributive lattices are completely join-generated by the completely join-prime elements.
Let us show the equivalence marked with ($\ast$): as is well known,  every  perfect distributive lattice is  join-generated by its completely join-irreducible elements; hence, $\Diamond j = \bigvee\{i\in \jty(A^\delta)\mid i\leq \Diamond j\}$, and since $\Diamond$ is completely join-preserving, $\Diamond \Diamond j = \bigvee\{\Diamond i\mid i\in \jty(A^\delta)\ \&\ i\leq \Diamond j\}$. Hence, since $k\in \jty(A^\delta)$ is completely join-prime, $k\leq \Diamond \Diamond j$ iff $k\leq \Diamond i$ for some $i\in \jty(A^\delta)$ s.t.~$i\leq \Diamond j$.

By Definition \ref{def: subordination space in adelta}, the last line of the chain of equivalences above is equivalent to $R_{\prec}$ being transitive.

 The proof of (iii) is argued in a similar way using item 1(iii)  of Proposition \ref{prop:characteriz}, and noticing that the modal inequality characterizing conditions (T) is analytic inductive.
\end{proof}
Likewise,  other items  of Proposition \ref{prop:characteriz} can be used to extend Celani's  results and provide  relational characterizations, on subordination spaces, of conditions (CT), (S9), (SL1), (SL2), noticing that the modal inequalities corresponding to those conditions are all analytic inductive (cf.~\cite[Definition 55]{greco2018unified}).

\begin{proposition}
\label{prop:more dual charact subordination}
For any subordination lattice $\mathbb{S}$,
\begin{enumerate}[label=(\roman*)]
    \item $\mathbb{S}\models \mathrm{(CT)}\quad $ iff $\quad \mathbb{S}_\ast\models\forall i\forall j(j R_\prec i\Rightarrow  \exists k(j \sqsubseteq k\ \&\ j R_\prec k\ \&\  k R_\prec i))$;
    \item $\mathbb{S}\models \mathrm{(S9)} \quad$ iff $\quad \mathbb{S}_\ast\models\forall i\forall j\forall k((kR_{\prec}i \ \&\ kR_{\prec}j)\Leftrightarrow \exists k'(k'\sqsubseteq i\ \&\ k'R_{\prec}j\ \&\ kR_{\prec}k' ))$;
    \item $\mathbb{S}\models \mathrm{(SL1)}\quad $ iff $\quad \mathbb{S}_\ast\models \forall i \forall j \forall k \forall i' (i'R_{\prec} i\ \&\ i' R_{\prec} j\ \&\ kR_{\prec} i' \Rightarrow \exists j' (j'\sqsubseteq i\ \&\ j'\sqsubseteq j\ \&\ kR_{\prec} j') )$;
    \item $\mathbb{S}\models \mathrm{(SL2)}\quad $ iff $\quad\mathbb{S}_\ast\models \forall i \forall j\forall k\forall i' (i'R_{\prec} k \ \&\ iR_{\prec}i' \ \&\ jR_{\prec}i'\Rightarrow \exists j'(j'R_{\prec}k \ \&\ i\sqsubseteq j' \ \&\ j\sqsubseteq j'))$.
\end{enumerate}
\end{proposition}
\begin{proof}
 (i) By Proposition \ref{prop:characteriz}.2(i), $\mathbb{S}\models \mathrm{(CT)}$ iff  $\mathbb{S}^{*}\models \Diamond a\leq \Diamond (a \wedge \Diamond a)$; the inequality $\Diamond a\leq \Diamond (a \wedge \Diamond a)$ is analytic inductive, and hence {\em slanted} canonical by \cite[Theorem 4.1]{de2021slanted}. Hence, from $\mathbb{S}^{*}\models \Diamond a\leq \Diamond (a \wedge \Diamond a)$ it follows that $(\mathbb{S}^{*})^\delta\models \Diamond a\leq \Diamond (a \wedge \Diamond a)$, with $(\mathbb{S}^{*})^\delta$ being a {\em standard} (perfect) distributive modal algebra. By \cite[Theorems 8.1 and 9.8]{conradie2012algorithmic}, $(\mathbb{S}^{*})^\delta\models \Diamond a\leq \Diamond (a \wedge \Diamond a)$ iff $(\mathbb{S}^{*})^\delta\models \forall \nomj ( \Diamond \nomj \leq \Diamond (\nomj \wedge \Diamond \nomj))$ where $\nomj$ ranges in the set $\jty((\mathbb{S}^{*})^\delta)$ of the completely join-irreducible elements of $(\mathbb{S}^{*})^\delta$. Therefore,
    \begin{center}
    \begin{tabular}{cll}
    & $ \forall \nomj (\Diamond \nomj \leq \Diamond (\nomj \wedge \Diamond \nomj)) $ & \\
    iff& $ \forall \nomi \forall \nomj (\nomi \leq \Diamond \nomj \Rightarrow \nomi \leq \Diamond (\nomj \land \Diamond \nomj))$  &
    \\
    iff&$ \forall \nomi \forall \nomj  (\nomi \leq \Diamond \nomj \Rightarrow \exists \nomk(\nomi \leq \Diamond \nomk\ \&\  \nomk \leq \nomj\ \&\  \nomk \leq \Diamond \nomj))$  & $(\ast)$\\
    \end{tabular}
    \end{center}
    The equivalence marked with $(\ast)$ follows from the fact that $\Diamond$ is completely join-preserving (hence monotone) and $\nomi$ is completely join-prime.
By Definition \ref{def: subordination space in adelta}, the last line of the chain of equivalences above translates to the right-hand side of (i).

(ii) By Proposition \ref{prop:characteriz}.3(ii) and (iii), $\mathbb{S}\models \mathrm{(S9)}$ iff  $\mathbb{S}^{*}\models \blacksquare ( a \vee \blacksquare b) = \blacksquare a \vee \blacksquare b$; both the inequalities $\blacksquare ( a \vee \blacksquare b) \leq \blacksquare a \vee \blacksquare b$ and  $\blacksquare ( a \vee \blacksquare b) \geq \blacksquare a \vee \blacksquare b$ are analytic inductive, and hence {\em slanted} canonical by \cite[Theorem 4.1]{de2021slanted}. Hence, from $\mathbb{S}^{*}\models \blacksquare ( a \vee \blacksquare b) = \blacksquare a \vee \blacksquare b$ it follows that $(\mathbb{S}^{*})^\delta\models \blacksquare ( a \vee \blacksquare b) = \blacksquare a \vee \blacksquare b$, with $(\mathbb{S}^{*})^\delta$ being a {\em standard} (perfect) distributive modal algebra. By \cite[Theorems 8.1 and 9.8]{conradie2012algorithmic}, $(\mathbb{S}^{*})^\delta\models \blacksquare ( a \vee \blacksquare b) = \blacksquare a \vee \blacksquare b$ iff $(\mathbb{S}^{*})^\delta\models \forall \cnomm \forall \cnomn (\blacksquare ( \cnomm \vee \blacksquare \cnomn) = \blacksquare \cnomm \vee \blacksquare \cnomn)$ where $\cnomm$ and $ \cnomn$ range in the set $\mty((\mathbb{S}^{*})^\delta)$ of the completely meet-irreducible elements of $(\mathbb{S}^{*})^\delta$. Therefore,
\begin{center}
    \begin{tabular}{cll}
    & $\forall \cnomm \forall \cnomn( \blacksquare ( \cnomm \vee \blacksquare \cnomn) \leq \blacksquare \cnomm \vee \blacksquare \cnomn) $ & \\
  iff  & $\forall \cnomm \forall \cnomn\forall \cnomo (\blacksquare \cnomm \vee \blacksquare \cnomn \leq \cnomo \Rightarrow \blacksquare (\cnomm \vee \blacksquare \cnomn) \leq \cnomo)$ & \\
  iff  & $\forall \cnomm \forall \cnomn\forall \cnomo ((\blacksquare \cnomm \leq \cnomo\ \&\  \blacksquare \cnomn \leq \cnomo) \Rightarrow \exists \cnomo^\prime (\cnomm \leq \cnomo^\prime\ \&\  \blacksquare \cnomn \leq \cnomo^\prime\ \&\  \blacksquare \cnomo^\prime \leq \cnomo) ),$ & ($\ast$) \\
\end{tabular}
\end{center}
where the last equivalence follows from the fact that $\blacksquare$ is completely meet preserving, and $\cnomo$ is completely meet-prime. Likewise, for the converse inequality:

\begin{center}
    \begin{tabular}{cll}
    & $\forall \cnomm \forall \cnomn(\blacksquare \cnomm \vee \blacksquare \cnomn \leq \blacksquare ( \cnomm \vee \blacksquare \cnomn) )$ & \\
   iff  & $\forall \cnomm \forall \cnomn\forall \cnomo ( \exists \cnomo^\prime (\cnomm \leq \cnomo^\prime\ \&\   \blacksquare \cnomn \leq \cnomo^\prime\ \&\   \blacksquare \cnomo^\prime \leq \cnomo) \Rightarrow \blacksquare \cnomm \leq \cnomo\ \&\   \blacksquare \cnomn \leq \cnomo).$ & \\
\end{tabular}
\end{center}

where $\cnomo$ and $\cnomo'$ also range in $\mty((\mathbb{S}^{*})^\delta)$. By Definition \ref{def: subordination space in adelta}, the last lines of the chains of equivalences above translate to the right-hand side of (ii), given that  the assignment $\kappa$ mentioned in Definition \ref{def: subordination space in adelta} is an order-isomorphism 
between  the completely meet-irreducible elements and the completely join-irreducible elements of perfect distributive lattices. 

 (iii) By Proposition \ref{prop:characteriz}.3(i), $\mathbb{S}\models \mathrm{(SL1)}$ iff $\mathbb{S}^{*}\models \blacksquare ( a \vee b) \leq \blacksquare (\blacksquare a \vee \blacksquare b)$; the inequality $\blacksquare ( a \vee b) \leq \blacksquare (\blacksquare a \vee \blacksquare b)$ is analytic inductive, and hence {\em slanted} canonical by \cite[Theorem 4.1]{de2021slanted}. Hence, from $\mathbb{S}^{*}\models \blacksquare ( a \vee b) \leq \blacksquare (\blacksquare a \vee \blacksquare b)$ it follows that $(\mathbb{S}^{*})^\delta\models \blacksquare ( a \vee b) \leq \blacksquare (\blacksquare a \vee \blacksquare b)$, with $(\mathbb{S}^{*})^\delta$ being  a {\em standard} (perfect) distributive modal algebra. By \cite[Theorems 8.1 and 9.8]{conradie2012algorithmic}, $(\mathbb{S}^{*})^\delta\models \blacksquare ( a \vee b) \leq \blacksquare (\blacksquare a \vee \blacksquare b)$ iff $(\mathbb{S}^{*})^\delta\models \forall \cnomm \forall \cnomn (\blacksquare ( \cnomm \vee \cnomn) \leq \blacksquare (\blacksquare \cnomm \vee \blacksquare \cnomn))$ where $\cnomm$ and $ \cnomn$ range in the set $\mty((\mathbb{S}^{*})^\delta)$ of the completely meet-irreducible elements of $(\mathbb{S}^{*})^\delta$. Therefore,
\begin{center}
    \begin{tabular}{cll}
    & $ \forall \cnomm \forall \cnomn(\blacksquare ( \cnomm \vee  \cnomn) \leq \blacksquare (\blacksquare \cnomm \vee \blacksquare \cnomn) )$ & \\
iff    & $\forall \cnomm \forall \cnomn \forall \cnomo (\blacksquare (\blacksquare \cnomm \vee \blacksquare \cnomn) \leq \cnomo \Rightarrow  \blacksquare ( \cnomm \vee  \cnomn) \leq \cnomo) $ & \\
 iff   & $\forall \cnomm \forall \cnomn \forall \cnomo \forall \cnomm' (\blacksquare \cnomm \vee \blacksquare \cnomn \leq \cnomm'\ \&\ \blacksquare \cnomm' \leq \cnomo \Rightarrow \exists \cnomn' ( \cnomm \vee \cnomn \leq \cnomn'\ \&\  \blacksquare \cnomn' \leq \cnomo) )$ &\\
  iff  & $\forall \cnomm \forall \cnomn \forall \cnomo \forall \cnomm' (\blacksquare \cnomm \leq \cnomm'\ \&\ \blacksquare \cnomn \leq \cnomm'\ \&\ \blacksquare \cnomm' \leq \cnomo \Rightarrow \exists \cnomn' ( \cnomm \leq \cnomn'\ \&\ \cnomn \leq \cnomn'\ \&\ \blacksquare \cnomn' \leq \cnomo) ).$ &\\
  i.e. &  $\forall i \forall j \forall k \forall i' (i'R_{\prec} i\ \&\ i' R_{\prec} j\ \&\ kR_{\prec} i' \Rightarrow \exists j' (j'\sqsubseteq i\ \&\ j'\sqsubseteq j\ \&\ kR_{\prec} j') ).$ &\\
 \end{tabular}
    \end{center}

The last line of the chain above is obtained by translating the penultimate line according to    Definition \ref{def: subordination space in adelta}.

 (iv) By Proposition \ref{prop:characteriz}.2(ii), $\mathbb{S}\models \mathrm{(SL2)}$ iff $\mathbb{S}^{*}\models \Diamond(\Diamond a\wedge \Diamond b)\leq \Diamond (a\wedge b)$; the inequality $\Diamond(\Diamond a\wedge \Diamond b)\leq \Diamond (a\wedge b)$ is analytic inductive, and hence {\em slanted} canonical by \cite[Theorem 4.1]{de2021slanted}. Hence, from $\mathbb{S}^{*}\models \Diamond(\Diamond a\wedge \Diamond b)\leq \Diamond (a\wedge b)$ it follows that $(\mathbb{S}^{*})^\delta\models \Diamond(\Diamond a\wedge \Diamond b)\leq \Diamond (a\wedge b)$, with $(\mathbb{S}^{*})^\delta$ being  a {\em standard} (perfect) distributive modal algebra. By \cite[Theorems 8.1 and 9.8]{conradie2012algorithmic}, $(\mathbb{S}^{*})^\delta\models \Diamond(\Diamond a\wedge \Diamond b)\leq \Diamond (a\wedge b)$ iff $(\mathbb{S}^{*})^\delta\models \forall \nomi \forall \nomj (\Diamond(\Diamond \nomi\wedge \Diamond \nomj)\leq \Diamond (\nomi\wedge \nomj))$ where $\nomi$ and $ \nomj$ range in the set $\jty((\mathbb{S}^{*})^\delta)$ of the completely join-irreducible elements of $(\mathbb{S}^{*})^\delta$. Therefore,
\begin{center}
    \begin{tabular}{cll}
    & $\forall \nomi \forall \nomj (\Diamond(\Diamond \nomi\wedge \Diamond \nomj)\leq \Diamond (\nomi\wedge \nomj))$\\
    iff & $\forall \nomi \forall \nomj\forall \nomk (\nomk\leq \Diamond(\Diamond \nomi\wedge \Diamond \nomj)\Rightarrow \nomk\leq \Diamond (\nomi\wedge \nomj))$\\
    iff & $\forall \nomi \forall \nomj\forall \nomk\forall \nomi' (\nomk\leq \Diamond \nomi' \ \&\ \nomi'\leq \Diamond \nomi\wedge \Diamond \nomj\Rightarrow \exists \nomj'(\nomk\leq \Diamond\nomj' \ \&\ \nomj'\leq  \nomi\wedge \nomj))$\\
    iff & $\forall \nomi \forall \nomj\forall \nomk\forall \nomi' (\nomk\leq \Diamond \nomi' \ \&\ \nomi'\leq \Diamond \nomi \ \&\ \nomi'\leq \Diamond \nomj\Rightarrow \exists \nomj'(\nomk\leq \Diamond\nomj' \ \&\ \nomj'\leq  \nomi \ \&\ \nomj'\leq \nomj))$\\
    i.e. & $\forall i \forall j\forall k\forall i' (i'R_{\prec} k \ \&\ iR_{\prec}i' \ \&\ jR_{\prec}i'\Rightarrow \exists j'(j'R_{\prec}k \ \&\ i\sqsubseteq j' \ \&\ j\sqsubseteq j'))$\\
    \end{tabular}
    \end{center}
\end{proof}

\begin{proposition}
\label{prop: precontact space}
For any precontact lattice $\mathbb{C}$,
\begin{enumerate}[label=(\roman*)]
    \item $\mathbb{C}\models \mathrm{(NS)}\quad$ iff  $\quad \mathbb{C}_\ast\models \forall j (j R_{\npcon} j) $

    \item $\mathbb{C}\models\;  \mathrm{(SFN)}$ iff $\quad \mathbb{C}_{*}\models \forall j (j R_{\npcon} j)$.

    \item $\mathbb{C}\models\;  \mathrm{(ALT)}^\wt $ iff $\ \mathbb{C}_*\models \forall j\forall k\forall i(iR_{\npcon}j  \ \&\ iR_{\npcon}k\Rightarrow kR_{\npcon}j) $.

    \item $\mathbb{C}\models  \mathrm{(CMO)} \quad$ iff  $\quad \mathbb{C}_*\models \forall k \forall j \forall i\forall h (kR_{\npcon} h \ \&\ j\sqsubseteq h \ \&\ i\sqsubseteq h \Rightarrow k R_{\npcon} i\ or \ jR_{\npcon} i )$

\end{enumerate}
\end{proposition}

\begin{proof}
(i) By Proposition \ref{prop:precontact correspondence}.1(v), $\mathbb{C}\models \mathrm{(NS)}$ iff
 $\ \mathbb{C}^*\models  a\wedge \wt a\leq \bot$; the inequality $ a\wedge \wt a\leq \bot$ is analytic inductive (cf.~\cite[Definition 55]{greco2018unified}), and hence {\em slanted} canonical (cf.~\cite[Theorem 4.1]{de2021slanted}). Hence, from$\ \mathbb{C}^*\models  a\wedge \wt a\leq \bot$ it follows that $\ (\mathbb{C}^*)^\delta \models  a\wedge \wt a\leq \bot$, where $(\mathbb{C}^{*})^\delta$  is a {\em standard} (perfect) distributive modal algebra. By algorithmic correspondence theory for distributive modal logic (cf.~\cite[Theorems 8.1 and 9.8]{conradie2012algorithmic}), $(\mathbb{C}^{*})^\delta\models a\wedge \wt a\leq \bot$ iff $(\mathbb{C}^{*})^\delta\models \forall \nomj (\nomj \leq \wt \nomj \Rightarrow \nomj \leq \bot)$ or equivalently, $(\mathbb{C}^{*})^\delta\models \forall \nomj (\nomj \nleq \bot\Rightarrow \nomj \nleq \wt \nomj)$, where $\nomj$ ranges in $\jty((\mathbb{C}^{*})^\delta)$. Since every element of $\jty((\mathbb{C}^{*})^\delta)$ is different from $\bot$, by Definition \ref{def: precontact space in adelta} and Footnote \ref{footn:kappa}, the latter condition translates to $\mathbb{C}_{*}\models \forall j (j R_{\npcon} j)$.

 (ii) By Proposition \ref{prop:precontact correspondence}.1(vi), $\mathbb{C}\models \mathrm{(SFN)}$ iff
 $\ \mathbb{C}^*\models  \top \leq a\vee \wt a$; the inequality $\top \leq a\vee \wt a$ is analytic inductive (cf.~\cite[Definition 55]{greco2018unified}), and hence {\em slanted} canonical (cf.~\cite[Theorem 4.1]{de2021slanted}). Hence, from$\ \mathbb{C}^*\models  \top \leq a\vee \wt a$ it follows that $\ (\mathbb{C}^*)^\delta \models  \top \leq a\vee \wt a$, where $(\mathbb{C}^{*})^\delta$  is a {\em standard} (perfect) distributive modal algebra. By algorithmic correspondence theory for distributive modal logic (cf.~\cite[Theorems 8.1 and 9.8]{conradie2012algorithmic}), $(\mathbb{C}^{*})^\delta\models \top \leq a\vee \wt a$ iff $(\mathbb{C}^{*})^\delta\models \forall \cnomm ( \wt \cnomm \leq  \cnomm  \Rightarrow \top\leq \cnomm)$ or equivalently, $(\mathbb{C}^{*})^\delta\models \forall \cnomm (\top\nleq \cnomm \Rightarrow \wt \cnomm \nleq  \cnomm) )$, where $\cnomm$ ranges in $\mty((\mathbb{C}^{*})^\delta)$. Since every element of $\mty((\mathbb{C}^{*})^\delta)$ is different from $\top$, by Definition \ref{def: precontact space in adelta}, Footnote \ref{footn:kappa}, and the fact that $\kappa: \jty(A^\delta)\to \mty(A^\delta)$ is an order-isomorphism, the latter condition translates to $\mathbb{C}_{*}\models \forall j (j R_{\npcon} j)$.

(iii) By Proposition \ref{prop:precontact correspondence}.4(i), $\mathbb{C}\models \mathrm{(ALT)^\wt}$ iff
 $\ \mathbb{C}^*\models  \top \leq \wt a\vee \wt\wt a$; the inequality $\top \leq \wt a\vee \wt\wt a$ is analytic inductive (cf.~\cite[Definition 55]{greco2018unified}), and hence {\em slanted} canonical (cf.~\cite[Theorem 4.1]{de2021slanted}). Hence, from$\ \mathbb{C}^*\models  \top \leq \wt a\vee \wt\wt a $ it follows that $\ (\mathbb{C}^*)^\delta \models  \top \leq \wt a\vee \wt\wt a $, where $(\mathbb{C}^{*})^\delta$  is a {\em standard} (perfect) distributive modal algebra. By algorithmic correspondence theory for distributive modal logic (cf.~\cite[Theorems 8.1 and 9.8]{conradie2012algorithmic}), the following chain of equivalences holds in $A^\delta$:
 \begin{center}
     \begin{tabular}{cll}
        & $\forall a(\top \leq \wt a\vee \wt\wt a )$\\
        iff & $\forall \nomj\forall \nomk\forall a(\nomj\leq a\ \&\ \nomk\leq \wt a\Rightarrow \top\leq \wt \nomj\vee \wt \nomk)$\\
        iff & $\forall \nomj\forall \nomk(\nomk\leq \wt \nomj\Rightarrow \top\leq \wt \nomj\vee \wt \nomk)$\\
        iff & $\forall \nomj\forall \nomk(\nomk\leq \wt \nomj\Rightarrow \forall\cnomm( \wt \nomj\vee \wt \nomk\leq \cnomm\Rightarrow \top\leq \cnomm))$\\
        iff & $\forall \nomj\forall \nomk\forall\cnomm(\nomk\leq \wt \nomj\ \&\  \wt \nomj\vee \wt \nomk\leq \cnomm\Rightarrow \top\leq \cnomm)$\\
        iff & $\forall \nomj\forall \nomk\forall\cnomm(\nomk\leq \wt \nomj\ \&\  \wt \nomj\leq \cnomm\ \&\ \wt \nomk\leq \cnomm\Rightarrow \top\leq \cnomm)$\\
        iff & $\forall \nomj\forall \nomk\forall\cnomm(  \wt \nomj\leq \cnomm\ \&\ \wt \nomk\leq \cnomm\Rightarrow \nomk\nleq \wt \nomj)$ & ($\ast$)\\
        i.e. & $\forall j\forall k\forall i(iR_{\npcon}j  \ \&\ iR_{\npcon}k\Rightarrow kR_{\npcon}j). $ & \\
     \end{tabular}
 \end{center}
 The equivalence marked with $(\ast)$ holds since every element of $\mty((\mathbb{C}^{*})^\delta)$ is different from $\top$, and hence the inequality $\top\leq \cnomm$ is always false. The last line of the chain in display is obtained  by applying Definition \ref{def: precontact space in adelta}, Footnote \ref{footn:kappa}, and the fact that $\kappa: \jty(A^\delta)\to \mty(A^\delta)$ is an order-isomorphism.

(iv) By Proposition \ref{prop:precontact correspondence}.3(i), $\mathbb{C}\models \mathrm{(CMO)}$ iff
 $\ \mathbb{C}^*\models  \wt a\leq \wt( a \wedge \wt a)$; the inequality $  \wt a\leq \wt( a \wedge \wt a)$ is analytic inductive (cf.~\cite[Definition 55]{greco2018unified}), and hence {\em slanted} canonical by \cite[Theorem 4.1]{de2021slanted}. Hence, from$\ \mathbb{C}^*\models  \wt a\leq \wt( a \wedge \wt a)$ it follows that $\ (\mathbb{C}^*)^\delta \models  \wt a\leq \wt( a \wedge \wt a)$, where $(\mathbb{C}^{*})^\delta$  is a {\em standard} (perfect) distributive modal algebra. By algorithmic correspondence theory for distributive modal logic (cf.~\cite[Theorems 8.1 and 9.8]{conradie2012algorithmic}),
 \begin{center}
     \begin{tabular}{cll}
     & $\forall a(\wt a\leq \wt(a \wedge \wt a))$\\
     iff & $\forall \nomk\forall \nomi\forall\nomj\forall a(\nomk\leq\wt a \ \&\ \nomi\leq a \ \&\ \nomj\leq \wt a\Rightarrow \nomk\leq \wt(\nomi\wedge \nomj)) $\\
     iff & $\forall \nomk \forall \nomj \forall \nomi (\nomk \leq \wt \nomi\ \&\ \nomj \leq \wt \nomi  \Rightarrow \nomk \leq \wt(\nomi \wedge \nomj))$\\
     iff & $\forall \nomk \forall \nomj \forall \nomi (\nomk \leq \wt \nomi\ \&\ \nomj \leq \wt \nomi  \Rightarrow \nomi \wedge \nomj\leq \bt \nomk)$\\
     iff & $\forall \nomk \forall \nomj \forall \nomi (\nomk \leq \wt \nomi\ \&\ \nomj \leq \wt \nomi  \Rightarrow \forall\nomh(\nomh\leq \nomj\wedge\nomi \Rightarrow \nomh\leq \bt \nomk))$\\
     iff & $\forall \nomk \forall \nomj \forall \nomi (\nomk \leq \wt \nomi\ \&\ \nomj \leq \wt \nomi  \Rightarrow \forall\nomh(\nomh\leq \nomj\ \&\ \nomh\leq\nomi \Rightarrow \nomh\leq \bt \nomk))$\\
     iff & $\forall \nomk \forall \nomj \forall \nomi\forall\nomh (\nomk \leq \wt \nomi\ \&\ \nomj \leq \wt \nomi  \ \&\ \nomh\leq \nomj \ \&\ \nomh\leq \nomi \Rightarrow \nomh\leq \bt \nomk)$\\
     iff & $\forall \nomk \forall \nomj \forall \nomi\forall\nomh (\nomh\nleq \bt \nomk \ \&\ \nomh\leq \nomj \ \&\ \nomh\leq \nomi \Rightarrow \nomk \nleq \wt \nomi\ or \ \nomj \nleq \wt \nomi )$\\
     iff & $\forall \nomk \forall \nomj \forall \nomi\forall\nomh ( \bt \nomk\leq \kappa(\nomh) \ \&\ \nomh\leq \nomj \ \&\ \nomh\leq \nomi \Rightarrow \wt \nomi\leq\kappa(\nomk)\ or \ \wt \nomi\leq \kappa(\nomj) )$\\
     i.e. & $\forall k \forall j \forall i\forall h (kR_{\npcon} h \ \&\ j\sqsubseteq h \ \&\ i\sqsubseteq h \Rightarrow k R_{\npcon} i\ or \ jR_{\npcon} i )$,\\
     \end{tabular}
 \end{center}
  where $\nomj,\nomk,\nomi, \nomh$ range in  $\jty((\mathbb{C}^{*})^\delta)$, and the last steps make use of Footnote \ref{footn:kappa} and Definition \ref{def: precontact space in adelta}.
\end{proof}
Conditions on dual precontact lattices, or on distributive lattice-based bi-subordination algebras or sp-algebras, can be dualized on the corresponding spaces by analogous arguments which make use of the remaining characterization results in Propositions \ref{prop:precontact correspondence}, \ref{prop:postcontact correspondence}, \ref{prop:poly correspondence}, and \ref{prop:poly correspondence permission}. A general characterization result encompassing all the instances presented above will be discussed in the companion paper \cite{dedomenico2024correspond}.

\subsection{On positive bi-subordination algebras and quasi-modal operators}
\label{ssec:dunn-saxioms}
In \cite{celani2021bounded}, {\em positive bi-subordination lattices} are introduced as those DL-based bi-subordination algebras $\mathbb{B} = (A, \prec_1, \prec_2)$ satisfying conditions (P1) and (P2) (cf.~Section \ref{sec: modal charact}, before Proposition \ref{prop:poly correspondence}). In the same paper, this definition is  motivated by the statement that the link between $\prec_1$ and $\prec_2$ in positive bi-subordination algebras is `similar' to the link
between the $\Box$ and $\Diamond$ operations in positive modal algebras \cite{dunn1995positive}. To  substantiate this statement, the authors state (without proof) that conditions (P1) and (P2) `say' that the following inclusions hold involving the quasi-modal operators associated with $\prec_1$ and $\prec_2$ for every $a, b\in A$:
\begin{equation}
\label{eq: clumsy dunn}
\Delta_{\prec_1}(a\vee b)\subseteq \nabla_{\prec_2} a\odot \Delta_{\prec_1}b \quad\quad
 \nabla_{\prec_2}(a \wedge b) \subseteq \Delta_{\prec_1}a\oplus \nabla_{\prec_2}b,
\end{equation}
where the quasi-modal operator  $\Delta_{\prec_1}$ (resp.$\nabla_{\prec_2}$) associates each $a\in A$ with the ideal ${\prec_1^{-1}}[a]$, (resp.~with the filter ${\prec_2}[a]$), and moreover, for any filter $F$ and ideal $I$ of $A$, the ideal $F \odot I$ and the filter $I \oplus F$ are defined as follows:
\[
F \odot I :=\bigcap\{\lceil I\cup\{a\}\rceil\mid a\in F\} \quad\quad  I \oplus F :=\bigcap\{\lfloor F\cup\{a\}\rfloor\mid a\in I\}.
\]
Having clarified (cf.~Proposition \ref{prop:poly correspondence} 3 (i) and (ii)) in which sense  the link between $\prec_1$ and $\prec_2$ encoded in (P1) and (P2) is `similar' to the link
between  $\Box$ and $\Diamond$ encoded in the interaction axioms in positive modal algebras,
in the present subsection, we clarify the link between (P1) and (P2) and  the inclusions \eqref{eq: clumsy dunn}, by showing, again as an application of
 Proposition \ref{prop:poly correspondence}, that these inclusions  equivalently characterize  (P1) and (P2).

 \smallskip
 Via the identification of any filter $F$  of $A$ with the closed element $k = \bigwedge F \in K(A^\delta)$ and of any ideal $I$ of $A$ with the open element $o = \bigvee I\in O(A^\delta)$, the ideal $F\odot I$ (resp.~the filter $I\oplus F$) can be  identified with the following open (resp.~closed) element of $A^\delta$:
\[
k \odot o :=\bigvee\{b\in A\mid \forall a(k\leq a\Rightarrow b\leq a\vee o)\} \quad\quad  o \oplus k :=\bigwedge\{b\in A\mid \forall a(a\leq o \Rightarrow  a\wedge k\leq b)\}.
\]
Notice  that \[k\odot o = int(k\vee o): = \bigvee \{o'\in O(A^\delta)\mid o'\leq k\vee o\},\] and \[o\oplus k = cl(o\wedge k): = \bigwedge \{k'\in K(A^\delta)\mid o\wedge k\leq k'\},\]
where $int(k\vee o)$ (i.e.~the {\em interior} of $k\vee o$) is defined as the greatest open element of $A^\delta$ which is smaller than or equal to $k\vee o$, and dually, $cl(o\wedge k)$ (i.e.~the {\em closure} of $o\wedge k$) is defined as the smallest closed element of $A^\delta$ which is greater than or equal to $o\wedge k$.  Indeed,
\begin{center}
\begin{tabular}{rcl}
$k \odot o $ & $=$ & $\bigvee\{b\in A\mid \forall a(k\leq a\Rightarrow b\leq a\vee o)\}$\\
& $=$ & $\bigvee\{b\in A\mid b\leq \bigwedge \{a\vee o\mid k\leq a\}\}$\\
& $=$ & $\bigvee\{b\in A\mid b\leq o\vee \bigwedge \{a\mid k\leq a\}\}$\\
& $=$ & $\bigvee\{b\in A\mid b\leq o\vee  k\}$\\
& $=$ & $\bigvee \{o'\in O(A^\delta)\mid o'\leq k\vee o\}$\\
& $=$ & $int(k\vee o)$,\\
\end{tabular}
\end{center}
and likewise one shows that $o\oplus k = cl(o\wedge k)$.
Hence, the inclusions \eqref{eq: clumsy dunn} equivalently translate as the following inequalities on $A^\delta$:
\begin{equation}
\label{eq: clumsy dunn translated}
\blacksquare_{1}(a\vee b)\leq int( \Diamond_{2} a\vee \blacksquare_{1}b) \quad\quad
  cl(\blacksquare_{1}a\wedge \Diamond_{2}b)\leq \Diamond_{2}(a \wedge b),
\end{equation}
and since $\blacksquare_{1}(a\vee b)\in O(A^\delta)$ and $\Diamond_{2}(a \wedge b)\in K(A^\delta)$, the inequalities in \eqref{eq: clumsy dunn translated} are respectively equivalent to the following inequalities
\begin{equation}
\label{eq: dunn }
\blacksquare_{1}(a\vee b)\leq  \Diamond_{2} a\vee \blacksquare_{1}b \quad\quad
  \blacksquare_{1}a\wedge \Diamond_{2}b\leq \Diamond_{2}(a \wedge b),
\end{equation}
which, in Proposition \ref{prop:poly correspondence}, are shown to be equivalent to (P1) and (P2), respectively.

Finally, notice that the equivalence between \eqref{eq: clumsy dunn translated} and \eqref{eq: dunn } hinges on the fact that the left-hand side (resp.~right-hand side) of the first (resp.~second) inequality is an open (resp.~closed) element of $A^\delta$.
Analogous equivalences  would not hold for inequalities such as $\Diamond_1\blacksquare_2  a\leq \blacksquare_2\Diamond_1 a $ (cf.~Proposition \ref{prop:poly correspondence} 4), in which each  side is neither closed nor open. Hence, we would not be able to equivalently characterize conditions on bi-subordination algebras such as $\prec_2  \circ\prec_1\ \subseteq\ \prec_1  \circ\prec_2$ (cf.~Proposition \ref{prop:poly correspondence} 4) as  {\em inclusions} of filters and ideals in the language of the quasi-modal operators associated with $\prec_1$ and $\prec_2$ and the term-constructors $\odot$ and $\oplus$.
\section{Conclusions}
\label{sec:conclusions}

\paragraph{Contributions of the present paper.} The present paper expands on the research initiated in \cite{wollic22}, in which a novel connection is established between the research themes of subordination algebras and of input/output logic, via notions and insights stemming from  algebraic logic. Building on \cite{dedomenico2024obligations}, in which various notions of normative and permission systems, originally introduced in \cite{Makinson00,Makinson03}, have been  introduced and  studied in the general context of the class of selfextensional logics \cite{font2017general},  the present paper focuses on various notions of algebras endowed with binary relations (collectively referred to in the present paper  as `relational algebras') and uses them as  semantic environment for normative and permission systems. 
 In particular, we have   studied  the notions of (dual) negative permission in the environment of (dual proto) precontact algebras \cite{dimov2005topological},  algebraic structures related to subordination algebras, and introduced in the context of a research program aimed at investigating spatial reasoning in a way that takes regions rather than points as the basic notion. This algebraic perspective has benefited from the insight, developed within the theory of unified correspondence \cite{conradie2014unified, de2021slanted}, that  relational algebras endowed with certain basic properties can be associated with  {\em slanted algebras}, i.e.~a generalized type of modal algebras, the modal operators of which  are not operations of the given  algebra, but are maps from the given algebra to its canonical extension.   

The environment of slanted algebras allows us to achieve  correspondence-theoretic  characterizations of several well known conditions on the various relational algebras (some of which are the algebraic counterparts of well known closure properties in input/output logic, while others have been considered in  research lines related to subordination algebras) in terms of the validity of certain modal axioms on their associated slanted algebra. Several applications of these results ensue,
  spanning from the characterization of various types of output operators arising from normative and permission systems (some of which novel, to our knowledge) in terms of properties of their associated slanted modal operators, to a shorter and more streamlined proof of Celani's dual characterization results for subordination algebras \cite{celani2020subordination}, which we can obtain as consequences of the general correspondence theory on slanted algebras. These diverse applications  are hence  relevant both to input/output logic and to subordination algebras.   
The modal characterization results of Section \ref{sec: modal charact}
are different from other modal formulations of input/output logic \cite{Makinson00,Parent2021,strasser2016adaptive}, including those set on a non-classical propositional base, in that the output operators themselves are semantically characterized as (suitable restrictions of) modal operators, and their properties characterized in terms of modal axioms (inequalities). Finally, Propositions \ref{prop:poly correspondence} and \ref{prop:poly correspondence permission}, and Section \ref{ssec:algebraizing static permission} illustrate how these modal characterizations   straightforwardly extend also to conditions describing the {\em interaction} between obligations and various kinds of permissions.   

\paragraph{Characterization results.}
The modal characterization results of Section \ref{sec: modal charact}, and hence the dual characterization  results of Section \ref{ssec: dual
 charact}, cover a finite number of  conditions. A natural direction is to generalize these results to infinite syntactic classes of conditions on relational algebras of various types. This is the focus of the companion paper \cite{dedomenico2024correspond}, currently in preparation.

\paragraph{Algebraizing dynamic positive permissions.} In \cite{dedomenico2024obligations},  the notion of dynamic positive permissions have been generalized to the environment of selfextensional logics. A natural application of the theory developed in the present paper is to achieve the algebraization of this notion, analogously to the algebraization of static positive permissions discussed in Section \ref{ssec:algebraizing static permission}. More in general, a natural future direction is to use the algebraic logical setting introduced in the present paper to explore various notions of  coherent interaction between obligations and permissions such as those recently discussed in \cite{olszewski2023permissive}.
\paragraph{Conceptual interpretations of input/output logic and relational algebras.}
As discussed earlier on,  the two research areas of input/output logic and relational algebras have been developed independently of each other, pivoting on different formal tools and insights and motivated in very different ways.
The bridge between these two areas which the present results contribute to develop can be used to cross-fertilize the two areas, not only from a technical viewpoint (e.g.~to import
 mathematical tools such as topological, algebraic, and duality-theoretic techniques in the formal study of normative reasoning), but also  to find conceptual interpretations and applications for  relational algebras. 

\paragraph{ Computational implementations.} The LogiKEy methodology \cite{J48} was introduced to design and engineer ethical and legal reasoners, as well as responsible AI systems. LogiKEy’s unifying formal
framework is based on the semantic embedding of various logical formalisms, including deontic logic,  into expressive classic higher-order logic,  which enables the use of off-the-shelf theorem provers. While this methodology proved successful for several  deontic logics, 
treating input/output logic in the same way presented a challenge, due to its operational semantics. It would be interesting to see whether this difficulty can be overcome thanks to the systematic connection with various modal logic languages established in the present paper. 

\appendix
\section{Selfextensional logics
with disjunction}
\label{sec: selfextensional w disj}
In the present section, we adapt some of the results in \cite{Jansana2006conjunction} to the setting of selfextensional logics with disjunction.   Throughout the present section, we fix an arbitrary similarity type $\tau$ which is common to  the logics and the algebras we consider.

\begin{definition}
    \label{def: disjunction of logic}
    For any logic $\mathcal{L} =(\mathrm{Fm}, \vdash)$,
     a binary term $t(x, y): = x\vee y$ is a \emph{disjunction} of $\mathcal{L}$ if, for all $\varphi, \psi, \chi\in\mathrm{Fm}$'
     \begin{center}
        $\varphi\vdash \chi$ and $\psi\vdash \chi\quad $ if and only if $\quad \varphi\vee \psi \vdash \chi$.
     \end{center}
If so, we also say that the logic is {\em disjunctive}.
\end{definition}
\begin{proposition}
\label{prop:idempotence assoc comm of vee}
    If $\mathcal{L}$ is a disjunctive logic, then for any $\varphi, \psi, \chi\in \mathrm{Fm}$,
    \[\varphi\dashv\vdash \varphi\vee \varphi\quad\quad  \varphi\vee (\psi\vee\chi)\dashv\vdash (\varphi\vee \psi)\vee \chi \quad\quad  \varphi\vee\psi\dashv\vdash \psi\vee \varphi\]
\end{proposition}
\begin{proof}
    Instantiating (a) with $\psi: = \varphi$ yields $\varphi\vdash \varphi\vee \varphi$; instantiating (b) with $\chi: =\psi: = \varphi$ yields $\varphi\vee \varphi\vdash \varphi$. By (b), to show that $\varphi\vee (\psi\vee\chi)\vdash (\varphi\vee \psi)\vee \chi$,  it is enough to show that $\varphi\vdash (\varphi\vee \psi)\vee \chi$ and $ \psi\vee\chi\vdash (\varphi\vee \psi)\vee \chi$, and, again by (b), to show the latter entailment, it is enough to show that $ \psi\vdash (\varphi\vee \psi)\vee \chi$ and $ \chi\vdash (\varphi\vee \psi)\vee \chi$. By (a), $\varphi\vdash \varphi\vee \psi$, and $\varphi\vee \psi\vdash (\varphi\vee \psi)\vee \chi$, hence $\varphi \vdash (\varphi\vee \psi)\vee \chi$. The remaining entailments are shown similarly.
\end{proof}
\begin{definition}
\label{def: algebra semilattice based}

A class of algebras $\mathsf{K}$  is \emph{$\vee$-semilattice based} if a binary term $t(x, y): = x\vee y$ exists such that the following identities are valid in  $\mathsf{K}$:
    \begin{itemize}
        \item[L1] $x\vee x=x$.
        \item[L2] $x\vee(y\vee z)=(x\vee y)\vee z$.
        \item[L3] $x\vee y=y\vee x$.
    \end{itemize}
If so, we also say that $\mathsf{K}$ is join-semilattice based relative to $\vee$ or that $\mathsf{K}$ is a semilattice class relative to $\vee$.
\end{definition}

If $\mathsf{K}$ is join-semilattice based relative to $\vee$, then 
the following partial order can be defined
 on every algebra $\mathbb{A}\in \mathsf{K}$: for all $a,b\in\mathbb{A}$,
\begin{center}
 $a\leq^\mathbb{A} b\quad $ iff  $\quad a\vee^\mathbb{A} b=b$
\end{center}
 We will omit the superscript in $\leq^\mathbb{A}$ and $\vee^\mathbb{A}$ when no confusion is likely to arise.

\begin{definition}
\label{def: logic semilattice based}
     A logic $\mathcal{L}=(\mathrm{Fm}, \vdash)$ is \emph{$\vee$-semilattice based} if    a class of algebras $\mathsf{K}$ exists which is join-semilattice based relative to some term $\vee$, and is such that, for all $\varphi, \psi, \chi\in \mathrm{Fm}$,
\begin{center}
     $\varphi\vdash\chi$ and $\psi\vdash\chi\quad $ iff  $\quad \forall\mathbb{A}\in \mathsf{K}\  \forall h\in Hom(\mathrm{Fm}, \mathbb{A})$, $h(\varphi)\vee h(\psi)\leq^\mathbb{A}h(\chi)$.
\end{center}
If so, we say that  $\mathcal{L}$ is {\em join-semilattice based relative to $\vee$ and $\mathsf{K}$}.
\end{definition}
Let  $V(\mathsf{K})$ denote the variety generated by $\mathsf{K}$.\footnote{A {\em variety} (cf.~\cite{burris1981course}) is a class of algebras  defined by equations; $V(\mathsf{K})$ is the class of algebras defined by all equations holding in $\mathsf{K}$, and is equivalently characterized as $HSP(\mathsf{K})$, that is, every $\mathbb{B}\in V(\mathsf{K})$ is the homomorphic image of a subalgebra of a direct product of algebras in $\mathsf{K}$.
}
\begin{proposition}\label{col: semilattice relative to V(K)}
    If  $\mathcal{L}$ is join-semilattice based relative to $\vee$ and $\mathsf{K}$, then $\mathcal{L}$ is also join-semilattice based relative to $\vee$ and $V(\mathsf{K})$.
\end{proposition}
\begin{proof}
     We need to show that, for all $\varphi,\psi,\chi\in \mathrm{Fm}$, \begin{center} $\varphi\vdash\chi\ $ and $\ \psi\vdash\chi\quad $ iff  $\quad\forall \mathbb{B}\in V(\mathsf{K})$  $\forall h\in Hom(\mathrm{Fm},\mathbb{B})$, $h(\varphi)\vee h(\psi)\leq^\mathbb{B} h(\chi)$.
     \end{center}
     The right-to-left direction immediately follows from the fact that  $\mathsf{K}\subseteq V(\mathsf{K})$, and $\mathcal{L}$ is join-semilattice based relative to $\vee$ and $\mathsf{K}$. For the left-to-right direction, let $\varphi, \psi, \chi\in \mathrm{Fm}$, and assume contrapositively that $h(\varphi)\vee h(\psi)\nleq^\mathbb{B} h(\chi)$ for some  
     algebra $\mathbb{B}\in V(\mathsf{K})$ and some $h\in Hom(\mathrm{Fm}, \mathbb{B})$. Since  $\mathbb{B}\in V(\mathsf{K})$, 
     \[\Pi_{\mathbb{A}\in\mathsf{K}'}\mathbb{A}\hookleftarrow \mathbb{C}\twoheadrightarrow \mathbb{B}\]
     a surjective homomorphism $k: \mathbb{C}\to \mathbb{B}$ exists from an algebra $\mathbb{C}$ for which an injective homomorphism $e: \mathbb{C}\to \Pi_{\mathbb{A}\in \mathsf{K}'}\mathbb{A}$ exists s.t.~$\mathsf{K}'\subseteq \mathsf{K}$.
     Let $k^\ast\in Hom(\mathrm{Fm}, \mathbb{C})$ be uniquely defined by the assignment $p\mapsto c$ for some $c\in\mathbb{C}$ s.t.~$k(c) = h(p)$. Hence, $h = k\circ k^\ast$, and so it must be $k^\ast(\varphi)\vee k^\ast(\psi)\nleq^\mathbb{C} k^\ast(\chi)$, for if not, then $h(\varphi)\vee h(\psi) = k (k^\ast (\varphi))\vee^\mathbb{A} k (k^\ast (\psi)) = k(k^\ast (\varphi)\vee^\mathbb{C} k^\ast (\psi))\leq^\mathbb{A}k(k^\ast(\chi)) = h(\chi)$, against the assumption.
     Let  $h^\ast\in Hom(\mathrm{Fm}, \Pi_{\mathbb{A}\in \mathsf{K}'}\mathbb{A})$ be uniquely defined by the assignment $p\mapsto e(k^\ast(p))$ for every $p\in\mathsf{Prop}$. Hence, $h^\ast = e\circ k^\ast$, and since $e$ is an injective homomorphism,  $k^\ast(\varphi)\vee k^\ast(\psi)\nleq^\mathbb{C} k^\ast(\chi)$ implies that $h^\ast(\varphi)\vee h^\ast(\psi)\nleq^{\Pi_{\mathbb{A}\in \mathsf{K}'}\mathbb{A}} h^\ast(\chi)$. For any algebra $\mathbb{A}$ in $\mathsf{K}'$ let $h' = \pi\circ h^\ast\in Hom (\mathrm{Fm}, \mathbb{A})$, where $\pi: \Pi_{\mathbb{A}\in \mathsf{K}'}\mathbb{A}\to \mathbb{A}$ is the canonical projection. Since equalities  in product algebras are defined coordinatewise, $h^\ast(\varphi)\vee h^\ast(\psi)\nleq^{\Pi_{\mathbb{A}\in \mathsf{K}'}\mathbb{A}} h^\ast(\chi)$ implies that   $h'(\varphi)\vee h'(\psi)\nleq^\mathbb{A} h'(\chi)$ for some
     $\mathbb{A}$ in $\mathrm{K}'\subseteq \mathsf{K}$, as required.
\end{proof}

\begin{proposition}\label{prop: interdiravable}
    If $\mathcal{L} = (\mathrm{Fm}, \vdash)$ is join-semilattice based relative to $\vee$ and $\mathsf{K}$, then, for any $\varphi, \psi\in \mathrm{Fm}$,
    \begin{enumerate}

        \item $\varphi\vdash\psi\quad $ iff $\quad \forall \mathbb{A}\in\mathsf{K}\ \forall h\in Hom(\mathrm{Fm},\mathbb{A} ), h(\varphi)\leq^\mathbb{A} h(\psi)$.
        \item $\varphi\dashv \vdash\psi\quad $ iff $\quad\mathsf{K}\models \varphi = \psi$.
        \item For any $\mathbb{A}\in\mathsf{K}$ and any $F\in Fi_{\mathcal{L}}(\mathbb{A})$, $F$ is  upward-closed w.r.t.~$\leq^\mathbb{A}$.

    \end{enumerate}
\end{proposition}

\begin{proof}
1. For the  left-to-right direction, let $\varphi, \psi\in\mathrm{Fm}$ s.t.~$\varphi\vdash\psi$. Hence, $\varphi\vdash\psi$ and $\psi\vdash\psi$, which implies, since $\mathcal{L}$ is join-semilattice based relative to $\mathsf{K}$ and $\vee$, that   $h(\varphi)\vee h(\psi)\leq^\mathbb{A} h(\psi)$ for any $\mathbb{A}\in \mathsf{K}$ and any $h\in Hom( \mathrm{Fm}, \mathbb{A})$. By the definition of $\leq^\mathbb{A}$, this is equivalent to  
$h(\varphi)\leq^\mathbb{A} h(\psi)$ for every $\mathbb{A}\in \mathsf{K}$ and any $h\in Hom( \mathrm{Fm}, \mathbb{A})$, as required.

Conversely, since $\mathcal{L}$ is join-semilattice based  relative to $\vee$ and $\mathsf{K}$, it is enough to show that, if $\varphi, \psi\in \mathrm{Fm}$ s.t.~$\mathsf{K}\vDash \varphi \leq \psi$, $\mathbb{A}\in \mathsf{K}$ and  $h\in Hom(\mathrm{Fm}, \mathbb{A})$, then $h(\varphi)\vee h(\psi)\leq^\mathbb{A} h(\psi)$. 
The assumption that  $\mathsf{K}\vDash \varphi \leq \psi$ implies that $h(\varphi)\leq^\mathbb{A} h(\psi)$, i.e.~$h(\varphi)\vee h(\psi)=h(\psi)$. Hence, $h(\varphi)\vee h(\psi)\leq^\mathbb{A} h(\psi)$, as required.

2. Immediate from the previous item.

3. Let $a, b\in \mathbb{A}$ s.t.~$a\leq^\mathbb{A}b$ and $a\in F$. By definition, $a\leq^\mathbb{A} b$ iff $a\vee b = b$. Let $p, q\in \mathsf{Prop}$, and let $h\in Hom(\mathsf{Fm}, \mathbb{A})$ s.t.~$h(p) = a$ and $h(q) = b$. Hence, $h(p\vee q) = h(p)\vee h(q) = a\vee b = b$. Moreover, since $F$ is an $\mathcal{L}$-filter, $p\vdash p\vee q$ and $h(p) = a\in F$ imply that  $b = h(p\vee q) \in F$, as required.
\end{proof}

\begin{corollary}
\label{cor: only one variety}
    If $\mathcal{L} = (\mathrm{Fm}, \vdash)$ is join-semilattice based relative to $\vee$ and $\mathsf{K}$, and  is join-semilattice-based relative to $\vee^{'} $ and $\mathsf{K}'$, then  $V(\mathsf{K}) =  V(\mathsf{K}')$.
\end{corollary}
\begin{proof}\
    \begin{center}
    \begin{tabular}{clll}
    $V(\mathsf{K})\models \varphi = \psi$  &  iff   & $\mathsf{K}\models\varphi = \psi$& definition of generated variety\\
     & iff   & $\varphi \dashv\vdash \psi$&Proposition \ref{prop: interdiravable}.2 \\
     & iff   & $\mathsf{K}'\models \varphi = \psi$& Proposition \ref{prop: interdiravable}.2\\
     & iff   & $V(\mathsf{K}')\models \varphi = \psi$& definition of generated variety
    \end{tabular}
     \end{center}
\end{proof}
We let $V(\mathcal{L})$ denote the only variety relative to which $\mathcal{L}$ is semilattice-based.
\begin{proposition}
\label{prop: semilattice based implies selfextensional}
    If $\mathcal{L}$ is join-semilattice-based relative to $\vee$, then
    \begin{enumerate}
        \item $\mathcal{L}$ is selfextensional,
        \item $\vee $ is a disjunction of $\mathcal{L}$,
        \item If $\mathcal{L}$ is also join-semilattice-based relative to  $\vee'$, then $\varphi\vee\psi \dashv\vdash \varphi \vee' \psi$ for all $\varphi, \psi\in \mathrm{Fm}$.

    \end{enumerate}
\end{proposition}

\begin{proof}
    \begin{enumerate}
        \item \ \begin{center}
    \begin{tabular}{clll}
    $\varphi \dashv\vdash \psi$ &  iff   & $V(\mathcal{L})\models \varphi = \psi$  & Proposition \ref{prop: interdiravable}.2 \\
    & implies   & $V(\mathcal{L})\models \delta [\varphi/p] = \delta[\psi/p]$ for every $\delta\in\mathrm{Fm}$\\

     & iff   & $\delta [\varphi/p] \dashv\vdash \delta[\psi/p]$ for every $\delta\in\mathrm{Fm}$ & Proposition \ref{prop: interdiravable}.2\\
    \end{tabular}
     \end{center}

        \item Let $\varphi,\psi, \chi\in \mathrm{Fm}$. By Definition \ref{def: logic semilattice based}, $\varphi\vdash\varphi\vee\psi$ and $\psi\vdash\varphi\vee\psi$ iff $h(\varphi)\vee^{\mathbb{A}}h(\psi)\leq h(\varphi \vee \psi) = h(\varphi) \vee^{\mathbb{A}} h(\psi)$ for any $\mathbb{A}\in V(\mathcal{L})$ and any $h\in Hom(\mathrm{Fm}, \mathbb{A})$, which is always the case. By Definition \ref{def: logic semilattice based}, $\varphi\vdash\chi$ and $\psi\vdash\chi$ iff $h(\varphi\vee\psi) = h(\varphi)\vee^{\mathbb{A}}h(\psi)\leq h(\chi)$ for any $\mathbb{A}\in V(\mathcal{L})$ and any $h\in Hom(\mathrm{Fm}, \mathbb{A})$,  which, by Proposition \ref{prop: interdiravable}.1, is equivalent to   $\varphi\vee\psi\vdash\chi$, as required.
        \item The assumptions imply, by item 2, that both $\vee$ and $\vee'$ are disjunctions. Let $\varphi, \psi\in\mathrm{Fm}$. Applying (a) to $\vee$ yields  $\varphi \vdash \varphi \vee \psi$ and $\psi \vdash \varphi \vee \psi$, which, applying (b) to $\vee'$, is equivalent to $\varphi\vee^{'}\psi \vdash \varphi \vee \psi$. The converse entailment is shown by the same argument,  inverting the roles of $\vee$ and $\vee'$.
    \end{enumerate}
\end{proof}
By item 3 of the proposition above,  we can say that $\mathcal{L}$ is join-semilattice-based in an absolute sense.

\begin{proposition}
\label{prop:charact self and disj}
    A logic $\mathcal{L}$ is selfextensional and disjunctive iff
it is join-semilattice based.
\end{proposition}

\begin{proof}
    The right-to-left direction is the content of items 1 and 2 of Proposition \ref{prop: semilattice based implies selfextensional}. As to the left-to-right direction, let us show that $\mathcal{L}$ is join-semilattice based relative to $\vee$ and $\mathsf{K}: = \{Fm\}$, where $Fm: = \mathrm{Fm}/{\equiv}$ (cf.~Section \ref{ssec:selfextensional}).
    Let us verify that $\mathsf{K}$ is $\vee$-semilattice based: for all $\varphi, \psi, \chi\in \mathrm{Fm}$, let $[\varphi], [\psi], [\chi]$ be the corresponding elements ($\equiv$-equivalence classes) in $Fm$. The following identities hold thanks to  the selfextensionality of $\mathcal{L}$ and Proposition \ref{prop:idempotence assoc comm of vee}:
\begin{enumerate}
    \item $[\varphi]\vee^{Fm}[\varphi]=[\varphi \vee \varphi]=[\varphi]$,
    \item $[\varphi]\vee^{Fm}([\psi]\vee^{Fm}[\chi])=[\varphi \vee (\psi \vee \chi)]=[(\varphi \vee \psi) \vee \chi]=([\varphi]\vee^{Fm}[\psi])\vee^{Fm}[\chi]$,
    \item $[\varphi]\vee^{Fm}[\psi]=[\varphi \vee \psi]=[\psi \vee \varphi]=[\psi]\vee^{Fm}[\varphi]$.
\end{enumerate}
Then, for any $\varphi, \psi\in \mathrm{Fm}$,
\begin{equation}
\label{eq: leq in LT algebra}[\varphi]\leq^{Fm}[\psi] \quad \text{ iff }\quad [\varphi\vee \psi] = [\varphi]\vee^{Fm}[\psi] = [\psi] \quad \text{ iff }\quad \varphi\vee\psi\dashv\vdash \psi\quad \text{ iff }\quad \varphi\vdash \psi.\end{equation}
    To finish the proof, we need to show that, for all $\varphi, \psi, \chi\in \mathrm{Fm}$,
    \[\varphi\vdash\chi \ \text { and } \ \psi\vdash\chi\quad \text{ iff } \quad  \forall h\in Hom(\mathrm{Fm}, Fm), \ h(\varphi) \vee^{Fm} h(\psi)\leq^{Fm} h(\chi).\]
Notice preliminarily that
    \begin{center}
    \begin{tabular}{r c ll}
    $\varphi\vdash\chi \ \text { and } \ \psi\vdash\chi$ & iff & $\varphi\vee\psi\vdash \chi$\\
    & iff & $[\varphi\vee\psi]\leq^{Fm}[\chi]$. & \eqref{eq: leq in LT algebra} \\
    \end{tabular}
    \end{center}
 As to the right-to-left direction,  if $h: \mathrm{Fm}\to Fm$ denotes the homomorphism uniquely defined by the assignment $p\mapsto [p]$, then, since $\mathcal{L}$ is selfextensional, $h(\varphi) = [\varphi]$ for every $\varphi\in \mathrm{Fm}$, and so $[\varphi\vee \psi] = h(\varphi\vee \psi) = h(\varphi)\vee^{Fm} h(\psi)  \leq^{Fm} h(\chi) = [\chi]$,  which entails the required statement by the preliminary observation above. As to the left-to-right direction, let $h\in Hom(\mathrm{Fm}, Fm)$. Then $h(\varphi) =  [\sigma(\varphi)]$, where $\sigma\in Hom(\mathrm{Fm}, \mathrm{Fm})$ is any substitution s.t.~$\sigma(p)\in h(p)$  for every $p\in \mathsf{Prop}$. Since $\vdash$ is structural, $\varphi\vdash \chi$ and $\psi\vdash \chi$ imply that $\sigma(\varphi)\vdash \sigma(\chi)$ and $\sigma(\psi)\vdash \sigma(\chi)$, i.e., by the preliminary observation above, $[\sigma(\varphi)\vee\sigma(\psi)]\leq^{Fm}[\sigma(\chi)]$. Hence,
 \[h(\varphi)\vee^{Fm} h(\psi) = [\sigma(\varphi)]\vee^{Fm}[\sigma(\psi)] = [\sigma(\varphi)\vee \sigma(\psi)]  \leq^{Fm}  [\sigma(\chi)] = h(\chi),\]
 as required.
\end{proof}

Recall that the intrinsic variety $\mathsf{K}_{\mathcal{L}}$ is the variety generated by $\mathsf{K}: = \{Fm\}$ (cf.~Section \ref{ssec:selfextensional}). Hence, from the proof of the left-to-right direction of the proposition above, Proposition \ref{col: semilattice relative to V(K)}, and Corollary \ref{cor: only one variety}, it immediately follows that
\begin{proposition}
     If $\mathcal{L}$ is join-semilattice-based, then   $V(\mathcal{L}) = \mathsf{K}_{\mathcal{L}} $. 
\end{proposition}

\paragraph{Supercompact selfextensional logics with disjunction.} A logic $\mathcal{L}$
 is {\em super-compact} if for any $\Gamma\cup\{\varphi\}\subseteq \mathrm{Fm}$, if $\Gamma\vdash \varphi$, then $\gamma\vdash\varphi$ for some $\gamma\in \Gamma$.

 By definition, super-compact logics are those the associated closure operator of which is {\em unitary}, i.e.~such that $Cn(\Gamma) = \bigcup\{Cn(\gamma)\mid \gamma\in \Gamma\}$ for any $\Gamma\subseteq \mathrm{Fm}$. Super-compact logics are not   uncommon; for instance, this property holds for the $\{\wedge\}$-free fragment of positive modal logic. More in general, logics  with no conjunction and the consequence relation of which is generated by a (finite) set of sequents $\varphi\vdash\psi$ where the antecedent is a singleton set are super-compact.

\begin{proposition}
\label{prop:supercompact and logical filters}
    For any   join-semilattice  based logic $\mathcal{L}$, if $\mathcal{L}$ is super-compact, then
    \begin{enumerate}
        \item for any $\mathbb{A}\in \mathsf{K}_{\mathcal{L}}$ and any $a\in \mathbb{A}$, the set $a{\uparrow} = \{b\in \mathbb{A}\mid a\leq^\mathbb{A} b\}$ is an $\mathcal{L}$-filter of $\mathbb{A}$.
        \item $\mathsf{Alg}(\mathcal{L}) = \mathsf{K}_{\mathcal{L}}$.
    \end{enumerate}
\end{proposition}
\begin{proof}
1. Let $a\in \mathbb{A}$, and $\Gamma\cup\{\varphi\}\subseteq \mathrm{Fm}$ s.t.~$\Gamma\vdash \varphi$; since $\mathcal{L}$ is super-compact,  $\gamma\vdash \varphi$ for some $\gamma\in \Gamma$; let $h\in Hom(\mathrm{Fm}, \mathbb{A})$ s.t.~$a\leq^\mathbb{A} h(\delta)$ for every $\delta\in \Gamma$. Hence, by  Proposition \ref{prop: interdiravable}.1, $a\leq^\mathbb{A} h(\gamma)\leq^\mathbb{A} h(\varphi)$, thus $h(\varphi)\in a{\uparrow}$, as required.

     2. Since $\mathsf{Alg}(\mathcal{L}) \subseteq  \mathsf{K}_{\mathcal{L}}$ holds for any deductive system (cf.~Section \ref{ssec:selfextensional}), we only need to argue for the converse inclusion, for which it is enough  to show that, for any $\mathbb{A}\in \mathsf{K}_{\mathcal{L}}$, the Frege relation $\equiv_{\lfilt{\mathbb{A}}}$ is the identity, i.e.~that any two different elements $a, b\in \mathbb{A}$ can be separated by an $\mathcal{L}$-filter of $\mathbb{A}$. If $a\neq b$, we can assume w.l.o.g.~that $a\nleq^\mathbb{A} b$, i.e.~$a\in a{\uparrow}$ and $b\notin a{\uparrow}$, which finishes the proof, since, by item 1, $a{\uparrow}$ is an $\mathcal{L}$-filter of $\mathbb{A}$.
\end{proof}
Thus, if $\mathcal{L}$ is a supercompact selfextensional logic with disjunction, then $V(\mathcal{L}) = \mathsf{K}_{\mathcal{L}} = \mathsf{Alg}(\mathcal{L})$.

\begin{corollary}
    If $\mathcal{L}$ is selfextensional, disjunctive, and supercompact, then $\mathcal{L}$ is fully selfextensional.
\end{corollary}
\begin{proof}
      Since $\mathcal{L}$ is selfextensional and disjunctive,  by Proposition \ref{prop:charact self and disj}, $\mathcal{L}$  is join-semilattice
 based. Hence,  by (the proof of) item 2 of Proposition \ref{prop:supercompact and logical filters}, the Frege relation $\equiv_{\lfilt{A}}$ of any $\mathbb{A}\in \mathsf{Alg}(\mathcal{L}) = \mathsf{K}_{\mathcal{L}}$ is the identity relation on $\mathbb{A}$, and hence
 $\equiv_{\lfilt{A}}$ is a congruence of $\mathbb{A}$. To finish the proof, we need to show that, if $\mathcal{A} = (\mathbb{A}, \mathcal{G})$ is a full g-model of $\mathcal{L}$, then $\equiv_{\mathcal{G}}$ is a congruence of $\mathbb{A}$.  By \cite[Proposition 2.40]{font2017general}, if $f: \mathcal{A}\to \mathcal{B}$ is a surjective strict homomorphism between g-matrices, then the Frege relation of $\mathcal{A}$ is a congruence iff the Frege relation of $\mathcal{B}$ is. If $\mathcal{A}$ as above is a full g-model, then its reduction  is of the
 form $\mathcal{B} =(\mathbb{B}, \lfilt{\mathbb{B}})$ for some $\mathbb{B}\in \mathsf{Alg}(\mathcal{L})$, and by what we have just proved, $\equiv_{\lfilt{\mathbb{B}}}$ is a congruence of $\mathbb{B}$. Hence, $\equiv_{\mathcal{G}}$ is a congruence of $\mathbb{A}$, as required.
\end{proof}


\section*{Acknowledgments}

\subsection*{Competing interests}
    The authors of this study declare that there is no conflict of interest with any commercial or financial entities related to this research.
\subsection*{Authors' contributions}
    Xiaolong Wang drafted the initial version of this article. Other authors have all made equivalent contributions to it.
\subsection*{Funding}
    The authors  affiliated to the Vrije Universiteit Amsterdam have received funding from the EU Horizon 2020 research and innovation programme under the MSCA-RISE grant agreement No. 101007627.
\\
Xiaolong Wang is supported by the China Scholarship Council No.202006220087.
\\
Krishna Manoorkar is supported by the NWO grant KIVI.2019.001 awarded to Alessandra Palmigiano.






\end{document}